\documentclass[11pt,reqno]{amsart}
\usepackage[dvipsnames,usenames]{color}
\usepackage{lipsum}
\usepackage{filecontents}
\usepackage[dvipsnames]{xcolor}
\usepackage{hyperref}
\usepackage{cleveref}

\newcommand\myshade{85}
\colorlet{mylinkcolor}{NavyBlue}
\colorlet{mycitecolor}{NavyBlue}
\colorlet{myurlcolor}{NavyBlue}

\hypersetup{
  linkcolor  = mylinkcolor!\myshade!black,
  citecolor  = mycitecolor!\myshade!black,
  urlcolor   = myurlcolor!\myshade!black,
  colorlinks = true,
}
\usepackage{xypic}
\usepackage{tabularx,booktabs}
\usepackage{amssymb,amsmath,amsthm}
\usepackage{graphicx,color}
\usepackage{hyperref}

\newtheorem{prop}{Proposition}[section]

\newtheorem{theorem}[prop]{Theorem}
\newtheorem{lemma}[prop]{Lemma}
\newtheorem{cor}[prop]{Corollary}

\theoremstyle{definition}
\newtheorem{definition}[prop]{Definition}
\newtheorem{ex}[prop]{Example}
\newtheorem{rmk}[prop]{Remark}

\newcommand{\ra}[1]{\renewcommand{\arraystretch}{#1}}
\usepackage{adjustbox}
\usepackage{caption}
\renewcommand{\phi}{\varphi}
\newcommand{\Z}{\mathbb{Z}}

\newcommand{\supp}{\operatorname{supp}}

\newcommand{\braket}[2]{\left\langle #1 , #2 \right\rangle} 

\newcommand{\abs}[1]{\lvert #1 \rvert}

\author[W. Ballinger]{William Ballinger}
\address {Department of Mathematics, California Institute of Technology, Pasadena, CA 91125}
\email{wballing@caltech.edu \\ yini@caltech.edu \\ tochse@caltech.edu \\ vafaee@caltech.edu}

\author[Y. Ni]{Yi Ni}

\author[T. Ochse]{Tynan Ochse}

\author[F. Vafaee]{Faramarz Vafaee}

\begin{document}
\title{The prism manifold realization problem II}

\date{}

\maketitle

\begin{abstract}
We continue our study of the realization problem for prism manifolds. Every prism manifold can be parametrized by a pair of relatively prime integers $p>1$ and $q$. We determine a complete list of prism manifolds $P(p, q)$ that can be realized by positive integral surgeries on knots in $S^3$ when $q>p$. The methodology undertaken to obtain the classification is similar to that of the case $q<0$ in an earlier paper.
\end{abstract}
\section{Introduction}\label{sec:Introduction}

This paper is a continuation of \cite{Prism2016}, where the authors studied the Dehn surgery realization problem of prism manifolds. Recall that prism manifolds are spherical three--manifolds with dihedral type fundamental groups.
Alternatively, an oriented prism manifold $P(p,q)$ has Seifert invariants 
\[
(-1; (2,1), (2,1), (p,q)),
\]
where $q$ and $p>1$ are relatively prime integers. A surgery diagram of $P(p,q)$ is depicted in Figure~\ref{links}A. 
When $q<0$, the realization problem for prism manifolds was solved in \cite{Prism2016}. More precisely, a complete list of $P(p,q)$, with $q < 0$, that can be obtained by positive Dehn surgery on knots in $S^3$ is tabulated in~\cite[Table~1]{Prism2016}. Indeed, every manifold in the table can be obtained by surgery on a {\it Berge--Kang knot}~\cite{BergeKang}.
 Our main result, Theorem~\ref{thm:Realization} below, provides the solution for those $P(p,q)$ with $q>p$: see Table~\ref{table:Types}.

\begin{theorem}\label{thm:Realization}
Given a pair of relatively prime integers $p>1$ and $q>p$, the prism manifold $P(p,q)$ can be obtained by $4q$--surgery on a knot $K\subset S^3$ if and only if $P(p,q)$ belongs to one of the six families in Table~\ref{table:Types}. Moreover, in this case, there exists a Berge--Kang knot $K_0$ such that $P(p,q)\cong S^3_{4q}(K_0)$, and that $K$ and $K_0$ have isomorphic knot Floer homology groups.
\end{theorem}


The methodology used to obtain Table~\ref{table:Types} is similar to that of~\cite{greene:LSRP,Prism2016}.  When $q>p$, the prism manifold $P(p,q)$ bounds a negative definite four--manifold $X=X(p,q)$ with a Kirby diagram as in Figure~\ref{links}D: see Section~\ref{sec:Preliminaries}. Let $P(p,q)$ arise from surgery on a knot $K\subset S^3$. Let also $W_{4q}=W_{4q}(K)$ be the corresponding two--handle cobordism obtained by attaching a two--handle to the four--ball along the knot $K$ with framing $4q$. Form the four--manifold $Z:=X\cup_{P(p,q)}(-W_{4q})$. It follows that $Z$ is a smooth, closed, negative definite four--manifold with $b_2(Z)=n+2$ for some $n\ge 1$: see Figure~\ref{links}D. Now, the celebrated theorem of Donaldson (``Theorem~A") implies that the intersection pairing on $H_2(Z)$ is isomorphic to $-\mathbb Z^{n+2}$~\cite{Donaldson1983}, the Euclidean integer lattice with the negation of its usual dot product. This provides a necessary condition for $P(p,q)$ to be positive integer surgery on a knot; namely, the lattice $C(p,q)$, specified by the negative of the intersection pairing on $H_2(X)$, must embed as a codimension one sublattice of $\mathbb Z^{n+2}$. The key idea we use to sharpen this into a necessary and sufficient condition is the work of Greene~\cite{greene:LSRP}, which is built mainly on the use of the {\it correction terms} in Heegaard Floer homology in tandem with Donaldson's theorem. In order to state his theorem, we first require a combinatorial definition.  
\begin{definition}\label{defn:changemaker}
A vector $\sigma=(\sigma_0,\sigma_1,\dots,\sigma_{n+1})\in\mathbb Z^{n+2}$ that satisfies $0\le\sigma_0\le\sigma_1\le\cdots\le\sigma_{n+1}$ is a {\it changemaker vector} if for every $k$, with $0\le k\le\sigma_0+\sigma_1+\cdots+\sigma_{n+1}$, there exists a subset $S\subset\{0,1,\dots,n+1\}$
such that $k=\sum_{i\in S}\sigma_i$.
\end{definition}

Using Lemma~\ref{XSharp}, the following is immediate from~\cite[Theorem~3.3]{Greene2015}.

\begin{theorem}\label{changemakerlatticeembedding}
Suppose $P(p, q)$ with $q>p$ arises from positive integer surgery on a knot in $S^3$. The lattice $C(p, q)$ is isomorphic to the orthogonal complement $(\sigma)^\perp$ of some changemaker vector $\sigma \in \Z^{n+2}$.
\end{theorem}

By determining the pairs $(p,q)$ which pass the embedding restriction of Theorem~\ref{changemakerlatticeembedding}, we get the list of all prism manifolds $P(p,q)$ with $q>p$ that can possibly be realized by integer surgery on a knot in $S^3$: again, see Table~\ref{table:Types}. We still need to verify that every manifold in our list is indeed realized by a knot surgery. In fact, this is the case.
\begin{theorem}\label{thm:lattice}
Given a pair of relatively prime integers $p>1$ and $q>p$,
$C(p,q)\cong(\sigma)^{\perp}$ for a changemaker vector $\sigma\in\mathbb Z^{n+2}$ if and only if $P(p,q)$ belongs to one of the six families in Table~\ref{table:Types}. Moreover, in this case, there exist a knot $K\subset S^3$ with $S^3_{4q}(K) \cong P(p,q)$ and an isomorphism of lattices
\[\phi: (\mathbb Z^{n+2},I)\to (H_2(Z),-Q_Z),\]
such that $\phi(\sigma)$ is a generator of $H_2(-W_{4q})$. Here $I$ denotes the standard inner product on $\mathbb Z^{n+2}$ and $Q_Z$ is the intersection form of $Z= X(p,q) \cup (-W_{4q})$.
\end{theorem}

{\rmk Theorem~\ref{thm:lattice}, in particular, highlights that the families in Table~\ref{table:Types} are divided so that each changemaker vector corresponds to a unique family. However, a prism manifold $P(p, q)$ may belong to more than one family in Table~\ref{table:Types}. 
We will address the overlaps between the families of Table~\ref{table:Types} in Section~\ref{realizable}: see Table~\ref{Overlap}.


Table~2 in \cite{Prism2016} gives a conjecturally complete list of prism manifolds $P(p,q)$ with $q>0$ that can be obtained by performing surgery on a knot in $S^3$. Every manifold in \cite[Table~2]{Prism2016} is obtained by integral surgery on a {Berge--Kang knot} (see~\cite[Table~4]{Prism2016} and~\cite{BergeKang}). Theorem~\ref{thm:Realization} proves~\cite[Conjecture~1.6]{Prism2016} for the case $q>p$ since the manifolds in Table~\ref{table:Types} coincide with those in~\cite[Table~2]{Prism2016} with $q>p$. We leave open the realization problem for prism manifolds $P(p,q)$ with $0<q<p$. We plan to address this case in a future paper.

\subsection{Organization} Section~\ref{sec:Preliminaries} collects the topological background on prism manifolds, and also reviews the essentials needed to prove our main results. In Section~\ref{sec:CLattices}, we study C--type lattices $C(p,q)$ that are central in the present work. To prove Theorem~\ref{thm:lattice}, we begin with a study of the changemaker lattices (Section~\ref{sec:ChangemakerLattices}), i.e. lattices of the form $(\sigma)^{\perp}\subset \mathbb Z^{n+2}$ for some changemaker vector $\sigma\in \mathbb Z^{n+2}$. We then study when a changemaker lattice, with a {\it standard basis}, is isomorphic to a C--type lattice, with its distinguished {\it vertex basis}. The key to answering this combinatorial question is detecting the {\it irreducible elements} in either of the lattices. Indeed, the standard basis elements of a changemaker lattice are irreducible (Lemma~\ref{lem:irred}), as are the vertex basis elements of a C--type lattice. Furthermore, the classification of the irreducible elements of C--type lattices is given in Proposition~\ref{prop:IntervalsIrreducible}. We collect many structural results about these lattices in Sections~\ref{sec:CLattices}~and~\ref{sec:ChangemakerLattices}.

We classify the changemaker C--type lattices based on how $x_0$, the first element in the ordered basis of a C--type lattice, is written in terms of the standard orthonormal basis elements of $\mathbb Z^{n+2}$. Accordingly, Sections~\ref{sec:Case2},~\ref{sec:k1k21},~and~\ref{sec:Case1} will enumerate the possible changemaker vectors whose orthogonal complements are C--type lattices. Section~\ref{pq} tabulates the corresponding prism manifolds. 

Finally, in Section~\ref{realizable}, we address the overlaps between the families in Table~\ref{table:Types}. More precisely, we provide distinct knots corresponding to distinct changemakers that result in the same prism manifold. We then proceed with proving Theorems~\ref{thm:Realization} and~\ref{thm:lattice}.

\subsection*{Acknowledgements} The heart of this project was done during the Caltech's Summer Undergraduate Research Fellowships (SURF) program in the summer of 2017. 
Y.~N. was partially supported by NSF grant number DMS-1252992
and an Alfred P. Sloan Research Fellowship. F.~V. was partially supported by an AMS-Simons Travel Grant.
W.~B. would like to thank William H. and Helen Lang, as well as Samuel P. and Frances Krown, for their generous support through the SURF program. T.~O. would like to thank Joanna Wall Muir and Mr. James Muir for their generous support through the SURF program.

\section{Preliminaries}\label{sec:Preliminaries}

For a pair of relatively prime integers $p>1$ and $q$, the prism manifold $P(p,q)$ is a Seifert fibered space with a surgery description depicted in Figure~\ref{links}A. It is shown in~\cite{Prism2016} that if $P(p,q)$ is obtained by surgery on a knot in $S^3$, $p$ must be odd. 

An equivalent surgery description for $P(p,q)$ is depicted in Figure~\ref{links}D. To get the coefficients $a_i$, write $\frac{2q-p}{q-p}$ in a Hirzebruch--Jung continued fraction
\begin{equation}\label{eq:ContFrac}
\displaystyle \frac{2q-p}{q-p} = a_1 - \frac{1}{a_2 - \displaystyle\frac{1}{\ddots - \displaystyle\frac{1}{a_n}}} = [a_1,a_2,\dots,a_n]^-.
\end{equation}
From this point on in the paper, we assume that $q>p$. As a result, we have $a_1\ge 3$ in Equation~\eqref{eq:ContFrac}. Moreover, each $a_i\ge2$.

\begin{definition}\label{def:CType}
The \emph{C-type lattice} $C(p,q)$ has a basis 
\begin{equation}\label{VertexBasis}
\{x_0,\dots,x_n\},
\end{equation}
and inner product given by
\begin{equation*}
\braket{x_i}{x_j} = \begin{cases}
4 & i = j = 0 \\
a_i & i = j > 0 \\
-2 & \{i,j\} =\{ 0,1\} \\
-1 & |i-j| = 1, i > 0, j > 0\\
0 &|i-j|>1,
\end{cases}
\end{equation*}
where the coefficients $a_i$, for $i\in \{1,\cdots,n\}$, are defined by the continued fraction~\eqref{eq:ContFrac}. We call~\eqref{VertexBasis} the \emph{vertex basis} of $C(p,q)$.
\end{definition}

\begin{figure}
\centering
\def\svgwidth{\textwidth}
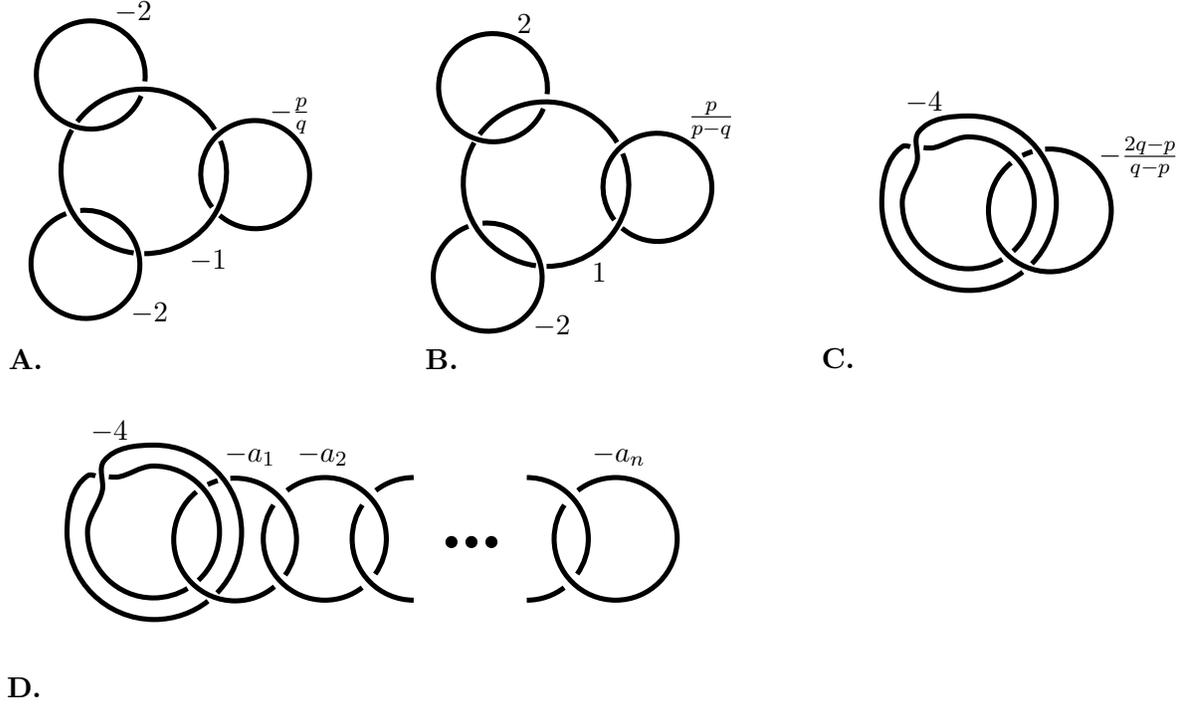
\caption{Surgery presentations of $P(p,q)$. A and B correspond to the two equivalent choices of Seifert invariants $(-1;(2,1),(2,1),(p,q))$ and $(1; (2,1),(2,-1),(p,q-p))$. To go from B to C, blow down two $1$-framed unknots in sequence: first blow down the middle unknot, changing the framing on the upper left unknot to $1$, and then blow down the upper left unknot. Finally, to get to D, use slam-dunk moves to expand $\frac{2q-p}{q-p}$ in a continued fraction. The last link gives a negative-definite four--manifold if $q < 0$ or $q > p$.}
\label{links}
\end{figure}

\begin{figure}
\def\svgwidth{\textwidth}
\includegraphics[scale=.7]{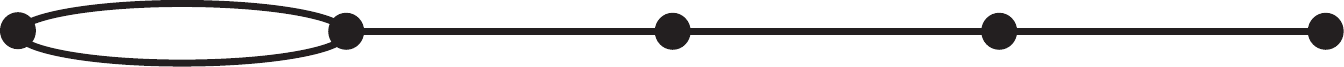}
\put(-272,-7.5){$4$}
\put(-207,-7.5){$a_1$}
\put(-142,-7.5){$a_2$}
\put(-75,-7.5){$\cdots$}
\put(-9,-7.5){$a_n$}
\caption{A C--type lattice $C(p,q)$ with $\frac{2q-p}{q-p} = [a_1,a_2,\cdots,a_n]^-$. Note that $a_1\ge 3$ when $q>p$.}
\label{CType}
\end{figure}
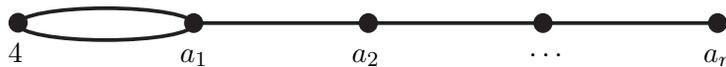
Let $X=X(p,q)$ be the four--manifold, bounded by $P(p,q)$, with a Kirby diagram as depicted in Figure~\ref{links}D. The inner product space $(H_2(X), -Q_X)$ equals $C(p,q)$, where $Q_X$ denotes the intersection pairing of $X$: see Figure~\ref{CType}. Note that $b_2(X)=n+1$, where $n$ is defined in~\eqref{eq:ContFrac}.
{\rmk When $q<0$ in Equation~\eqref{eq:ContFrac}, it follows that $a_1=2$ and $C(p,q)$ is indeed isomorphic to a D--type lattice~\cite[Definition~2.8]{Prism2016}. The prism manifold realization problem is solved in this case~\cite{Prism2016}.
}

\begin{figure}

\centering
\def\svgwidth{.8\textwidth}
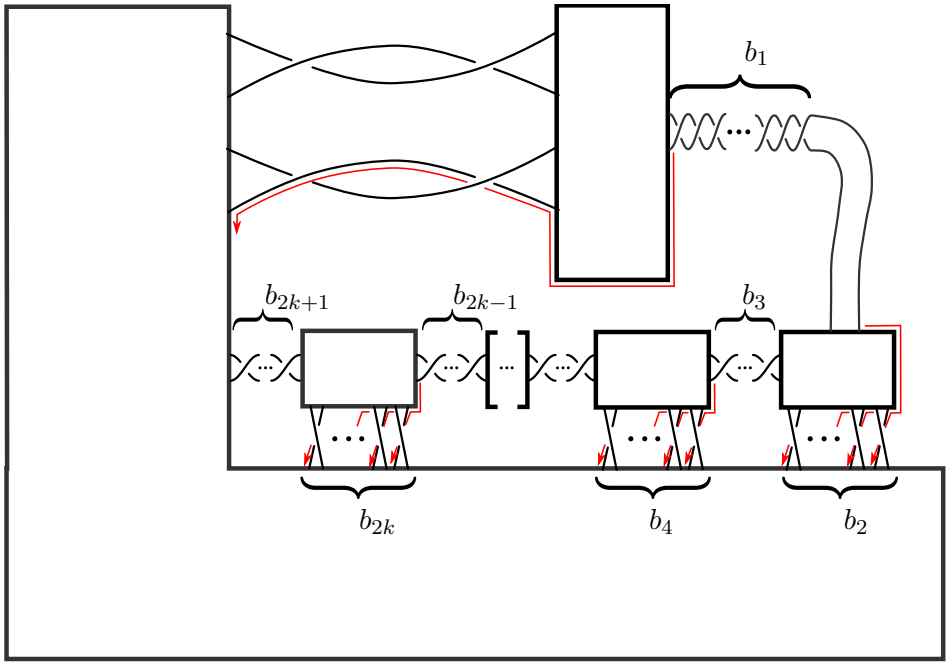
\caption{A handle decomposition of a surface embedded in $S^3$. The boundary of this surface is an alternating Montesinos link whose branched double cover is $P(p,q)$, and the branched double cover of $B^4$ over this surface with its interior pushed into the interior of $B^4$ is $X(p,q)$. Sliding the 1--handles in this picture along the red arrows and then cancelling all but one of the 0--handles gives Figure~\ref{SlidMontesinos}. This surface depends on parameters $b_1,\dots,b_m$ where $m$ is either $2k+1$ or $2k$; if $m = 2k$ omit the band labelled $b_{2k+1}$.}\label{Montesinos}

\end{figure}
\subsection{The four--manifold $X(p,q)$ revisited}

In this subsection, we present a different construction of the four--manifold $X(p,q)$ as the branched double cover of $B^4$ over a particular surface: see Figure~\ref{Montesinos}. As a Seifert fibered rational homology sphere, the prism manifold $P(p,q)$ is the branched double cover of $S^3$ branched along a Montesinos link~\cite{Montesinos1973}: choose $b_1,\dots,b_n$ so that
\begin{equation}
\frac{p}{q-p} = b_1 + \frac{1}{b_2 + \displaystyle \frac{1}{ \ddots + \displaystyle \frac{1}{b_m}}} = [b_1,b_2,\dots,b_m]^+.
\end{equation}
Since $q > p$, $\frac{p}{q-p} > 0$ and we can choose the $b_i$ so that $b_1 \ge 0$ and $b_i > 0$ for $i > 1$. The boundary of the surface $\Sigma$ drawn in Figure~\ref{Montesinos} is an alternating Montesinos link $L$, and $\Sigma$ itself is the surface formed by the black regions in a checkerboard coloring of the alternating diagram. We point out that we are using the coloring convention as in Figure~\ref{convention}. The branched double cover of $S^3$ branched along $L$ is $P(p,q)$. Let $X_{\Sigma}$ be the branched double cover of $B^4$ over the surface $\Sigma$ with its interior pushed into the interior of $B^4$. With this notation in place:

\begin{figure}[t]
\includegraphics[scale=.6]{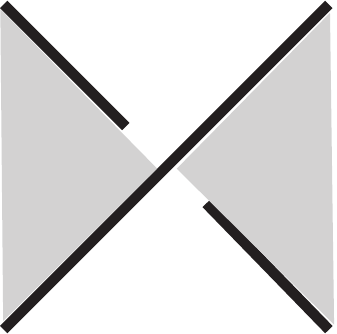}
\caption{The coloring convention}
\label{convention}
\end{figure}

\begin{figure}

\centering
\def\svgwidth{.8\textwidth}
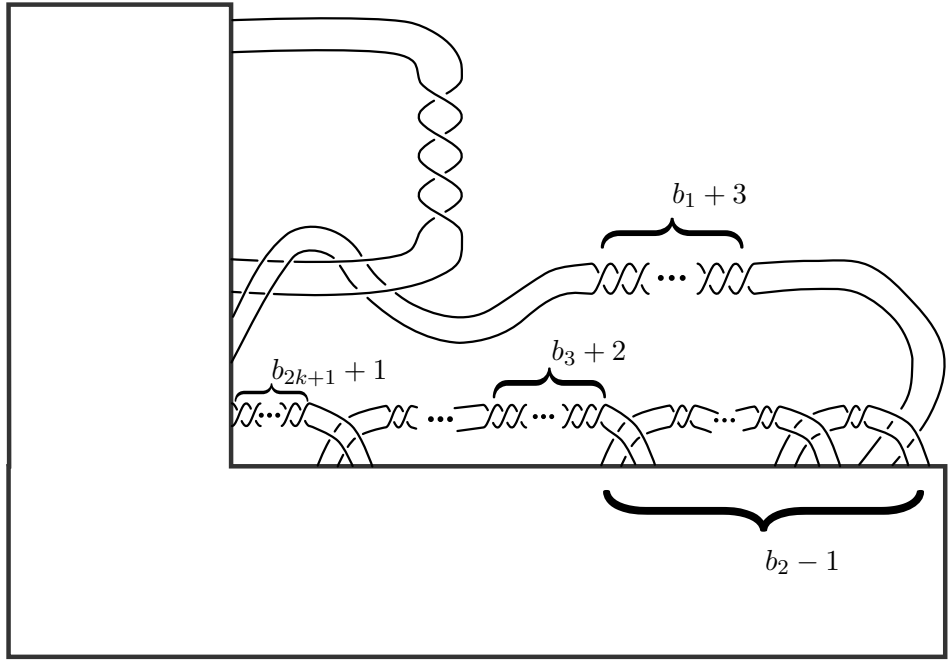
\caption{Another view of the surface shown in Figure~\ref{Montesinos}. From this picture a Kirby diagram representing the branched double cover of $B^4$ over this surface (shown in Figure~\ref{CtypeFromMontesinos}) can be read off using the methods of Figure 4 in \cite{AkbulutKirby1980}. As before, if $m$ is even omit the band labelled $b_{2k+1}$.}\label{SlidMontesinos}

\end{figure}

\begin{prop}\label{BDC}
$X(p,q)\cong X_\Sigma$.
\end{prop}

We first recall the following lemma that will be used in the proof of Proposition~\ref{BDC} and also in Section~\ref{pq}. 
\begin{lemma}[Lemma~9.5~(1)~and~(3) of~\cite{greene:LSRP}]\label{greene9.5}
For integers $r,s,t \ge 0$,
\begin{itemize}
    \item[1.] $[\dots,r,2^{[s]},t,\dots]^- = [\dots, r-1, -(s+1), t-1,\dots]^-$, and
    \item[2.] $[\dots, s, 2^{[t]}]^- = [\dots, s-1, -(t+1)]^-$,
\end{itemize}
where $2^{[a]}$ means that the entry $2$ appears $a$ times.
\end{lemma}

We now proceed to prove Proposition~\ref{BDC}. In order to obtain a Kirby diagram of branched double covers, we closely follow the treatment of~\cite{AkbulutKirby1980}; in particular, see \cite[Figure~4]{AkbulutKirby1980}.

\begin{proof}[Proof of Proposition~\ref{BDC}]
Figure~\ref{Montesinos} depicts a handle decomposition of the surface $\Sigma$ whose branched double cover is $X_\Sigma$. By sliding the 1--handles along the red arrows in Figure~\ref{Montesinos} and then canceling all but only one of the 0--handles, we obtain the surface in Figure~\ref{SlidMontesinos}: a disc with several bands attached. The odd-numbered $b_{2i+1}$ with $0<i<\frac{m-1}2$ contribute bands with $b_{2i+1} + 2$ half-twists, $b_1$ contributes a band with $b_1 + 3$ half-twists, and $b_m$ contributes a band with $b_m + 1$ half-twists when $m$ is odd. The even-numbered $b_{2i}$ contribute $b_{2i} - 1$ bands each, each with $2$ half-twists. Therefore, the coefficients $a_1,\dots, a_n$ of Figure~\ref{CtypeFromMontesinos} are
\begin{equation}\label{equivalent}
(a_1,\dots,a_n) = \begin{cases}
(b_1 + 3, 2^{[b_2 - 1]}, b_3 + 2, 2^{[b_4 - 1]},\dots,2^{[b_{m-1} - 1]},b_m + 1) & m\text{ odd}, \\
(b_1 + 3, 2^{[b_2 - 1]}, b_3 + 2, 2^{[b_4 - 1]},\dots,b_{m-1} + 2,2^{[b_{m} - 1]}) & m\text{ even}.
\end{cases}
\end{equation}
 Using Lemma~\ref{greene9.5}, 
\begin{eqnarray*}
[a_1,\dots,a_n]^-& =& [b_1 + 2, -b_2, b_3, -b_4, \dots, \pm b_m]^-\\ &=& [b_1 + 2, b_2,\dots,b_m]^+ \\ & =& \frac{p}{q-p} + 2 \\ & =& \frac{2q-p}{q-p}. 
\end{eqnarray*} 
That is, the $a_i$ in Equation~\eqref{equivalent} are the same as those of Equation~\eqref{eq:ContFrac}. The branched double cover of $B^4$ branched over the surface in Figure~\ref{SlidMontesinos} is depicted in Figure~\ref{CtypeFromMontesinos}; comparing it with Figure~\ref{links}D, the result follows.
\end{proof}

\subsection{Input from Heegaard Floer homology}

We assume familiarity with Floer homology and only review the essential input here for completeness. See, for instance,~\cite{Ozsvath2004, Ozsvath2004a}. In~\cite{OSzAbGr}, Ozsv\'ath and Szab\'o defined the correction term $d(Y, \mathfrak t)$ that associates a rational number to an oriented rational homology sphere $Y$ equipped with a Spin$^c$ structure $\mathfrak t$. 
If $Y$ is boundary of a negative definite four--manifold $X$, then
\begin{equation}\label{eq:CorrBound}
 c_1(\mathfrak s)^2 + b_2(X)\le 4d(Y, \mathfrak t),
\end{equation}
for any $\mathfrak s \in \text{Spin}^c(X)$ that extends $\mathfrak t \in \text{Spin}^c(Y)$.

\begin{definition}
A smooth, compact, negative definite four--manifold $X$ is {\it sharp} if for every $\mathfrak t \in \text{Spin}^c(Y)$, there exists some $\mathfrak s\in \text{Spin}^c(X)$ extending $\mathfrak t$ such that the equality is realized in Equation (\ref{eq:CorrBound}).
\end{definition}

\begin{figure}

\centering
\def\svgwidth{.8\textwidth}
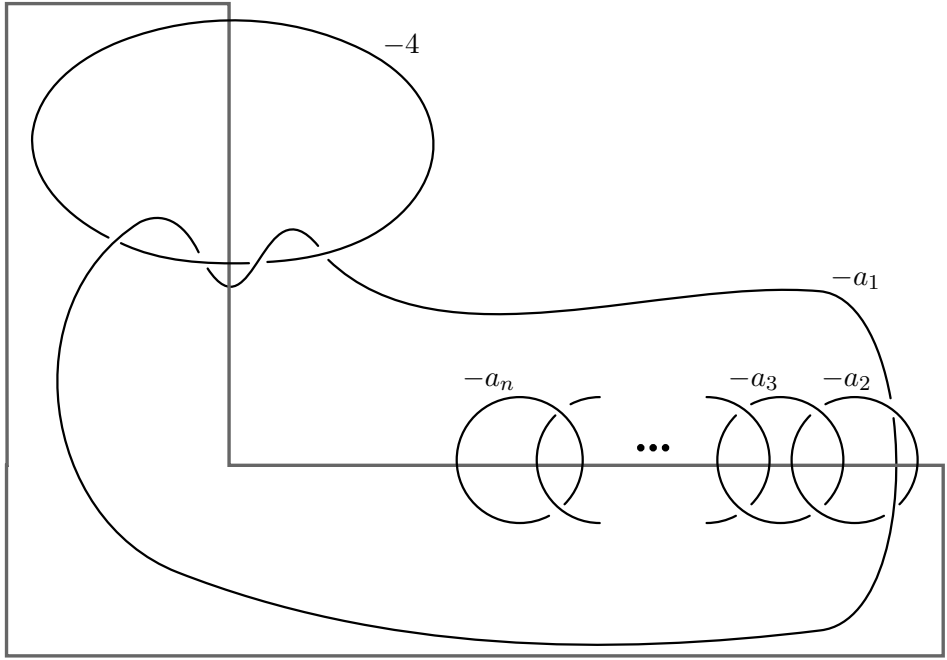
\caption{A Kirby diagram representing the branched double cover of the surface in Figure~\ref{Montesinos}. This is the same as the diagram defining $X(p,q)$. The grey box is not part of the link, but is included only to show the relationship with Figure~\ref{SlidMontesinos}.}\label{CtypeFromMontesinos}

\end{figure}

Using Proposition~\ref{BDC}, the following is immediate from~\cite[Theorem~3.4]{OSzBrDoub}.

{\lemma \label{XSharp} $X(p,q)$ is a sharp four--manifold.
}

\subsection{Alexander polynomials of knots on which surgery yield $P(p,q)$ with $q>p$}
Using techniques that will be developed in the next sections in tandem with Theorem~\ref{changemakerlatticeembedding}, we will find the classification of all C-type lattices $C(p,q)$ that are isomorphic to $(\sigma)^{\perp}$ for some changemaker vector $\sigma$ in $\mathbb Z^{n+2}$. If the corresponding prism manifold $P(p,q)$ is indeed arising from surgery on a knot $K \subset S^3$, we are able to compute the Alexander polynomial of $K$ from the values of the components of $\sigma$: let $S$ be the closed surface obtained by capping off a Seifert surface for $K$ in $W_{4q}$. It is straightforward to check that the class $[S]$ generates $H_2(W_{4q})$. It follows from Theorem~\ref{changemakerlatticeembedding} that, under the embedding $H_2(X)\oplus H_2(-W_{4q})\hookrightarrow H_2(Z)$, the homology class $[S]$ gets mapped to a changemaker vector $\sigma$. Let $\{e_0, e_1, \cdots, e_{n+1}\}$ be the standard orthonormal basis  for $\mathbb Z^{n+2}$, and write
\[
\sigma = \sum_{i=0}^{n+1} \sigma_i e_i.
\] 
Also, define the \textit{characteristic covectors} of $\mathbb{Z}^{n+2}$ to be 
\[
\text{Char}(\mathbb Z^{n+2})=\left \{ \left.\sum_{i=0}^{n+1}\mathfrak c_i e_i \right| \mathfrak c_i\text{ odd for all } i\right \}.
\]
We remind the reader that, writing the Alexander polynomial of $K$ as 
\begin{equation}\label{eq:AlexanderPolynomial}
\Delta_K(T)= b_0 + \sum_{i>0}b_i(T^i+T^{-i}),
\end{equation}
the $k$-th {\it torsion coefficient} of $K$ is
\[
t_k(K)= \sum_{j\ge 1}jb_{k+j},
\]
where $k\ge 0$. The following lemma is immediate from~\cite[Lemma~2.5]{Greene2015}.
\begin{lemma}\label{lem:AlexanderComputation}
The torsion coefficients satisfy
\[
t_i(K)=
\left\{
\begin{array}{cl}
\displaystyle\min_{\mathfrak c}\frac{\mathfrak c^2-n-2}8, &\text{for each $i\in\{0,1,\dots,2q\}$,}\\
&\\
0,&\text{for $i>2q$.}
\end{array}
\right.
\]
where $\mathfrak c$ is subject to 
\[
\mathfrak c\in\mathrm{Char}(\mathbb Z^{n+2}),\quad\langle\mathfrak c,\sigma\rangle+4q\equiv2i\pmod{8q}.
\]
And for $i>0$, 
\[
b_i=t_{i-1}-2t_i+t_{i+1},\quad\text{for }i>0,
\]
and
\[b_0=1-2\sum_{i>0}b_i,\]
where the $b_i$ are as in~\eqref{eq:AlexanderPolynomial}.
\end{lemma}

\section{C-Type Lattices}\label{sec:CLattices}

This section assembles facts about C-type lattices that will be used in the classification. We mainly use the notation of~\cite{greene:LSRP, Prism2016}. Recall that we always assume $q > p$, so $a_1 \ge 3$: see Figure~\ref{CType}.

Let $L$ be a lattice. Given $v \in L$, let $|v| = \langle v, v \rangle$ be the norm of $v$. An element $\ell \in L$ is {\it reducible} if $\ell = x+y$ for some nonzero $x, y \in L$, with $\langle x, y \rangle \ge 0$, and {\it irreducible} otherwise. An element $\ell \in L$ is {\it breakable} if $\ell = x+y$ with $|x|, |y| \ge 3$ and $\langle x, y \rangle =-1$, and {\it unbreakable} otherwise.



Among the irreducible elements of a lattice, intervals are the most convenient for us:
\begin{definition}
In a C-type lattice, if $I$ is any subset of $\{x_0,x_1,\dots,x_n\}$ then write $[I] = \sum_{x \in A} x$. An {\it interval} is an element of the form $[I]$ with $I = \{x_a,x_{a+1},\dots,x_b\}$ for $0 \le a \le b \le n$. We say that $a$ is the left endpoint of the interval, and $b$ is the right endpoint of the interval. Say that $[I]$ contains $x_i$ if $I$ does. 
\end{definition}


Given the fact that $a_1 \ge 3$, the following is immediate from~\cite[Proposition~3.3]{greene:LSRP}.

\begin{prop}\label{prop:IntervalsIrreducible}
If $v \in C(p,q)$ is irreducible, $v = \epsilon[I]$ for some $\epsilon = \pm 1$ and $[I]$ an interval.
\end{prop}

\begin{definition}\label{defn:pairinggraph}
Given a lattice $L$ and a subset $V\subset L$, the {\it pairing graph} is $\hat{G}(V) = (V, E)$, where $e = (v_i, v_j) \in E$ if $\langle v_i, v_j \rangle \neq 0$.
\end{definition}

\begin{cor}\label{indecomposable}
The lattice $C(p,q)$ is indecomposable; that is, $C(p,q)$ is not the direct sum of two nontrivial lattices.
\end{cor}
\begin{proof}
Suppose that $C(p,q) \cong L_1 \oplus L_2$. Then each $x_i$, being irreducible, must be in either $L_1$ or $L_2$. However, any element of $L_1$ has zero pairing with any element of $L_2$. Since $\braket{x_i}{x_{i+1}}\ne0$, $\hat{G}(\{x_0, \dots, x_n\})$ is connected. This means that all of the $x_i$ are in the same part of the decomposition, and the other is trivial.
\end{proof}

In a C-type lattice, we have that $|\langle x_0, x_1\rangle |=2$. It turns out that the inner product of $x_0$ with any other element in the $C$-type lattice lives in $2\mathbb Z$. The following lemma is straightforward to prove.

\begin{lemma}\label{Lem:X0Odd}
For any $v \in C(p,q)$, $\braket{x_0}{v}$ is even. In particular, the reflection $r_{x_0}: v \mapsto v - 2\frac{\braket{x_0}{v}}{\braket{x_0}{x_0}}x_0$ about $x_0^{\perp}$ is an involution of $C(p,q)$.
\end{lemma}

\begin{definition}\label{Def:HighNorm}
A vertex $x_i$ has {\it high weight} if $i > 0$ and $|x_i| = a_i > 2$.
\end{definition}

\begin{prop}
An element $\epsilon[I] \in C(p,q)$ with $\epsilon \in \{\pm 1 \}$ is unbreakable if and only if $[I]$ contains at most one element of high weight.
\end{prop}
\begin{proof}
The conclusion is obvious when $I=\{x_0\}$. Now we assume $I\ne\{x_0\}$.
If $[I]$ does not contain $x_0$, this reduces to the analogous fact about linear lattices~\cite[Corollary~3.5 (4)]{greene:LSRP}. The reflection $r_{x_0}$ exchanges intervals with left endpoint $0$ and intervals with left endpoint $1$, which reduces the case of intervals containing $x_0$ to the case of intervals not containing $x_0$.
\end{proof}

\begin{definition}
Consider the graph $C$ on vertex set $\{x_0,\dots,x_n\}$ that has two edges between $x_0$ and $x_1$ and one edge between $x_i$ and $x_{i+1}$ for $0 < i < n$. Given two intervals $[I]$ and $[J]$, say that an edge of $C$ is {\it dangling} if one of its ends is in $I$, the other is in $J$, and at least one of the ends is not in $I \cap J$. Write $\delta([I],[J])$ for the number of dangling edges.
\end{definition}

\begin{lemma}\label{intervalproduct}
For two intervals $[I], [J]$, $\braket{[I]}{[J]} = |[I \cap J]| - \delta([I],[J])$. 
\end{lemma}
\begin{proof}
Suppose $I = \{x_a,\dots,x_b\}$ and $J = \{x_c,\dots,x_d\}$. Then we can express
\begin{equation*}
\braket{[I]}{[J]} = \sum_{i = a}^b \sum_{j = c}^d \braket{x_i}{x_j}
\end{equation*}
Terms in this sum with $|i - j| > 1$ vanish. The remaining terms either have $x_i$ and $x_j$ in $I \cap J$, so occur as terms in the expansion of $|[I \cap J]|$, or have at least one of $x_i$ or $x_j$ not in $I \cap J$, so contribute to $\delta([I],[J])$. 
\end{proof}

We frequently use the following lemma, which is stated without proof.

\begin{lemma}\label{Lem:IntervalNorm}
Let $I\ne\{x_0\}$ be an interval.
Then 
\[|[I]|=2+\sum_{x_i\in I\setminus\{x_0\}}(|x_i|-2).\]
\end{lemma}
Given the structure of a C-type lattice, the following is immediate.

\begin{lemma}\label{lem:delta}
For any intervals $I,J$, $\delta([I],[J])$ is $0,1,2,$ or $3$. If $\delta([I],[J]) = 3$, then $\braket{x_0}{[I]}=-\braket{x_0}{[J]}=\pm 2$. 
\end{lemma}

To more precisely describe the value $\delta([I],[J])$, it will be convenient to use some terminology from \cite{greene:LSRP}:

\begin{definition}
For two intervals $[I]$ and $[J]$ with left endpoints $i_0,j_0$ and right endpoints $i_1,j_1$, say that $[I]$ and $[J]$ are {\it distant} if either $i_1 + 1 < j_0$ or $j_1 + 1 < i_0$, that $[I]$ and $[J]$ {\it share a common end} if $i_0 = j_0$ or $i_1 = j_1$, and that $[I]$ and $[J]$ are {\it consecutive} if $i_1 + 1 = j_0$ or $j_1 + 1 = i_0$. Write $[I] \prec [J]$ if $I \subset J$ and $[I]$ and $[J]$ share a common end, and $[I] \dagger [J]$ if they are consecutive. If $[I]$ and $[J]$ are either consecutive or share a common end, say that they {\it abut}. If $I \cap J$ is nonempty and $[I]$ and $[J]$ do not share a common end, write $[I] \pitchfork [J]$. 
\end{definition}

\begin{rmk}
If $\braket{[I]}{x_0} = \braket{[J]}{x_0}$ or if either $\braket{[I]}{x_0}$ or $\braket{[J]}{x_0}$ is zero, then $\delta([I],[J])$ is $0$ if $[I]$ and $[J]$ are distant, $1$ if $[I]$ and $[J]$ abut, and $2$ if $[I] \pitchfork [J]$. If $\braket{[I]}{x_0} \neq \braket{[J]}{x_0}$ and both are nonzero, $\delta([I],[J])$ is $2$ if $[I]$ and $[J]$ abut, and $3$ if $[I] \pitchfork [J]$. In the latter case, $[I]$ and $[J]$ are never distant. 
\end{rmk}

We will also need to know which irreducible elements of $C(p,q)$ are breakable. In light of Proposition~\ref{prop:IntervalsIrreducible}, we only need to study that for intervals.
\begin{lemma}[Lemma~3.10 of \cite{Prism2016}]\label{lem:TwoIndBr}
    An interval $[A]$ is breakable if there are at least two high weight vertices.
\end{lemma}

\begin{definition}\label{Def:Zj} 
For an unbreakable interval $[I_j] \in C(p,q)$ with $|[I_j]|\ge 3$, let $x_{z_j}$ be the unique element with $|x_{z_j}|\ge 3$.
\end{definition}
We end this section by determining when two C-type lattices are isomorphic.

\begin{prop}\label{pp}
If $C(p, q) \cong C(p', q')$, then $p = p'$ and $q = q'$. 
\end{prop}
\begin{proof}
If $L$ is a lattice isomorphic to $C(p,q)$, then to recover $p$ and $q$ from $L$ it suffices to recover the ordered sequence of norms $(|x_1|,|x_2|,\cdots,|x_n|)$. To do this, we will first identify the elements of this sequence that are at least $3$, and then fill in the $2$'s. 

We claim that unless $(p,q)=(2,3)$, there is a unique (up to sign) unbreakable irreducible element $y$ such that $|y|=4$ and $\braket{y}{v}$ is even for all $v$ in $L$, and $y=\pm x_0$. Let $I\ne\{x_0\}$ be any interval representing an unbreakable irreducible element with norm 4. Suppose $I=\{x_a,x_{a+1},\dots,x_b\}$.
If $a>1$, then $\braket{[I]}{x_{a-1}}=-1$ is odd. If $b<n$, then $\braket{[I]}{x_{b+1}}=-1$ is odd. So we assume $a=0$ or $1$, and $b=n$. If $I$ contains at least two high weight vertices, then $I$ is breakable. So $x_1$ is the only high weight vertex, and $4=|[I]|=|x_1|$. If $n>1$, then $\braket{[I]}{x_{b}}=1$ is odd. So $n=1$, $|x_1|=4$. From (\ref{eq:ContFrac}) we get $(p,q)=(2,3)$.

From now on, we assume $(p,q)\ne(2,3)$.
 Let $R$ be the sublattice of $L$ generated by $x_0$ and all vectors of norm $2$. 
Since $L$ contains no vectors of norm $1$, any vector of norm $2$ in $L$ is irreducible. By Lemma~\ref{Lem:IntervalNorm}, then, $R$ is generated by $x_0$ and the $x_i$ with $|x_i| = 2$.

Now, let $V_0$ be the set of irreducible, unbreakable elements of $L\setminus\{\pm x_0\}$ with norm at least $3$, and let $V$ be the quotient of $V_0$ by the relation $v \sim u$ whenever either $v-u \in R$ or $v+u \in R$. Every element of $V_0$ corresponds to an interval containing a unique high-weight vertex, and $v \sim u$ if and only if these high-weight vertices are the same. Therefore, $V$ consists of precisely the equivalence classes of the $x_i$ with $|x_i| \ge 3$, $i>0$, and if $v \in V_0$ with $v \sim x_i$ we have $|v| = |x_i|$.

Finally, let $W$ be the set of indecomposible components of $R$, so each element of $W$ corresponds to either $x_0$ or a run of $2$'s in the sequence of norms $(|x_1|,|x_2|,\dots,|x_n|)$. Let $\mathcal{B}$ be the bipartite graph with  vertex set $V \cup W$, and an edge between $v \in V$ and $w \in W$ if there is a representative $\tilde{v} \in L$ of $v$ and an element $\tilde{w} \in W$ such that $\braket{\tilde{v}}{\tilde{w}} = -1$, or $w$ corresponds to $x_0$ and $\braket{\tilde{v}}{x_0} = -2$. Then $v$ and $w$ neighbor in $\mathcal B$ if and only if the element $x_i$ representing $v$ is adjacent to $x_0$ or the run of $2$'s corresponding to $w$, so $\mathcal{B}$ is in fact a path. Furthermore, there is a unique element $w_0 \in W$ that contains $x_0$, and $w_0$ must be one of the ends of the path $\mathcal{B}$. We can now recover $(|x_1|,|x_2|,\dots,|x_n|)$ as follows:  The vertex $w_0$ neighbors a unique element $v \in V$ in $\mathcal{B}$. The rest of the sequence is completed in the following way - as we travel down the path $\mathcal{B}$, when we encounter an element $w \in W$ we add $\text{rk } w$-many $2$'s to the sequence, and when we encounter an element $v \in V$ we add $|\tilde{v}|$ to the sequence for $\tilde{v}$ a representative of $v$. 
\end{proof}

\section{Changemaker Lattices}\label{sec:ChangemakerLattices}


A lattice is called a {\it changemaker lattice} if it is isomorphic to the orthogonal complement of a changemaker vector. Whenever $P(p, q)$, with $q>p$, comes from positive integer surgery on a knot, $C(p,q)$ is isomorphic to a changemaker lattice $(\sigma)^\perp\subset \Z^{n+2}$. 
In this section, we will assemble some basic structural results about C-type lattices that are isomorphic to changemaker lattices.

Write $(e_0,e_1, \dots, e_{n+1})$ for the orthonormal basis of $\Z^{n+2}$, and write $\sigma = \sum_i \sigma_i e_i$. 
Since $C(p,q)$ is indecomposable (Corollary~\ref{indecomposable}), $\sigma_0\ne0$, otherwise $(\sigma)^\perp$ would have a direct summand $\mathbb Z$. So $\sigma_0=1$.

We will need several results from \cite[Section~3]{greene:LSRP} about changermaker lattices:

\begin{definition}\label{stbasis}
The {\it standard basis} of $(\sigma)^\perp$ is the collection $S = \{v_1, \dots, v_{n}\}$, where
\begin{equation*}
    v_j = \left(2e_0 + \sum_{i = 1}^{j - 1} e_i\right) - e_j
\end{equation*}
whenever $\sigma_j = 1 + \sigma_0 + \cdots + \sigma_{j-1}$, and
\begin{equation*}
    v_j = \left(\sum_{i \in A} e_i\right) - e_j
\end{equation*}
whenever $\sigma_j = \sum_{i \in A} \sigma_i$, with $A \subset \{0, \dots, j-1\}$ chosen to maximize the quantity $\sum_{i \in A} 2^i$.  A vector $v_j \in S$ is called \emph{tight} in the first case, \emph{just right} in the second case as long as $i < j-1$ and $i \in A$ implies that $i+1\in A$, and \emph{gappy} if there is some index $i$ with $i \in A$, $i < j-1$, and $i+1 \not \in A$. Such an index, $i$, is a \emph{gappy~index} for $v_j$.
\end{definition}

The standard basis $S$ is in fact a basis of $C(p,q)$.

\begin{definition}
For $v \in \Z^{n+2}$, $\supp v = \{i | \braket{e_i}{v} \neq 0\}$ and $\supp^+ v = \{i | \braket{e_i}{v} > 0\}$.
\end{definition}

\begin{lemma}[Lemma~3.12~(3) in \cite{greene:LSRP}] \label{gappy3}
If $|v_{k+1}|=2$, then $k$ is not a gappy index for any $v_j$ with $j \in \{1, \cdots, n+1 \}$.
\end{lemma}

\begin{lemma}[Lemma~3.13 in \cite{greene:LSRP}] \label{lem:irred}
Each $v_j \in S$ is irreducible.
\end{lemma}

\begin{lemma}[Lemma~3.15 in \cite{greene:LSRP}]\label{lem:BrIsTight}
If $v_j \in S$ is breakable, then it is tight.
\end{lemma}

\begin{lemma}[Lemma~3.14~(2)~(3) in \cite{greene:LSRP}]\label{lem:SumIrr}
Suppose that $v_t \in S$ is tight.
\newline(1) If $v_j = e_t + e_{j-1} - e_j$, $j > t$, then $v_t + v_j$ is irreducible.
\newline(2) If $v_{t+1}=e_0+e_1+\cdots+e_t-e_{t+1}$, then $v_{t+1}- v_t$ is irreducible.
\end{lemma}

\begin{lemma}[Lemma~4.9 in \cite{Prism2016}]\label{lem:j-1}
For any $v_j\in S$, we have $j-1\in\supp v_j$.
\end{lemma}


For the rest of this section, suppose $\sigma = (\sigma_0,\sigma_1,\dots,\sigma_{n+1}) \in \Z^{n+2}$ is a changemaker vector such that $(\sigma)^\perp$ is isomorphic to a C-type lattice $C(p,q)$ with $q > p$. Also, let $x_0,\dots,x_n$ be the vertex basis of $C(p,q)$, and let $S = (v_1,\dots,v_{n+1})$ be the standard basis of $(\sigma)^\perp$. Each $v_i$ is an irreducible element in a C-type lattice (Lemma~\ref{lem:irred}), so corresponds to some interval (Proposition~\ref{prop:IntervalsIrreducible}). By a slight abuse of notation, denote $[v_i]$ for the interval corresponding to $v_i$. Let $\epsilon_i\in\{\pm1\}$ satisfy $v_i = \epsilon_i [v_i]$. 

The C-type lattice $C(p,q)$ contains an element $x_0$ with $|x_0| = 4$, and any vector of norm $4$ in $\Z^{n+2}$ is of the form either $ \pm 2e_k$ or $\pm e_{k_0}\pm e_{k_1}\pm e_{k_2}\pm e_{k_3}$ for distinct indices $k_i$. Vectors of the first form cannot be in $(\sigma)^\perp$ since $\sigma_0 \neq 0$, so $x_0$ must be of the second form. In fact, we can say a little bit more about how $x_0$ can be written in terms of the $e_i$. We start by the following lemma.

\begin{lemma}\label{Lem:NoV2}
There is no element $v\in C(p,q)$ with $\braket{v}{x_0} \neq 0$ and $|v| = 2$. 
\end{lemma}
\begin{proof}
Since $C(p,q)$ is indecomposible, it contains no $x$ with $|x| = 1$ (such an $x$ would generate a $\Z$-summand of $C(p,q)$). Therefore, if $v \in C(p,q)$ with $|v| = 2$, it must be irreducible, so $v = \pm[I]$ for $[I]$ an interval. By Lemma~\ref{Lem:IntervalNorm}, $[I]$ contains only $x_0$ or elements of norm $2$. In particular, $[I]$ does not contain $x_1$, since $a_1 \ge 3$. This means that $[I]$ also cannot contain $x_0$, since then $[I] = x_0$ and $|v| = 4$. Therefore, $\braket{[I]}{x_0} = 0$, and so $\braket{v}{x_0} = 0$. 
\end{proof}

\begin{prop}\label{x0}
For some indices $k_1 < k_2 < k_3$, $x_0$ is equal to one of ${e_0 + e_{k_1} + e_{k_2} - e_{k_3}}$ or ${e_0 - e_{k_1} - e_{k_2} + e_{k_3}}$, possibly after a global sign change in the isomorphism between $(\sigma)^\perp$ and $C(p,q)$. 
\end{prop}
\begin{proof}
Since $|x_0| = 4$ and $x_0 \in (\sigma)^\perp$,
\begin{equation*}
x_0 = \delta_0 e_{k_0} + \delta_1 e_{k_1} + \delta_2 e_{k_2} + \delta_2 e_{k_3}
\end{equation*}
for indices $k_0 < k_1 < k_2 < k_3$ and signs $\delta_i$ such that $\sum_i \delta_i \sigma_i = 0$. By a global sign change, we might as well assume that $\delta_0 = 1$. If $k_0 > 0$, $\braket{x_0}{v_{k_0}} = -1$ is odd, violating Lemma~\ref{Lem:X0Odd}. So $k_0 = 0$.

We claim that if $\sigma_{k_i} = \sigma_{k_j}$, then $\delta_i = \delta_j$. Otherwise $v = \delta_i e_{k_i} + \delta_j e_{k_j}$ would be in $(\sigma)^\perp$ with $|v| = 2$ and $\braket{v}{x_0} = 2$, which contradicts Lemma~\ref{Lem:NoV2}. Therefore, if $\delta_1 = -1$ then $\sigma_1 > \sigma_0$, and so $\delta_0 \sigma_0 + \delta_1 \sigma_1 < 0$. Therefore, $\delta_2 \sigma_2 + \delta_3 \sigma_3 > 0$. Since $\sigma_2 \le \sigma_3$, this means that $\delta_3 = 1$, and then $\delta_2 = -1$ since $\sigma_1 < \sigma_0 + \sigma_2 + \sigma_3$. In the other case, if $\delta_1 = 1$ then $\delta_0 \sigma_0 + \delta_1 \sigma_1 > 0$, so $\delta_2 \sigma_2 + \delta_3 \sigma_3 < 0$ and $\delta_3 = -1$. If also $\delta_2 = -1$, then
\begin{equation*}
\sigma_0 + \sigma_1 = \sigma_2 + \sigma_3.
\end{equation*}
Since $\sigma_0 \le \sigma_1 \le \sigma_2 \le \sigma_3$, this can only happen if all of the $\sigma_i$ are equal, again contradicting the fact that if $\sigma_i = \sigma_j$ we must have $\delta_i = \delta_j$.
\end{proof}

\begin{cor}\label{v1}
The vector $v_1$ is equal to $2e_0 - e_1$ if $k_1 > 1$, and $e_0 - e_1$ otherwise. If ${x_0 = e_0 - e_{k_1} - e_{k_2} + e_{k_3}}$, the first of these occurs. 
\end{cor}
\begin{proof}
Note that $v_1$ is always either $e_0 - e_1$ or $2e_0 - e_1$. Using Lemma~\ref{Lem:X0Odd}, the first statement of the lemma follows. For the second statement, if $k_1 = 1$ and $v_1 = e_1 - e_0$, then if ${x_0 = e_0 - e_{k_1} - e_{k_2} + e_{k_3}}$ we have that $\braket{v_1}{x_0} = 2$ and $|v_1| = 2$, contradicting Lemma~\ref{Lem:NoV2}. 
\end{proof}

\begin{lemma}\label{lem:tightvector}
If $k_1 > 1$, $v_1$ is the only tight vector. If $k_1 = 1$, $v_{k_2}$ can be tight but there is no other tight vector.
\end{lemma}
\begin{proof}
We claim that if $v_t$ is tight, then either $t < k_1$ or $t= k_2$. Using Lemma~\ref{Lem:X0Odd}, we must have that either $k_2 \le t < k_3$ or $t < k_1$ as otherwise $v_t$ will have odd pairing with $x_0$. If $k_2<t < k_3$, then 
\[
\sigma_t = 1 + \sigma_0 + \sigma_1 + \cdots + \sigma_{t-1} \ge 1 + \sigma_0 + \sigma_{k_1} + \sigma_{k_2}.
\]
 However, by Proposition~\ref{x0}, the fact that $\braket{x_0}{\sigma} = 0$ implies that
\begin{equation*}
\sigma_{k_3} = \sigma_{k_2} + \sigma_{k_1} \pm \sigma_0 \le \sigma_{k_2} + \sigma_{k_1} + \sigma_0 < \sigma_t,
\end{equation*}
contradicting the fact that $t < k_3$. The claim follows. 

If $k_1 = 1$, it is only possible that $t = k_2$, so the second statement of the lemma follows. Suppose now that $k_1 > 1$. We have that $v_1 = 2e_0 - e_1$ by Corollary~\ref{v1}. So if $v_t$ is tight with $t > 1$, we get that $\braket{v_1}{v_t} = 3$ and $|v_t| > |v_1| = 5$. Also, since either $t < k_1$ or $t = k_2$, $\braket{v_t}{x_0} = \braket{v_1}{x_0} = 2$. Therefore, either $\epsilon_1 = -1$ and $[v_1]$ has left endpoint $1$, or $\epsilon_1 = 1$ and $[v_1]$ has left endpoint $0$, and the same holds for $\epsilon_t$ and $[v_t]$. By Lemma~\ref{intervalproduct}, 
\begin{equation*}
3 = \braket{v_1}{v_t} = \epsilon_1\epsilon_t(|[v_1 \cap v_t]| - \delta([v_1],[v_t]) ),
\end{equation*}
$|[v_1 \cap v_t]| \ge 2$ and $\delta([v_1],[v_t]) \le 3$, so if $\epsilon_1\neq \epsilon_t$, the right hand side of this equation is at most $1$. Therefore, $\epsilon_1 = \epsilon_t$, and the left endpoints of $[v_1]$ and $[v_t]$ are equal. Since $|v_t| > |v_1|$, the right endpoint of $[v_t]$ is to the right of the right endpoint of $[v_1]$. This means that $\delta([v_1],[v_t]) = 1$ and $v_1 \cap v_t =v_1$, so
\begin{equation*}
\braket{v_1}{v_t} = \epsilon_1\epsilon_t(|[v_1 \cap v_t]| - \delta([v_1],[v_t]) ) = |[v_1]| - 1 = 4 \neq 3.
\end{equation*}
Therefore, $v_1$ is the only tight vector.
\end{proof}


\begin{lemma}\label{positivity}
For $j \neq k_3$, $\braket{v_j}{x_0} \ge 0$. 
\end{lemma}
\begin{proof}
Using~Proposition~\ref{x0}, either ${x_0 = e_0 + e_{k_1} + e_{k_2} - e_{k_3}}$ or $x_0= e_0 - e_{k_1} - e_{k_2} + e_{k_3}$. If $x_0 = e_0 + e_{k_1} + e_{k_2} - e_{k_3}$, it would only be possible to have $\braket{v_j}{x_0} < 0$ for $j = k_1$ or $j = k_2$. However, in these cases one has $\braket{v_j}{x_0} \ge -1$, and since $\braket{v_j}{x_0}$ is even, it follows that $\braket{v_j}{x_0} \ge 0$. If ${x_0 = e_0 - e_{k_1} - e_{k_2} + e_{k_3}}$, then $\braket{v_j}{x_0}$ is always at least $-3$, since $\braket{v_j}{e_0} \ge 0$. Therefore, since it is even, $\braket{v_j}{x_0} \ge -2$. Given that $j \ne k_3$, the only possible way to have $\braket{v_j}{x_0} = -2$ is that $k_1, k_2 \in \supp^+(v_j)$, and $0,k_3 \not \in \supp^+(v_j)$. Observe that this cannot happen since then $v_j + x_0$ is still of the form $-e_j + \sum_{i \in A'} e_i$ for some $A' \subset \{0,\dots,j-1\}$, but $A'$ is lexicographically after $\supp^+ v_j$, contradicting the maximality criterion in Definition~\ref{stbasis}.
\end{proof}

\begin{lemma}\label{excludedintervals}
If $v_i$ and $v_j$ are two unbreakable standard basis vectors with $i,j \neq k_3$, then it cannot be the case that $[v_i]$ contains $x_0$ and $[v_j]$ contains $x_1$ but not $x_0$. In particular, $\delta([v_i],[v_j]) \le 2$.
\end{lemma}
\begin{proof}
Assume the contrary. Since $i,j \neq k_3$, and $k_3 = \max \supp(x_0)$, neither $v_i$ nor $v_j$ is equal to $\pm x_0$, and by Lemma~\ref{positivity}, $\braket{v_i}{x_0}$ and $\braket{v_j}{x_0}$ are both nonnegative. Therefore, $\braket{v_i}{x_0} = \braket{v_j}{x_0} = 2$. Since $x_0$ is contained in $[v_i]$, the left endpoint of $[v_i]$ is $0$ and $\epsilon_i = 1$. Similarly, $[v_j]$ has left endpoint $1$ and $\epsilon_j = -1$. Therefore, $\delta([v_i], [v_j])$ is either $2$ or $3$, and since $v_i$ and $v_j$ are unbreakable and $a_1\ge3$, $z_i=z_j=1$ and $|[v_i \cap v_j]| = |v_i| = |v_j|=a_1$. This means that
\begin{equation}\label{OwnRef}
\braket{v_i}{v_j} = \epsilon_i \epsilon_j \left(|[v_i \cap v_j]| - \delta([v_i],[v_j])\right) = -|v_i| + \delta([v_i],[v_j]) = -|v_j| + \delta([v_i],[v_j])
\end{equation}
Since $v_i$ and $v_j$ are standard basis vectors, $\braket{v_i}{v_j} \ge -1$. Since $|v_i| \ge 3$ and $\delta([v_i],[v_j])$ is either $2$ or $3$, $|v_i|$ is either $3$ or $4$. That is, using~Equation~\eqref{OwnRef}, $\braket{v_i}{v_j}$ is equal to $-1$ if $|v_i| = 4$ and either $0$ or $-1$ if $|v_i| = 3$. In particular, 
\begin{equation}\label{eq:ijNeg}
\braket{v_i}{v_j}\le 0.
\end{equation}

Using Proposition~\ref{x0}, suppose first that $x_0 = e_0 + e_{k_1} + e_{k_2} - e_{k_3}$. Then since $\braket{v_i}{x_0} = \braket{v_j}{x_0} = 2$ and $i,j \neq k_3$, each of $\supp^+(v_i)$ and $\supp^+(v_j)$ contain at least two of $0, k_1$, and $k_2$, and $i,j\notin\{k_1,k_2\}$. In particular, $\supp^+(v_i)$ and $\supp^+(v_j)$ intersect, so $\braket{v_i}{v_j} \ge 0$. Therefore, using Equation~\eqref{OwnRef} and the earlier discussion, we must have $|v_i| = |v_j| = 3$, so $\supp^+(v_i)$ and $\supp^+(v_j)$ in fact contain no elements outside of $\{0,k_1,k_2\}$. In particular, $\supp^+(v_i)$ does not contain $j$, and vice versa, $\supp^+(v_j)$ does not contain $i$. Therefore, we get that $\braket{v_i}{v_j} \ge 1$ which is a contradiction to (\ref{eq:ijNeg}). 

If now $x_0 = e_0 - e_{k_1} - e_{k_2} + e_{k_3}$, then since $\braket{v_i}{x_0} = 2$ and $i \neq k_3$, there are two cases: Case~1 is that
$\supp^+(v_i)$ contains $0$ and $k_3$ but not $k_1$ and $k_2$, and Case~2 is that $i = k_2$ or $k_1$, $\supp^+(v_i)$ contains $0$, and (if $i = k_2$), $\supp^+(v_i)$ does not contain $k_1$. The same holds for $v_j$. 
If one of $v_i$ and $v_j$ is in Case~1, then $\braket{v_i}{v_j} \ge 1$, a contradiction to (\ref{eq:ijNeg}). If both $v_i$ and $v_j$ are in Case~2, we may assume $i=k_1$ and $j=k_2$, and we still have $\braket{v_i}{v_j} \ge 1$, a contradiction.
\end{proof}

\begin{cor}\label{unbreakablepairing}
If $v_i$ and $v_j$ are two unbreakable standard basis vectors with $i \neq j$ and $i,j \neq k_3$, then $|\braket{v_i}{v_j}| \le 1$, with equality if only if $[v_i]$ abuts $[v_j]$. 
\end{cor}
\begin{proof}
If neither $[v_i]$ nor $[v_j]$ contains $x_0$, then both $v_i$ and $v_j$ are contained in a linear sublattice of $C(p,q)$ and this reduces to~\cite[Lemma~4.4]{greene:LSRP}. Similarly, if one of $[v_i]$ or $[v_j]$ contains $x_0$ and the other contains neither $x_0$ nor $x_1$, or if both $[v_i]$ and $[v_j]$ contain $x_0$, then reflecting both $v_i$ and $v_j$ about $x_0^{\perp}$ puts both of them in a linear sublattice of $C(p,q)$. Using Lemma~\ref{excludedintervals}, these are the only possibilities.
\end{proof}

\begin{cor}\label{Cor:ZjDistinct}
If $v_i$ and $v_j$ are unbreakable with $|v_i|,|v_j| \ge 3$, $i \neq j$ and $i,j \neq k_3$, then ${z_i} \neq {z_j}$, where ${z_i}$ and ${z_j}$ are defined in Definition~\ref{Def:Zj}.
\end{cor}
\begin{proof}
Suppose for contradiction $x_{z_i} = x_{z_j}$. By Lemma~\ref{excludedintervals}, $\delta([v_i],[v_j]) \le 2$. Therefore, using Lemmas~\ref{intervalproduct}~and~\ref{Lem:IntervalNorm},
\begin{equation}\label{eq:vivjGe1}
\braket{[v_i]}{[v_j]} = |[v_i \cap v_j]| - \delta([v_i],[v_j]) = |x_{z_i}| - \delta([v_i],[v_j]) \ge 3 - 2 = 1,
\end{equation}
By Corollary~\ref{unbreakablepairing}, $\braket{[v_i]}{[v_j]} =1$ and $[v_i]$ abuts $[v_j]$. We would then have $\delta=1$, so the equality in (\ref{eq:vivjGe1}) cannot be attained, a contradiction.
\end{proof}

\begin{cor}\label{x0pairing}
There is at most one $j \neq k_3$ for which $v_j$ is unbreakable and $\braket{v_j}{x_0}$ is nonzero. 
\end{cor}
\begin{proof}
Since $a_1 \ge 3$, if there exists an unbreakable standard basis element $v_j$ for which $\braket{v_j}{x_0} \neq 0$, $j \neq k_3$, then $x_{z_j} =x_1$. It follows from Corollary~\ref{Cor:ZjDistinct} that there exists at most one such $j$.
\end{proof}

Since the pairings of $v_{k_3}$ with other standard basis vectors are difficult to control, and since Corollary~\ref{x0pairing} gives good control on the pairings between $x_0$ and the other standard basis vectors, it will be easier in what follows if we replace $S$ with the modified basis 
\begin{equation}\label{Eq:S'}
S' = \left(S \setminus \{v_{k_3}\}\right) \cup \{x_0\}. 
\end{equation}
The set $S'$ is still a basis of $(\sigma)^\perp$ because $\braket{x_0}{e_{k_3}} = \pm 1$ but $\braket{x_0}{e_j} = 0$ for $j > k_3$, so if we write $x_0$ as a linear combination of elements of $S$, the coefficient of $v_{k_3}$ will be $\pm 1$.

Using Lemmas~\ref{unbreakablepairing} and \ref{x0pairing}, we can relate the pairings between elements of $S'$ very closely to the geometry of the intervals. It will be convenient to use two graphs associated to a C-type lattice. Recall that the pairing graph $\hat{G}(V)$ for a subset $V$ of a lattice $L$ has vertex set $V$ and an edge $(v_i,v_j)$ whenever $\braket{v_i}{v_j} \neq 0$ (Definition~\ref{defn:pairinggraph}).

 
\begin{definition}
If $T$ is a set of irreducible vectors in a C-type lattice $C(p,q)$, the {\it intersection graph} $G(T)$ has vertex set $T$, and an edge between $v$ and $w$ if the intervals corresponding to $v$ and $w$ abut. We write $v\sim w$ if $v$ and $w$ are connected in $G(T)$.
\end{definition}

\begin{lemma}\label{lem:PairingNonzero}
If the intervals corresponding to $v$ and $w$ abut, then $\braket vw\ne0$.
\end{lemma}
\begin{proof}
If one of $v,w$ is $x_0$, $\braket vw=\pm2\ne0$. If none of $v,w$ is $x_0$, then $\delta([v],[w])=1$, our conclusion follows from Lemma~\ref{intervalproduct}.
\end{proof}

The following is immediate from Corollary~\ref{unbreakablepairing} and Lemma~\ref{lem:PairingNonzero}:
\begin{prop}
For $T \subset S'$, $G(T)$ is obtained from $\hat{G}(T)$ by removing some edges incident to breakable vectors.
\end{prop}

In particular, if we write $\bar{S'}$ for the set of unbreakable elements of $S'$, ${G(\bar{S'}) = \hat{G}(\bar{S'})}$. The main use we have for this result is the following structural facts about the intersection graph. 

\begin{definition}
A {\it claw} in a graph $G$ is a quadruple $(v,w_1,w_2,w_3)$ of vertices such that $v$ neighbors all the $w_i$, but no two of the $w_i$ neighbor each other.
\end{definition}

\begin{lemma}[Lemma~4.8 of~\cite{greene:LSRP}]\label{claw}
The intersection graph $G(T)$ has no claws.
\end{lemma}

\begin{definition}
Given a set $T$ of unbreakable elements in a C-type lattice and $v_1,v_2,v_3 \in T$, $(v_1,v_2,v_3)$ is a {\it heavy triple} if $|v_i| \ge 3$ and $v_i\ne\pm x_0$ for each $i$, and if each pair among the $v_i$ is connected by a path in $G(T)$ disjoint from the third. 
\end{definition}

\begin{lemma}[Based on Lemma~4.10 of~\cite{greene:LSRP}]\label{heavytriple}
$G(\bar{S'})$ has no heavy triples.
\end{lemma}
\begin{proof}
If $v_i,v_j,$ and $v_k$ are unbreakable and have norm at least $3$, and none of them is $\pm x_0$, then by Corollary~\ref{Cor:ZjDistinct} we might as well assume ${z_i} < {z_j} < {z_k}$. Then any path from $v_i$ to $v_k$ in $G(\bar{S'})$ includes some $v_\ell \in \bar{S'}$ such that $[v_\ell]$ contains $x_{z_j}$, where $\bar{S'}$ is defined in~\eqref{Eq:S'}. But then $\ell = j$, so $(v_i,v_j,v_k)$ is not heavy.
\end{proof}

The proof of the following lemma is identical to~\cite[Lemma~3.8]{greene:LSRP}.

\begin{lemma}\label{Lem:CompleteSubgraph}
If the elements of $T$ are linearly independent, any cycle in $G(T)$ induces a complete subgraph.
\end{lemma}

\begin{cor}[Based on Lemma~4.11 of~\cite{greene:LSRP}]\label{cycles}
Any cycle in $G(\bar{S'})$ has length three.
\end{cor}
\begin{proof}
By Corollary~\ref{x0pairing}, any cycle in $G(\bar{S'})$ does not contain $x_0$.
Using Lemma~\ref{Lem:CompleteSubgraph}, the cycle will contain at most two vertices of norm $>2$ to avoid producing a heavy triple. (See Definition~\ref{Def:HighNorm}.) If it had two vertices of norm $2$, using Lemma~\ref{Lem:CompleteSubgraph}, they would have nonzero inner product, so must be of the form $v_i = e_{i-1} - e_i$ and $v_{i+1} = e_i - e_{i+1}$ for some $i$. But for any other $j$ ($j\ne i, i+1$), Lemma~\ref{gappy3} implies that $\supp(v_j) \cap \{i-1,i,i+1\}$ is one of $\emptyset$, $\{i+1\}$, $\{i,i+1\}$, or $\{i-1,i,i+1\}$. In none of these cases does $v_j$ have nonzero inner product with both $v_i$ and $v_{i+1}$, a criterion that must be fulfilled by Lemma~\ref{Lem:CompleteSubgraph}. That is, any cycle in $G(\bar{S'})$ must be of length three. 
\end{proof}

\begin{lemma}\label{lem:AllNorm2}
Let $m<N$ be two possitive integers satisfying $k_3\notin[m,N]$. Suppose that $v_m$ is unbreakable and it neighbors either $x_0$ or some unbreakable $v_j$ with $j<m$. Suppose that for any index $i$ satisfying $m< i\le N$, we have $\min\supp(v_i)\ge m$, and $v_i$ is unbreakable. Then $|v_i|=2$ for any $i$ satisfying $m< i\le N$.
\end{lemma}
\begin{proof}
When $i=m+1$, we clearly have $|v_i|=2$. Now assume $|v_i|=2$ for any $i$ satisfying $m< i< l\le N$, we want to prove $|v_l|=2$. Let $t=\min\supp(v_l)\ge m$, then $v_l$ is just right by Lemmas~\ref{gappy3} and \ref{lem:j-1}. If $m<t<l-1$, we would have a claw $(v_t,v_l,v_{t-1},v_{t+1})$. If $t=m$ and $v_m$ neighbors $x_0$, we would have a claw $(v_m,v_l,x_0,v_{m+1})$ by Corollary~\ref{x0pairing}. If $t=m$ and $v_m$ neighbors an unbreakable $v_j$ with $j<m$, we would have a claw
$(v_m,v_l,v_{j},v_{m+1})$. So $t=l-1$ and $|v_l|=2$.
\end{proof}

\section{{\bf $k_1=1, k_2>2$}}\label{sec:Case2}

In this section we consider, in the notation of Proposition~\ref{x0}, the case where $k_1=1$ and $k_2>2$. Using Corollary~\ref{v1}, one has
\begin{equation}\label{Eq:Case2X0}
x_0=e_0 + e_1 + e_{k_2} - e_{k_3}, 
\end{equation}
where $k_2>2$. 
Also, we have that $v_1 = e_0 - e_1$. So 
\begin{equation}\label{Eq:Case2Sigma}
\sigma_0 = \sigma_1 = 1.
\end{equation} 
By Lemmas~\ref{lem:BrIsTight} and \ref{lem:tightvector}, the only possible breakable vector is $v_{k_2}$.
In what follows we classify all changemaker vectors whose orthogonal complements are isomorphic to C-type lattices with $x_0$ as given in~\eqref{Eq:Case2X0} and $k_2>2$. We start by determining the first $k_3+1$ components of such changemaker vectors.  

\begin{prop}\label{prop:Segk_1=1k_2>2}
If $k_1 = 1$ and $k_2 > 2$, the initial segment $(\sigma_0,\sigma_1,\cdots,\sigma_{k_3})$ of $\sigma$ is equal to $(1,1,2^{[s]}, \sigma_{k_2},\sigma_{k_2} + 2)$ for some $s > 0$.
\end{prop}
\begin{proof}
We start by observing that, using Lemma~\ref{Lem:X0Odd}, we must have $v_2 = e_0 + e_1 - e_2$. So $\sigma_2 = 2$. By Corollary~\ref{x0pairing}, $\min\supp(v_i)\ge2$ for all $2 < i < k_2$. It follows from Lemma~\ref{lem:AllNorm2} that $|v_i| = 2$ for all $2 < i < k_2$. So $\sigma_i = 2$ for $2 \le i < k_2$. Now, using \eqref{Eq:Case2X0}~and~\eqref{Eq:Case2Sigma} together with the fact that $\braket{\sigma}{x_0} = 0$, we get that $\sigma_{k_3} = \sigma_{k_2} + 2$. We claim that $k_3 = k_2 +1$. Suppose for contradiction that $k_3 \neq k_2 +1$. The component $\sigma_{k_2 + 1}$ must be between $\sigma_{k_2}$ and $\sigma_{k_2} + 2 = \sigma_{k_3}$. If $\sigma_{k_2 + 1}$ is equal to either $\sigma_{k_2}$ or $\sigma_{k_3}$, there will be an element $v \in (\sigma)^\perp$ with $\braket{v}{x_0} = 1$, contradicting Lemma~\ref{Lem:X0Odd}. If $\sigma_{k_2 + 1} = \sigma_{k_2} + 1$, then $v_{k_2 + 1} = e_1 + e_{k_2} - e_{k_2 + 1}$. But then $\braket{v_{k_2 + 1}}{x_0} = 2 \neq 0$, contradicting Corollary~\ref{x0pairing} since $\braket{v_2}{x_0} = 2$. This finishes the proof.
\end{proof}

\begin{cor}\label{cor:sigmak2}
In the situation of Proposition~\ref{prop:Segk_1=1k_2>2}, the component $\sigma_{k_2}$ of the changemaker vector is one of $2s-1$, $2s+1$, or $2s+3$. These correspond to $v_{k_2}$ being gappy, just right, or tight, respectively.
\end{cor}
\begin{proof}
If $v_{k_2}$ is tight, the third of these possibilities occurs. If not, using Corollary~\ref{x0pairing}, we get that $\braket{v_{k_2}}{x_0} = 0$. (Note that $\braket{v_2}{x_0} = 2$.) So $1 \in \supp^+(v_{k_2})$ and $0 \not \in \supp^+(v_{k_2})$. Since $|v_j| = 2$ for $2 < j < k_2$, Lemma~\ref{gappy3} implies that the only possible gappy index for $v_{k_2}$ is $1$, so
\begin{equation*}
v_{k_2} = e_1 + e_j + e_{j+1} + \cdots + e_{k_2 - 1} - e_{k_2},
\end{equation*}
for some $1 < j < k_2$. If $j > 3$, the pairing graph will have a cycle on $v_2, \cdots, v_j, v_{k_2}$ of length larger than $3$, contradicting Corollary~\ref{cycles}. In particular, if $1$ is indeed a gappy index for $v_{k_2}$, then $j = 3$, and $\sigma_{k_2} = 2s -1$. Otherwise one has $j = 2$, and therefore $\sigma_{k_2} = 2s+1$. 
\end{proof}

It turns out that the classification will highly depend on the type of the vector $v_{k_2}$: whether it is tight, just right, or gappy. For $j > k_3$, let 
\begin{equation}\label{Eq:Sj}
S_j = \supp(v_j) \cap \{0,1,\dots,k_3\},
\end{equation} 
and let 
\begin{equation}\label{Eq:S'j}
S_j' = \supp(v_j) \cap \{0,1,k_2,k_3\}.
\end{equation} 
Given that $\braket{v_2}{x_0} = 2$ and, using Corollary~\ref{x0pairing}, we must have $\braket{v_j}{x_0} = 0$, and that $S_j'$ is one of $\emptyset$, $\{1,k_3\}$, or $\{k_2,k_3\}$ by Lemma~\ref{gappy3}. Figure~\ref{pairinggraphs} depicts the paring graphs of the possible changemaker C-type lattices on their first $k_3$ vectors of the basis $S'$, defined in~\eqref{Eq:S'}, depending on the type of $v_{k_2}$. With a slight abuse of notation, we often use $v_{k_3}$ in place of $x_0$.

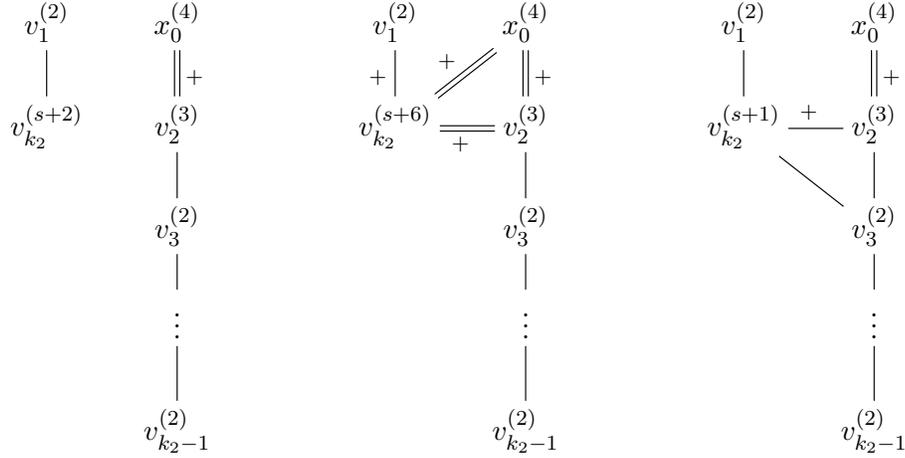
\begin{figure}[t]
\centering
\begin{align*}
\xymatrix@R1.5em@C1.5em{
v_1^{(2)} \ar@{-}[d] & x_0^{(4)} \ar@{=}[d]^+ \\
v_{k_2}^{(s+2)} & v_2^{(3)} \ar@{-}[d] \\
 & v_{3}^{(2)} \ar@{-}[d] \\
 & \vdots \ar@{-}[d] \\
 & v_{k_2 - 1}^{(2)}
}
&&
\xymatrix@R1.5em@C1.5em{
v_1^{(2)} \ar@{-}[d]_+ & x_0^{(4)} \ar@{=}[d]^+ \ar@{=}[dl]_+\\
v_{k_2}^{(s+6)} \ar@{=}[r]_+ & v_2^{(3)} \ar@{-}[d] \\
 & v_{3}^{(2)} \ar@{-}[d] \\
 & \vdots \ar@{-}[d] \\
 & v_{k_2 - 1}^{(2)}
}
&&
\xymatrix@R1.5em@C1.5em{
v_1^{(2)} \ar@{-}[d] & x_0^{(4)} \ar@{=}[d]^+ \\
v_{k_2}^{(s+1)} \ar@{-}[r]^+ \ar@{-}[dr] & v_2^{(3)} \ar@{-}[d] \\
 & v_{3}^{(2)} \ar@{-}[d] \\
 & \vdots \ar@{-}[d] \\
 & v_{k_2 - 1}^{(2)}
}
\end{align*}
\caption{Pairing graphs of the standard basis when $v_{k_2}$ is just right (left), tight (center), and gappy (right). Superscripts give the norm of the basis vector, the number of edges gives the absolute value of the inner product, and an edge is labelled with $+$ if the inner product is positive. }\label{pairinggraphs}
\end{figure}

With the notation of this section in place:
\begin{lemma}\label{Sjempty}
If $S_j' = \emptyset$, $S_j$ is either $\emptyset$ or $\{k_2 - 1\}$. In the second case, $v_{k_2}$ is not gappy. 
\end{lemma}
\begin{proof}
Set $i = \min  S_j$. Suppose for contradiction that $S_j$ is nonempty and $i < k_2 - 1$. If $i > 2$,  then there will be a claw on $v_i, v_{i+1}, v_{i-1}, v_j$. If $i = 2$ there will be a claw $(v_2, v_3, x_0, v_j)$. Therefore, $i = k_2 - 1$, and so the first statement follows. If $S_j = \{k_2 - 1\}$, then $\braket{v_j}{v_{k_2 - 1}} = -1$, $\braket{v_j}{v_{k_2}} = 1$, and $\braket{v_j}{v_i} = 0$ for all other $i \le k_3$, so if $v_{k_2}$ is gappy there is a claw $(v_{k_2}, v_1, v_2, v_j)$ (see Figure~\ref{pairinggraphs}).
\end{proof}

\begin{lemma}\label{Sjup}
If $S_j' = \{k_2,k_3\}$, $S_j$ is either $\{k_2,k_3\}$ or $\{k_2 - 1, k_2, k_3\}$. In either case, $v_{k_2}$ is not gappy. 
\end{lemma}
\begin{proof}
Again, set $i = \min  S_j$. If $i< k_2-1$, there will be a claw on either $v_i, v_{i+1}, v_{i-1}, v_j$ or $v_2, v_3, x_0, v_j$, depending on whether $i>2$ or $i=2$. So the first statement follows. Corresponding to the two possibilities for $S_j$, the vector $v_j$ will have nonzero inner product with either $v_{k_2}$ or $v_{k_2 - 1}$, but no other $v_i$ with $i \le k_3$. If $v_{k_2}$ is gappy, this creates a claw $(v_{k_2}, v_1, v_2, v_j)$ in the first case, and  a heavy triple $(v_2,v_{k_2},v_j)$ in the second: again, see Figure~\ref{pairinggraphs}.
\end{proof}

\begin{lemma}\label{Sjdown}
If $S_j' = \{1,k_3\}$, either $S_j$ is one of $\{1,2,3,\dots,k_2 - 1, k_3\}$ and $\{1,3,\dots,k_2 - 1, k_3\}$ and $v_{k_2}$ is tight, or $S_j = \{1,k_3\}$, $s = 1$, and $v_{k_2}$ is not gappy.
\end{lemma}
\begin{proof}
Using Lemma~\ref{gappy3}, none of $2,\dots,k_2-2$ can be a gappy index for $v_j$. Thus, we must have either $S_j = \{1,k_3\}$ or
$
S_j = \{1,k,k+1,\dots,k_2 - 1, k_3\}
$
for some $1 < k < k_2$. 

In the first case, $v_j$ will have nonzero inner product with just $v_1$, $v_2$, and $v_{k_2}$. If $v_{k_2}$ is gappy, this creates a heavy triple $(v_2, v_{k_2}, v_j)$. If $v_{k_2}$ is just right or tight, this creates a claw $(v_2, v_j, x_0, v_3)$, unless $s = 1$: see Figure~\ref{pairinggraphs}.

In the second case, to avoid a cycle $(v_2,v_3,\dots,v_k,v_j)$ of length longer than $3$ (Corollary~\ref{cycles}) we must have $k$ equal to $2$ or $3$. Then $\braket{v_j}{v_{k_2}}$ is either $s$ or $s+1$, and unless $v_{k_2}$ is tight this must be at most $1$ (Corollary~\ref{unbreakablepairing}). Since $s \ge 1$, if $v_{k_2}$ is not tight, we must have $\braket{v_j}{v_{k_2}} = s = 1$. Note that in this case $k_3 = 4$, $k_2=3$, $v_{k_2} = e_1 + e_2 - e_3$, and $S_j = \{ 1,2,4\}$. Consequently, $\braket{v_j}{v_3} = 2$, again contradicting Corollary~\ref{unbreakablepairing}.
\end{proof}

\begin{prop}\label{example}
If $v_{k_2}$ is gappy, then $s\ge2$ and $n+1 = k_3$ (i.e. $v_{k_3}$ is the last standard basis vector). The corresponding changemaker vectors are 
\[(1,1,2^{[s]},2s-1,2s+1), s\ge2.\]
\end{prop}
\begin{proof}
By Corollary~\ref{cor:sigmak2}, $\sigma_{k_2}=2s-1\ge2$, so $s\ge2$.
By Lemmas~\ref{Sjdown},~\ref{Sjup},~and~\ref{Sjempty}, we get that $S_{j} = \emptyset$ for all $j > k_3$. If $v_{k_3 +1}$ existed it would have $k_3 \in S_{k_3 + 1}$. 
\end{proof}

\begin{prop}\label{1k2>2}
If $v_{k_2}$ is just right, then one of the following holds:
\begin{enumerate}
\item $v_{k_3 + 1} = e_{k_2} + e_{k_3} - e_{k_3 + 1}$, $v_{k_3 + 2} = e_{k_2 - 1} + e_{k_2} + e_{k_3} + e_{k_3 + 1} - e_{k_3 + 2}$, and $k_3 + 2 = n+1$.
\item $v_{k_3 + 1} = e_{k_2 - 1} + e_{k_2} + e_{k_3} - e_{k_3 + 1}$, $v_{k_3 + 2} = e_{k_2} + e_{k_3} + e_{k_3 + 1} - e_{k_3 + 2}$, and $k_3 + 2 = n+1$.
\item $s = 1$, so $k_2 = 3$. $v_5 = e_3 + e_4 - e_5$, $|v_i| = 2$ for $5 < i < \ell$, $v_\ell = e_1 + e_4 + e_5 + \cdots + e_{\ell-1} - e_\ell$, and either $v_{\ell+1} = e_{\ell-1} + e_\ell - e_{\ell+1}$ and $|v_i| = 2$ for $i > \ell+1$, or $\ell = n+1$.
\item $s = 1$, so $k_2 = 3$. $v_5 = e_1 + e_4 - e_5$, and either $v_6 = e_3 + e_4 + e_5 - e_6$ and $|v_i| = 2$ for $i > 6$ or $5 = n+1$. 
\end{enumerate}
The corresponding changemaker vectors are
\begin{enumerate}
\item $(1,1,2^{[s]},2s+1,2s+3,4s+4,8s+10)$, $s\ge1$.
\item $(1,1,2^{[s]},2s+1,2s+3,4s+6,8s+10)$, $s\ge1$.
\item $(1,1,2,3,5,8^{[s]},8s+6,8s+14^{[t]})$, $s,t\ge0$, (the parameter $s$ in this family is not the previous $s$.)
\end{enumerate}
\end{prop}
\begin{proof}
We divide the proof into two cases, based on whether or not there is some $\ell$ with $S_\ell = \{1,k_3\}$. If there is no such $\ell$, then by Lemmas~\ref{Sjdown},~\ref{Sjup},~and~\ref{Sjempty}, for any $j > k_3$, $S_j$ is either empty or one of the three possibilities: $\{k_2 - 1\}$, $\{k_2,k_3\}$, or $\{k_2 - 1, k_2, k_3\}$. If $S_j=\{k_2 - 1\}$, $\braket{v_j}{v_{k_2 - 1}}$ and $\braket{v_j}{v_{k_2}}$ are both nonzero, but $\braket{v_j}{v_i} = 0$ for all other $i \le k_3$. If $S_j=\{k_2,k_3\}$, $\braket{v_j}{v_{k_2}}$ is nonzero but $\braket{v_j}{v_i} = 0$ for all other $i \le k_3$, and if $S_j=\{k_2-1,k_2,k_3\}$ only $\braket{v_j}{v_{k_2 - 1}}$ is nonzero. In particular, no $v_j$ with $j \le k_3$ except for $v_{k_2}$ and $v_{k_2 - 1}$ can have nonzero pairing with $v_i$ for some $i > k_3$. Furthermore, for $j$ equal to either $k_2$ or $k_2 - 1$, we claim that there can be at most one $i > k_3$ with $\braket{v_j}{v_i}$ nonzero: if there were two, there would be either a claw if they did not neighbor each other, or a heavy triple if they did. See Figure~\ref{pairinggraphs}. (For instance, if $v_r$ and $v_t$, with $r, t >k_3$, both have nonzero pairing with $v_{k_2-1}$, and also if $v_r$ and $v_t$ pair with each other, then there will be a heavy triple $(v_r, v_t, v_2)$.) Since the pairing graph of a basis must be connected, there in fact must be some $j > k_3$ with $\braket{v_j}{v_{k_2}}$ nonzero, and some $j > k_3$ with $\braket{v_j}{v_{k_2-1}}$ nonzero. This has two implications. First that the vector $v_{k_3+1}$ exists, and either $S_{k_3 + 1}=\{k_2,k_3\}$ or $S_{k_3 + 1}=\{k_2 - 1, k_2, k_3\}$. Second, there is another index $j' > k_3+1$ with $S_{j'}$ equal to the other of these two possibilities of $S_{k_3+1}$.

It remains only to show that $j' = k_3 +2$, and that there is no further standard basis vector. Since $S_{k_3 + 1}\cap S_{j'}=\{k_2,k_3\}$, in order to keep $\braket{v_{k_3 + 1}}{v_{j'}} \le 1$ (Corollary~\ref{unbreakablepairing}), it must be the case that $k_3 + 1 \in \supp^+(v_{j'})$, and in this case $\braket{v_{k_3 + 1}}{v_{j'}} = 1$. Therefore, $v_{k_3 + 1}$ and $v_{j'}$ are adjacent in the intersection graph. If $j' > k_3 + 2$, then since $S_{k_3 + 2} = \emptyset$, we get that $|v_{k_3 + 2}| = 2$. Therefore, using Lemma~\ref{gappy3}, $k_3 + 1$ cannot be a gappy index for $v_{j'}$, so $k_3 + 2 \in \supp^+(v_{j'})$. This means that $\braket{v_{j'}}{v_{k_3 + 2}} = 0$, so there is a claw on either $v_{k_3 + 1}, v_{k_2}, v_{k_3 + 2}, v_{j'}$ or $v_{k_3 + 1}, v_{k_2-1}, v_{k_3 + 2}, v_{j'}$, depending on the possibilities for $S_{k_3+1}$. Therefore, $j' = k_3 + 2$.

Finally, if $v_{k_3 + 3}$ existed, it would have $S_{k_3 + 3} = \emptyset$, so would equal either $e_{k_3 + 1} + e_{k_3 + 2} - e_{k_3 + 3}$ or $e_{k_3 + 2} - e_{k_3 + 3}$. Therefore, $v_{k_3 + 3}$ would have nonzero inner product with either $v_{k_3 + 1}$ or $v_{k_3 + 2}$ but not both, hence we get a claw centered at either $v_{k_3 + 1}$ or $v_{k_3 + 2}$. 

If there is some $\ell$ with $S_\ell = \{1,k_3\}$, then $s = 1$ by Lemma~\ref{Sjdown}. In this case, $\braket{v_\ell}{v_1} = -1$, $\braket{v_\ell}{v_2} = 1$, $k_2=3$, and $\braket{v_\ell}{v_{k_2}} = 1$. If, for any $ i > k_3$ with $i \neq \ell$, we had $\braket{v_i}{v_2} \neq 0$, there would be either a claw $(v_2,x_0,v_i,v_\ell)$ or a heavy triple $(v_2,v_i,v_\ell)$ depending on whether or not $[v_i]$ and $[v_\ell]$ abut. Since we must have $\braket{v_i}{v_2}=0$ for all $i > k_3$ with $i \neq \ell$, the set $S_i$ cannot be $\{1,k_3\}$, $\{ k_{2}-1\}$ or $\{ k_2-1, k_2, k_3\}$, so by Lemmas~\ref{Sjempty} \ref{Sjup} and \ref{Sjdown}, 
\begin{equation}\label{eq:Si2Cases}
S_i=\emptyset\text{ or }\{k_2,k_3\}.
\end{equation}
Also,  we have 
\begin{equation}\label{eq:pairingvl0}
\braket{v_i}{v_\ell} = 0,\quad \text{ for any $i > k_3$ with $i \neq \ell$}.
\end{equation}
 Otherwise, either $S_i = \emptyset$ in which case there would be a claw $(v_\ell, v_1, v_2, v_i)$, or $S_i = \{k_2,k_3\}$ and there would be a heavy triple $(v_i, v_\ell, v_{k_2})$. 

Now, $k_3 \in S_{k_3 + 1}$ (Lemma~\ref{lem:j-1}), so $S'_{k_3 + 1}$ is either $\{1,k_3\}$ or $\{k_2,k_3\}$. 
It follows from Lemmas~\ref{Sjup} and \ref{Sjdown} and (\ref{eq:Si2Cases}) that $S_{k_3 + 1}=S'_{k_3 + 1}$. If $S_{k_3 + 1} = \{1,k_3\}$, from (\ref{eq:pairingvl0}) we get that $\braket{v_{k_3+2}}{v_{k_3+1}} = 0$ if $n+1\ge k_3+2$, and therefore by (\ref{eq:Si2Cases}), $S_{k_3+2} = \{k_2,k_3\}$. We claim that $S_i = \emptyset$ for $i > k_3 + 2$, and also $k_3 + 1 \not \in \supp(v_i)$. Note that from (\ref{eq:Si2Cases}) if $S_i \not = \emptyset$, one necessarily has $S_i = \{ k_2, k_3\}$. Also, to avoid pairing with $v_{k_3+1}$, it must be the case that $k_3+1 \in \supp^+(v_i)$, but this would imply $\supp^+(v_i)\cap\supp^+(v_{k_3+2})=\{k_2,k_3,k_3+1\}$ hence $\braket{v_i}{v_{k_3+2}}\ge2$, contradicting Corollary~\ref{unbreakablepairing}. So $S_i=\emptyset$, hence $k_3 + 1 \not \in \supp(v_i)$ by (\ref{eq:pairingvl0}). This justifies the claim. It follows from Lemma~\ref{lem:AllNorm2} that $|v_i| = 2$ for $i > k_3 + 2$. This is the last of the possibilities listed in the statement of the proposition.

Lastly, suppose that $S_{k_3 + 1} = \{k_2,k_3\}$ (note that $S_\ell = \{ 1, k_3\}$). When $i>k_3+1$ and $i\ne\ell$, $S_i\ne\{k_2,k_3\}$, otherwise we get a heavy triple $(v_i,v_{k_2},v_{k_3+1})$. So $S_i=\emptyset$  by (\ref{eq:Si2Cases}).
By Lemma~\ref{lem:AllNorm2}, $|v_i| = 2$ for $k_3 + 1 < i < \ell$. By (\ref{eq:pairingvl0}), $v_\ell$ is orthogonal to all of $v_{k_3 + 1}, \dots, v_{\ell - 1}$, so all of $k_3 + 1,\dots,\ell-1$ are members of $\supp v_\ell$, forcing $v_\ell$ to be of the listed form. If $n+1\ge l+1$, $v_{\ell+1}$ is also orthogonal to $v_\ell$, so $\supp v_{\ell+1}\cap\{k_3 + 1,\dots,\ell-1\}$ contains exactly one element, which must be $\ell-1$ by Lemma~\ref{gappy3}. It follows that $v_{\ell+1} = e_{\ell-1} + e_\ell - e_{\ell+1}$, as desired. If, for some $i > \ell+1$, $\braket{v_i}{v_{\ell-1}}$ is nonzero, then $\ell-1\in\supp(v_i)$, and $\ell\in\supp(v_i)$ by (\ref{eq:pairingvl0}), so $\braket{v_i}{v_{\ell+1}}\ne0$ and hence  $(v_{k_3+1}, v_{\ell+1}, v_i)$ is a heavy triple. Therefore, $v_i$ is orthogonal to both $v_{\ell-1}$ and $v_\ell$ for $i > \ell+1$, so by Lemma~\ref{gappy3} $\min \supp v_{i} \ge \ell+1$. Then  Lemma~\ref{lem:AllNorm2} implies that $|v_i| = 2$ for $i > \ell+1$, so we are in the third listed situation.
\end{proof}

\begin{lemma}\label{lem:tightSj}
If $v_{k_2}$ is tight, $S_j$ is one of $\emptyset$, $\{k_2 - 1\}$, or $\{1,2,3,\dots,k_2 - 1, k_3\}$ for each $j > k_3$. 
\end{lemma}
\begin{proof}
By Lemmas~\ref{Sjdown},~\ref{Sjup},~and~\ref{Sjempty}, it suffices to show that $S_j$ cannot be $\{k_2,k_3\}$, $\{k_2 - 1, k_2, k_3\}$, $\{1,3,\dots,k_2 - 1, k_3\}$, or $\{1,k_3\}$. In the first case, $\braket{v_j}{v_{k_2}} = -1$ and $\braket{v_j}{v_i} = 0$ for all other $i \le k_3$. In particular since $v_j$ is orthogonal to $v_1$ and $v_2$, $v_j$ cannot neighbor $v_{k_2}$ in the intersection graph without creating a claw. Therefore, $[v_j] \pitchfork [v_{k_2}]$, and so $\delta([v_j],[v_{k_2}]) = 2$. In order to have $\braket{v_j}{v_{k_2}} = -1$, then, we must have $|v_j| =|[v_j\cap v_{k_2}]|= 3$ and $\epsilon_j = -\epsilon_{k_2}$. Since $\epsilon_j = -\epsilon_{k_2}$ and $[v_j] \pitchfork [v_{k_2}]$, $v_j + v_{k_2}$ is the sum of two distant intervals, so is reducible. However, since $|v_j| = 3$, $j = k_3 + 1$ and $v_j = e_{k_2} + e_{k_3} - e_{k_3 + 1}$, and so $v_{k_2} + v_j$ is irreducible by Lemma~\ref{lem:SumIrr}.

In the second case, $\braket{v_j}{v_{k_2 - 1}} = -1$ and all other $\braket{v_j}{v_i}$ with $i \le k_3$ are zero. Since $\braket{v_2}{x_0}\ne0$, $[v_2]$ contains $x_1$, so $3=|v_2|=|x_1|$.
Since $|v_{k_2}| > 3$, $[v_{k_2}]$ contains high weight elements other than $x_1$. Since $[v_2]$ contains $x_1$ and $v_{k_2 - 1}$ is connected by a path of norm-two vectors to $v_2$, the unique high weight element $x_{z_j}$ of $[v_j]$ is contained in $[v_{k_2}]$. This implies that $\braket{v_j}{v_{k_2}}$ must be nonzero, a contradiction. 

In the last two cases, $v_j$ has nonzero inner product with both $v_1$ and $v_2$, so $[v_j]$ abuts both $[v_1]$ and $[v_2]$. Since $[v_1]$ and $[v_2]$ abut $[v_{k_2}]$ at opposite ends, $[v_{k_2}]$ must be contained in the union of $[v_1], [v_2]$, and $[v_j]$. However, $\braket{v_j}{v_{k_2}} \le s$, so $|v_j| \le s + \delta([v_{k_2}],[v_j]) \le s+2$. This means that there are only two high weight elements in $[v_{k_2}]$, with one being $x_1$ and the other having norm at most $s+2$, so by Lemma~\ref{Lem:IntervalNorm}, $|v_{k_2}|\le s+3$. This contradicts the fact that $|v_{k_2}| = s+6$. 
\end{proof}

\begin{prop}\label{2k2>2}
If $v_{k_2}$ is tight, $v_{k_3 + 1} = e_1 + e_2 + \cdots + e_{k_2 - 1} + e_{k_3} - e_{k_3 + 1}$, $v_{k_3 + 2}$ is either $e_{k_3 + 1} - e_{k_3+2}$ or $e_{k_2 - 1} +  e_{k_3 + 1} - e_{k_3+2}$, and $|v_j| = 2$ for all $j > k_3 + 2$. (None of the vectors past $v_{k_3}$ are necessary to make the lattice C-type --- $n+1$ could be $k_3$ or anything larger.)

The corresponding changemaker vectors are
\begin{enumerate}
\item $(1,1,2^{[s]},2s+3,2s+5,4s+6^{[t]})$, $s\ge1,t\ge0$.
\item $(1,1,2^{[s]},2s+3,2s+5,4s+6,4s+8^{[t]})$, $s\ge1,t\ge1$.
\end{enumerate}
\end{prop}
\begin{proof}
Since $k_3 \in \supp(v_{k_3 + 1})$, $S_{k_3 + 1}$ is necessarily equal to $\{1,2,3,\dots,k_2 - 1, k_3\}$ by Lemma~\ref{lem:tightSj}, and so $v_{k_3 + 1} = e_1 + e_2 + \cdots + e_{k_2 - 1} + e_{k_3} - e_{k_3 + 1}$. 
For any other $j$ with $S_j = S_{k_3 + 1}$, we get that $\braket{v_j}{v_{k_3+1}} \ge k_2-1\ge2$, contradicting Corollary~\ref{unbreakablepairing}. Therefore, for $j > k_3 + 1$, $S_j$ is either $\emptyset$ or $\{k_2 - 1\}$. Suppose for some $j > k_3 + 1$ we have $S_j = \{k_2 - 1\}$. Then $\braket{v_j}{v_{k_2}} = 1$ while $v_j$ is orthogonal to both $x_0$ and $v_1$. Since $\braket{v_{k_2}}{v_1}=1$ and $\braket{x_0}{v_1}=0$, $[v_1]$ abuts the right endpoint of $[v_{k_2}]$. Hence $x_{z_j} \in [v_{k_2}]$.
By Lemma~\ref{intervalproduct}, we get that $|v_j| = 3$,  and $\epsilon_j = \epsilon_{k_2}$. Since also $\braket{v_{k_3 +1}}{v_{k_2}} = s+1$, $\epsilon_{k_3 + 1} = \epsilon_{k_2} = \epsilon_j$, so $\braket{v_j}{v_{k_3 + 1}}$ is either $-1$ or $0$ depending on whether their intervals abut. However, since $|v_j| = 3$, $v_j = e_{k_2 - 1} + e_{j-1} - e_j$, so $\braket{v_j}{v_{k_3 + 1}}$ is $1$ if $j > k_3 + 2$ and $0$ if $j = k_3 + 2$. Therefore, $j = k_3 + 2$ and $S_i = \emptyset$ for $i > k_3 + 2$. 
For any $i > k_3 + 2$, if $\min\supp(v_i)=k_3+1$, $v_i\sim v_{k_3+1}$. Since $v_{k_3+1}\sim v_1$, $\braket{v_{k_3+1}}{v_{k_2}}\ne0$ and $[v_1]$ abuts the right endpoint of $[v_{k_2}]$, $x_{z_{k_3+1}}$ is the rightmost high weight vertex in $[v_{k_2}]$ and $[v_1]$ abuts the right endpoint of  $[v_{k_3+1}]$. As $\braket{v_i}{v_{k_2}}=0$, $[v_i]$ must abut the right endpoint of $[v_{k_3+1}]$. We then conclude that $[v_1]$ and $[v_i]$ abut, which is impossible.
 So $\min\supp(v_i)>k_3+1$ when $i > k_3 + 2$.
Using Lemma~\ref{lem:AllNorm2}, we conclude that $|v_i| = 2$ for $i > k_3 + 2$. 
\end{proof}

\section{{\bf $k_1 = 1$, $k_2 = 2$}}\label{sec:k1k21}

In this section we consider the case where $k_1=1$ and $k_2=2$. Using Corollary~\ref{v1}, we get that
\begin{equation}\label{Eq:Case1X0}
x_0=e_0 + e_1 + e_{2} - e_{k_3}. 
\end{equation}
Also, we have that $v_1 = e_0 - e_1$. So 
\begin{equation}\label{Eq:Case1Sigma}
\sigma_0 = \sigma_1 = 1.
\end{equation} 
By Lemma~\ref{lem:tightvector}, the only possible tight vector is $v_2$.
In what follows we classify all the changemaker vectors whose orthogonal complements are isomorphic to C-type lattices with $x_0$ as given in~\eqref{Eq:Case1X0}. As in the previous section, we start by determining the first $k_3+1$ components of such changemaker vectors. It turns out that the initial segment of $\sigma$ depends on whether or not $v_2$ is tight. 

\begin{lemma}\label{prop1}
If $v_2$ is tight, the initial segment $(\sigma_0,\sigma_1,\cdots,\sigma_{k_3})$ of $\sigma$ is equal to $(1, 1, 3, 5)$.
\end{lemma}
\begin{proof}
By assumption, $v_2 = 2e_0 + e_1 - e_2$, so $\sigma_2 = 3$ and $|v_2| = 6$. This together with~\eqref{Eq:Case1X0}~and~\eqref{Eq:Case1Sigma}, yields $\sigma_{k_3} = 5$. We claim that $k_3 = k_2+1 = 3$. Suppose for contradiction that $k_3 \not = k_2 + 1$. Recall from Lemma~\ref{lem:tightvector} that $v_{k_2+1}$ cannot be tight. By combining this together with Lemma~\ref{Lem:X0Odd}, it can only be the case that $\sigma_{k_2+1}=4$ and $v_3 = e_1+e_2-e_3$. Note that $\braket{v_2}{x_0}=2$, $\braket{v_1}{x_0}=0$, and $\braket{v_1}{v_2}=1$. Therefore, $[v_1]$ abuts the right endpoint of $[v_2]$. Given that $[v_3]$ abuts both $x_0$ and $[v_1]$, it follows that the only high weight vertex of $[v_2]$ is that of $[v_3]$ (see Definition~\ref{Def:HighNorm} and Lemma~\ref{lem:TwoIndBr}). This implies that $|[v_2]|=|[v_3]|=3$ which is a contradiction. Hence $k_3 = 3$ and $v_3 = e_0+e_1+e_2 - e_3$. 
\end{proof}

\begin{lemma}\label{prop2}
If $v_2$ is not tight, the initial segment $(\sigma_0,\sigma_1,\cdots,\sigma_{k_3})$ of $\sigma$ is equal to either $(1, 1, 1, 3)$ or $(1, 1, 1, 2, 3)$.
\end{lemma}

\begin{proof}
When $v_2$ is not tight, using Lemma~\ref{Lem:X0Odd} together with the fact that $k_2 = 2$, we get that $v_2 = e_1 - e_2$, so $\sigma_2 = 1$. This together with~\eqref{Eq:Case1X0}~and~\eqref{Eq:Case1Sigma}, gives us that $\sigma_{k_3} = 3$. Either $k_3=3$ and we get the first possibility stated in the proposition, or $k_3>3$. In the latter case, using Lemmas~\ref{Lem:X0Odd}~and~\ref{lem:j-1}, we must have that $v_3=e_1+e_2-e_3$, so $\sigma_3=2$. We claim that, if $k_3>3$, then $k_3=4$. If $k_3\not = 4$, then we must have $v_4=e_3-e_4$. That will produce a claw on $(v_3, v_4, x_0, v_1)$. This gives the second stated possibility.  
\end{proof}

We use the notation of Equations~\eqref{Eq:Sj}~and~\eqref{Eq:S'j} in Section~\ref{sec:Case2}. Again, we use the basis $S'$, defined in~\eqref{Eq:S'}. Note that in this section, $v_{k_3} = x_0$. Moreover, if $k_3=3$, then $S_j=S'_j$.


\begin{prop}\label{lem:k1=1,k2=2,v2tight}
	If $v_2$ is tight, then one of the following is true: 
	\begin{enumerate}
		\item $\abs{v_3} = 4$, $v_4 = e_1 + e_3 - e_4$, and $\abs{v_j} = 2$ for all $5\leq j \leq 4+t$, $t\geq 0$. 
		\item $\abs{v_3} = 4$, $v_4 = e_1 + e_3 - e_4$, $v_5 = e_0 + e_1 +e_4-e_5$, and $\abs{v_j} = 2$ for all $6\leq j \leq 5 + t$, $t \geq 0$. 
	\end{enumerate}
	The corresponding changemaker vectors are:
	\begin{enumerate}
		\item $(1,1,3,5,6^{[t]})$
		\item $(1,1,3,5,6,8^{[t+1]})$
	\end{enumerate}
\end{prop}
\begin{proof}
When $v_2$ is tight, using Lemma~\ref{prop1}, the initial segment $(\sigma_0, \cdots, \sigma_{k_3})$ of $\sigma$ is $(1,1,3,5)$. For any $j > 3$, $S_j$ will be one of $\emptyset,\{ 1,2\}$, $\{ 2, 3\}$, $\{ 1,3\}$, $\{0,1\}$, or $\{ 0,1,2,3\}$ by Lemma~\ref{Lem:X0Odd} and Lemma~\ref{gappy3}. We will first show that $\{ 1,2\}$, $\{ 2, 3\}$ and $\{ 0,1,2,3\}$  do not occur. If $S_j = \{1,2\}$ for some $j > 4$, then $\braket{v_j}{v_1} = -1$, $\braket{v_j}{x_0} = 2$, and $\braket{v_j}{v_2} = 0$. Since $[x_0]$ and $[v_1]$ abut $[v_2]$ on opposite ends, and $[v_j]$ abuts both $[x_0]$ and $[v_1]$, the interval $[v_2]$ is contained in the union of $[x_0]$, $[v_j]$, and $[v_1]$. Therefore, $|[v_2 \cap v_j]| =|v_2|= 6$, so $|\braket{v_j}{v_2}| = 6 - \delta([v_j],[v_2]) \ge 3$, a contradiction. If $S_j = \{2,3\}$, then $\braket{v_j}{v_2} = -1$ but $\braket{v_j}{v_1} = \braket{v_j}{x_0} = 0$. To avoid a claw $(v_2,v_1,x_0,v_j)$, then, we must have $[v_2] \pitchfork [v_j]$. Since $v_j$ is orthogonal to $x_0$, this means that $\delta([v_2],[v_j]) = 2$, so $|v_j| = |[v_j\cap v_2]| = 3$ and $\epsilon_2 \neq \epsilon_j$. Therefore, $v_j + v_2$ is reducible. Since $j-1 \in \supp^+(v_j)$, the only way to have $|v_j| = 3$ is to have $j = 4$, but then $v_j + v_2$ is irreducible by Lemma~\ref{lem:SumIrr}. If $S_j = \{0,1,2,3\}$, $\braket{v_j}{v_2} = 2$, $\braket{v_j}{x_0} = 2$, and $\braket{v_j}{v_1} = 0$. Also, $|[v_2 \cap v_j]| = |v_j| \ge 5$, so in order to have $\braket{v_j}{v_2} = 2$ we must have $\epsilon_j = \epsilon_2$ and $\delta([v_2],[v_j]) = 3$. By Lemma~\ref{lem:delta}, $\braket{v_j}{x_0}=-\braket{v_2}{x_0}=\pm2$, a contradiction. Therefore, for each $j > 3$, $S_j$ is one of $\emptyset$, $\{0,1\}$ and $\{1,3\}$. Furthermore, if $S_j = \{0,1\}$, then $\braket{v_j}{x_0} \neq 0$, so by Corollary~\ref{x0pairing} there is at most one $j$ with $S_j = \{0,1\}$.

If the index $4$ exists, $3 \in S_4$, so $S_4 = \{1,3\}$, $v_4 = e_1 + e_3 - e_4$, and $\sigma_4 = 6$. If, for some $j > 4$, $S_j = \{1,3\}$, then also $4 \in \supp^+(v_j)$ by Corollary~\ref{unbreakablepairing}. Therefore, $|v_j| \ge 4$ and $\braket{v_j}{v_4} = 1$, so $[v_4]$ abuts $[v_j]$. Since $v_j$ is orthogonal to $x_0$, $\delta([v_2],[v_j])\le 2$, so since $|v_j| \ge 4$ and $\braket{v_j}{v_2} = 1$ we must have $[v_2]\dagger [v_j]$. Therefore, using Corollary~\ref{Cor:ZjDistinct}, either $[v_2]$ and $[v_4]$ are distant or they share a common end, but in either case we cannot have $\braket{v_2}{v_4} = 1$. Therefore, there is at most one $j > 4$ with $S_j = \{0,1\}$, and for all other $i$ we have $S_i = \emptyset$. 
Suppose that for some $j$ we have $S_j = \{0,1\}$. It follows from Lemma~\ref{lem:AllNorm2} that $|v_i|=2$ when $4<i<j$.
By Lemma~\ref{gappy3}, $v_j = e_0 + e_1 + e_k + e_{k+1} + \cdots + e_{j-1} - e_j$ for some $4 \le k < j$, and to avoid a claw $(v_j,v_1,x_0,v_k)$ we must have $k = 4$. Therefore, $|v_j| = j-1 \ge 4$. Since $\braket{v_j}{v_2} = 3$, we must have $\epsilon_j = \epsilon_2$, and since $\braket{v_j}{x_0}=\braket{v_2}{x_0}=2$ this means that $\delta([v_2],[v_j]) = 1$. Therefore, $|v_j| = \braket{v_j}{v_2} + 1 = 4$, so $j = 5$. This means that $S_5$ is either $\emptyset$ or $\{0,1\}$, and $S_i = \emptyset$ for $i > 5$. 

If $S_5=\emptyset$, by  Lemma~\ref{lem:AllNorm2}, $|v_i|=2$ when $i \ge 5$. If $S_5=\{0,1\}$, we will show that $\min \supp v_i \ge5$ when $i>5$.

We first claim that $x_{z_4}\in[v_2]$. Otherwise, as $\braket{v_4}{v_2}=1$, we get $[v_2]\dagger[v_4]$ and $\epsilon_2=-\epsilon_4$. We also have $\braket{v_2}{v_1}=-\braket{v_4}{v_1}=1$. Thus we have either $[v_1]\prec[v_2]$ or $[v_1]\prec[v_4]$. If $[v_1]\prec[v_2]$, then $\epsilon_1=\epsilon_2$ and $\epsilon_1=\epsilon_4$, a contradiction to $\epsilon_2=-\epsilon_4$. Similarly, we can rule out $[v_1]\prec[v_4]$. This proves the claim.

Note that $\sigma_0=\sigma_1$ are the only two $1$'s in the coordinates of $\sigma$, so there does not exist any norm $2$ vector $y\in (\sigma)^{\perp}$ such that $\braket{y}{v_1}=-1$. Thus $[v_1]$ contains only one vertex which does not neighbor any norm $2$ vertex. 
Since $v_1\sim v_2$ and $\braket{v_1}{x_0}=0$, $[v_1]$ abuts the right end of $[v_2]$. 
As $x_{z_4}\in[v_2]$ and $v_4\sim v_1$, $x_{z_4}$ is the rightmost high weight vertex in $[v_2]$.
If $\min \supp v_i=4$ for some $i>5$, then $v_i\sim v_4$ and $|v_i|\ge3$. 
As $\braket{v_i}{v_2}=0$, $x_{z_i}$ is the leftmost high weight vertex to the right of $[v_2]$. So $[v_1]$ is the unique vertex between $x_{z_4}$ and $x_{z_i}$. We then see that $[v_1]$ and $[v_i]$ abut, which is not possible as $\braket{v_1}{v_i}=0$.
This proves that  $\min \supp v_i \ge5$ when $i>5$.
By  Lemma~\ref{lem:AllNorm2}, $|v_i|=2$ when $i > 5$.
\end{proof}

\begin{prop}\label{prop:k1=1,k2=2,v2justright1}
	If $v_2$ is not tight and $(\sigma_0,\dots,\sigma_{k_3})\not = (1,1,1,2,3)$, then one of the following is true (if only the norm of a standard basis vector is given, it is just right):
	\begin{enumerate}
		\item $\abs{v_3} = 4$, $\abs{v_4} = 3$, $\abs{v_j} = 2$ for $5\leq j \leq 4+t$, $v_{5+t} = e_1 + e_2 + e_4 + e_5 + \dots + e_{4+t} - e_{5+t}$, $\abs{v_{6+t}} = 3$, and $|v_j| = 2$ for $j > 6+t$ ($t \ge 0$).
		\item $\abs{v_3} = 4$, $\abs{v_4} = 3$, and $\abs{v_5} = 6$.
		\item $\abs{v_3} = 4$, $\abs{v_4} = 5$, and $\abs{v_5} = 4$.
			\end{enumerate}
	with corresponding changemaker vectors:
	\begin{enumerate}
		\item $(1,1,1,3,4,4^{[t]},4t+6, (4t+10)^{[s]})$, $s,t\ge0$
		\item $(1,1,1,3,4,10)$	
		\item $(1,1,1,3,6,10)$
			\end{enumerate}
\end{prop}
\begin{proof}
If $v_2$ is not tight and $(\sigma_0,\dots,\sigma_{k_3})\not = (1,1,1,2,3)$, using Lemma~\ref{prop2}, it follows that $(\sigma_0, \cdots, \sigma_{k_3})$ is $(1,1,1,3)$. Note that, using Lemmas~\ref{Lem:X0Odd} and \ref{gappy3}, 
\begin{equation}\label{eq:Siwheni>4}
S_i=\emptyset, \{1,2\}, \{2,3\}, \text{ or } \{0,1,2,3\},\quad\text{when }i\ge4.
\end{equation}
Using Lemma~\ref{lem:j-1}, we get that $S_4$ is either $\{ 2,3\}$ or $\{ 0,1,2,3\}$, that is, $\sigma_4$ is either $4$ or $6$. 

When $\sigma_4=6$, $v_4=e_0 + e_1 + e_2 + e_3 - e_4$.  Since $\braket{v_4}{x_0}=2$, using Corollary~\ref{x0pairing} and (\ref{eq:Siwheni>4}), 
\begin{equation}\label{eq:NewSiwheni>4}
S_i=\emptyset \text{ or }\{2,3\} \quad\text{ when }i>4.
\end{equation}
Since the intersection graph must be connected, there will be some index $j$ for which $S_j=\{2, 3\}$. 
Additionally, using Corollary~\ref{unbreakablepairing}, we get that $4\in \supp^+ v_j$, as otherwise $\braket{v_j}{v_4}=2$. It turns out that there is only one such $j$. In fact, if there were two such indices $j_1,j_2$, then $\{2,3,4\}\subset S_{j_1}\cap S_{j_2}$, we would have $\braket{v_{j_1}}{v_{j_2}}\ge2$, a contradiction.
We claim that $j=5$. If $j\not = 5$, then $S_5 = \emptyset$ by (\ref{eq:NewSiwheni>4}). Therefore, $|v_5|=2$, so, by Lemma~\ref{gappy3}, $4$ cannot be a gappy index for $v_j$. This will give us a claw $(v_4, x_0, v_5, v_j)$. This justifies the claim; in particular, $\sigma_5=10$. If the index $6$ existed, by (\ref{eq:NewSiwheni>4}) we must have $S_6=\emptyset$. Thus, $v_6$ is either $e_4 + e_5 - e_6$ or $e_5 - e_6$. In the first case, there will be a claw $(v_4, v_5, v_6, x_0)$ and in the second case there will be a claw $(v_5, v_4, v_6, v_2)$. So the index $6$ does not exist, and we get the third possibility listed in the proposition.

Now, suppose that $\sigma_4=4$. If $\sigma_5\not = 4, 6$, by Lemma~\ref{lem:j-1} and (\ref{eq:Siwheni>4}), $S_5$ is either $\{ 0, 1, 2, 3\}$ or $\{ 2,3\}$. If $S_i=\{ 0, 1, 2, 3\}$ or $\{ 2,3\}$ for some $i>5$, we will get a heavy triple $(v_4, v_5, v_i)$.
So $S_i=\emptyset$ or $\{1,2\}$ when $i>5$.

If $S_5=\{ 2,3\}$, then $e_5=e_2+e_3+e_4-e_5$. Since the pairing graph is connected, there exists an index $i>5$ such that
$S_i=\{1,2\}$. Using the path $v_i\sim v_1\sim v_2$, we will get a heavy triple $(v_4, v_5, v_i)$. 

If $S_5=\{0,1,2,3\}$, $\sigma_5=10$. If the index $6$ does exist, using Corollary~\ref{x0pairing}, $S_6=\emptyset$. We will have a claw $(v_4,v_2,v_5,v_6)$ or $(v_5,x_0,v_4,v_6)$, depending on whether or not $4\in \supp^+(v_6)$. So we get the second possibility listed in the proposition.

If $\sigma_5=6$, since $\braket{v_5}{x_0}=2$, by Corollary~\ref{x0pairing} and (\ref{eq:Siwheni>4}) we have $S_i=\emptyset$ or $\{2,3\}$ when $i>5$.
Assume that there exists $i>5$ such that $S_i=\{2,3\}$. Since $\braket{v_i}{v_4}\le1$, $4\in\supp(v_i)$. Since $\braket{v_i}{v_5}\le1$, $5\in\supp(v_i)$. We would then have a heavy triple $(v_4,v_5,v_i)$. So $S_i=\emptyset$ whenever $i>5$.
If $|v_6|=2$, there will be a claw $(v_5,v_1,x_0,v_6)$. So $v_6=e_4+e_5-e_6$.
Since $\braket{v_5}{x_0}=2$, $x_{z_5}=x_1$. Since $v_5$ is connected to $v_4$ by a path of norm 2 vectors, $x_{z_4}$ is the leftmost high weight vertex to the right of $x_{z_5}$. Since $v_4\sim v_6$, by Corollary~\ref{Cor:ZjDistinct}, 
$\braket{v_i}{v_5}=\braket{v_i}{v_4}=0$, whenever $i>6$.
We  then conclude that $\min\supp(v_i)\ge6$ when $i>6$. Using Lemma~\ref{lem:AllNorm2}, we get $|v_i|=2$ when $i>6$. This gives us the case $t=0$ in the first possibility listed in the proposition.

If $\sigma_5=4$, 
since the pairing graph is connected, there must be a unique index $j>4$ for which $\braket{v_j}{x_0}=2$. Then $\sigma_j>4$, and $S_j$ is either $\{ 0,1,2,3\}$ or $\{ 1,2\}$ by (\ref{eq:Siwheni>4}). Let $t+5$ be the index such that $\sigma_{t+4}=4<\sigma_{t+5}$.

If $S_j=\{ 0,1,2,3\}$, then in order to avoid $\braket{v_j}{v_4}=2$  (which contradicts Corollary~\ref{unbreakablepairing}) we must have $4\in \supp^+(v_j)$. Moreover, using Lemma~\ref{gappy3}, neither of $4,5,\dots, t+3$ can be a gappy index for $v_j$. Hence we get a claw $(v_4, v_2, v_j, v_5)$ as $j>t+4\ge5$. That is, we must have $S_j=\{ 1,2 \}$. 

We claim that $j=t+5$. Suppose for contradiction that $j \ne t+5$. Then, using Corollary~\ref{x0pairing} and (\ref{eq:Siwheni>4}), $S_{t+5}$ is either $\emptyset$ or $\{ 2, 3\}$. If $S_{t+5} = \{ 2,3\}$, then there will be a heavy triple $(v_4, v_{t+5}, v_j)$, where the paths connecting the three high norm vertices are through $v_1$ and/or $v_2$. If $S_{t+5}=\emptyset$, set $i=\min\supp(v_{t+5})$. Using Lemma~\ref{gappy3}, none of $4, \cdots, t+3$ can be a gappy index for $v_{t+5}$. Then there will be a claw on either $(v_i, v_{i-1}, v_{t+5}, v_{i+1})$ or $(v_4, v_2, v_{t+5}, v_5)$, depending on whether $4<i<t+4$ or $i=4$. (Note that $i\not = t+4$ since $\sigma_{t+5}>4$.) This finishes the proof of the claim, that is, $j=t+5$ and $S_{t+5}=\{ 1,2\}$. 

To avoid a cycle $v_{t+5}\sim v_4\sim v_2\sim v_1\sim v_{t+5}$ of length bigger than $3$ (which violates Corollary~\ref{cycles}), we must have $4\in \supp^+(v_{t+5})$. Furthermore, using Lemmas~\ref{gappy3}~and~\ref{lem:j-1}, all the indices $5, \cdots, t+4 \in \supp(v_{t+5})$, so $\sigma_{t+5}=4t+6$. For $i>t+5$, using Corollary~\ref{x0pairing} and (\ref{eq:Siwheni>4}), the set $S_i$ is either $\emptyset$ or $\{2, 3\}$. If $S_i=\{2,3\}$, we will get a heavy triple $(v_i, v_4, v_{t+5})$. This proves that $S_i=\emptyset$ whenever $i>t+5$. 
Set $\ell=\min \supp(v_{i})$. If $\ell=t+5$, there will be a claw $(v_{t+5}, x_0, v_{i}, v_1)$. If $4<\ell<t+4$, there will be a claw $(v_{\ell}, v_{\ell-1}, v_{i}, v_{\ell+1})$, and if $\ell=4$ the claw will be on $v_4, v_2, v_{i}, v_5$. Therefore $\ell=t+4$ or $\ell\ge t+6$. In particular, $e_{t+6}=e_{t+4}+e_{t+5}-e_{t+6}$ and $\sigma_{t+6}= 4t+10$. 
When $i>t+6$, if $\ell=t+4$, we get a heavy triple $(v_i,v_4,v_{t+6})$. So $\ell\ge t+6$ when $i>t+6$. Now we can apply Lemma~\ref{lem:AllNorm2} to conclude that $|v_i|=2$ whenever $i>t+6$, and we will get the first possibility listed in the proposition.  
\end{proof}

\begin{prop}\label{prop:k1=1,k2=2,v2justright2}
If $(\sigma_0,\dots,\sigma_{k_3}) = (1,1,1,2,3)$, $v_5 = e_2 + e_3 + e_4 - e_5$, and $|v_j| = 2$ for $j > 5$. In this case, $\sigma = (1,1,1,2,3,6^{[t]})$, $t\ge1$.
\end{prop}
\begin{proof}
Since $4\in S'_5$, $S'_5 = \{ 2, 4\}$ by Lemma~\ref{Lem:X0Odd} and Corollary~\ref{x0pairing}, so the set $S_5$ is equal to either $\{ 2, 4\}$ or $\{ 2, 3, 4\}$. If $S_5=\{ 2, 4\}$, then there will be a cycle of length $4$ on $(v_3, v_1, v_2, v_5)$. Therefore, $S_5 = \{ 2, 3, 4\}$, and so, $\sigma_5=6$. There is a path $v_3\sim v_1\sim v_2\sim v_5$. 
For any $i>5$, to avoid a heavy triple $(v_i,v_3,v_5)$, $v_i$ cannot neighbor $v_1$ or $v_2$. Combined with Lemmas~\ref{gappy3} and~\ref{Lem:X0Odd} and Corollary~\ref{x0pairing}, we must have $S'_i=\emptyset$. If $3\in S_i$, we would have a claw $(v_3,v_i,v_1,x_0)$. So $S_i=\emptyset$. By Lemma~\ref{lem:AllNorm2}, we have $|v_i|=2$ whenever $i>5$.

Now  $\sigma = (1,1,1,2,3,6^{[t]})$, $t\ge0$. If $t=0$, then $p=1$, (see Section~\ref{pq}.) So we must have $t\ge1$.
\end{proof}


\section{$k_1 > 1$}\label{sec:Case1}
In the present section we classify all the changemaker C-type lattices that have 
\[
x_0=e_0 \pm e_{k_1} \pm e_{k_2} \pm e_{k_3},
\]
where $k_1>1$. Using Lemma~\ref{Lem:X0Odd}, we know that
\begin{equation}\label{eq:v1}
v_1=2e_0 - e_1,
\end{equation}
and therefore, $\sigma_1 =2$ and $|v_1|=5$. We remind the reader that, by Lemma~\ref{lem:tightvector}, $v_1$ is the only tight vector in the C-type lattices that concern us in this section. We also note that 
\begin{equation}\label{eq:0inVk1}
0\in\supp(v_{k_1})
\end{equation}
 by Lemma~\ref{Lem:X0Odd}.
Compared to Sections~\ref{sec:Case2}~and~\ref{sec:k1k21}, it will take longer to determine the initial segment $(\sigma_0, \cdots, \sigma_{k_3})$ of $\sigma$. We start by specifying the positive integer $k_1$.

\begin{lemma}\label{lem:k1>1,v2+t}
		The segment $(\sigma_0, \cdots \sigma_{k_1})$ is either $(1,2,3)$ or $(1,2,2,3)$. In particular, $k_1=2$ or $3$, and $\sigma_{k_1}=3$. 
\end{lemma}

\begin{proof}
Using Lemma~\ref{lem:tightvector}, we get that $v_2$ is either $e_0+e_1-e_2$ or $e_1-e_2$. In the former case, using Lemma~\ref{Lem:X0Odd}, we get that $k_1=2$, and so $\sigma_{k_1}=3$. 

Now suppose that $v_2=e_1-e_2$. More generally, suppose that there exists $t\ge 1$ such that $(\sigma_0, \sigma_1, \cdots, \sigma_{t+1})=(1,2,2^{[t]})$, and that $|v_{t+2}|>2$. We will show that $t=1$, $k_1=3$, and that $\sigma_{t+2}$ (or simply $\sigma_3$) is $3$. 

Set $j=\min\supp(v_{t+2})$. We argue that $j=0$. (Note that, by Lemma~\ref{gappy3}, none of $1,2, \cdots, t$ is a gappy index for $v_{t+2}$.) If $1<j<t+1$, there will be a claw on $v_j, v_{j-1}, v_{t+2}, v_{j+1}$. If $j=1$, then $\braket{v_{t+2}}{v_1}=-1$ and $v_{t+2}$ will be orthogonal to $v_2$. Then $k_1>t+2$ by \eqref{eq:0inVk1}.
There will be a claw on $v_1, x_0, v_{t+2}, v_2$, unless $[v_{t+2}]\pitchfork [v_1]$, $|[v_1\cap v_{t+2}]|=|v_{t+2}|=3$, and $\epsilon_{t+2}=-\epsilon_1$. Thus $v_1+v_{t+2}$ is the sum of two distant intervals and so is reducible. Since $|v_{t+2}|=3$, $v_{t+2}=e_1+e_{t+1}-e_{t+2}$, and so $v_1+v_{t+2}$ is irreducible by Lemma~\ref{lem:SumIrr}, a contradiction. That is, $j=0$, and that, 
\begin{equation}\label{vt+2}
v_{t+2}=e_0+e_{i}+e_{i+1} + \cdots + e_{t+1}-e_{t+2},
\end{equation}
with $i\ge 1$. 

Since $0\in \supp^+(v_{t+2})$, using Lemma~\ref{Lem:X0Odd}, we get that $k_1=t+2$. Furthermore, we claim that 
$x_0=e_0+e_{t+2}+e_{k_2}-e_{k_3}$. See Proposition~\ref{x0}. If $x_0=e_0-e_{t+2}-e_{k_2}+e_{k_3}$, then $\braket{v_{t+2}}{x_0}=2$. 
Observe that $\braket{v_{t+2}}{v_1}=1$ or $2$ depending on whether or not $i=1$ in~\eqref{vt+2}; in particular, $\braket{v_{t+2}}{v_1}>0$. Since $|v_{t+2}\cap v_1|=|v_{t+2}|\ge 3$ and $\delta([v_1],[v_{t+2}])\le 3$, using Lemma~\ref{intervalproduct}, it must be that $\epsilon_1 = \epsilon_{t+2}$. Since $\braket{v_1}{x_0}=\braket{v_{t+2}}{x_0}=2$, $[v_1]$ and $[v_{t+2}]$ share their left endpoint, and $\delta([v_{t+2}], [v_1])=1$. Moreover, we must have $|v_{t+2}|=3$ (as otherwise $\braket{v_{t+2}}{v_1}>2$). That is, $v_{t+2}=e_0+e_{t+1}-e_{t+2}$. We have $\braket{v_2}{v_1}=-1$ and $\braket{v_2}{x_0}=0$, so $[v_2]$ abuts the right end of $[v_1]$.
Since also $v_{t+2}\sim v_{t+1}$, $|v_i|=2$ for $i\in \{ 2, \cdots, t+1\}$,
\[
v_2\sim v_3\sim\cdots\sim v_{t+1},
\]
 the interval $[v_1]$ is a subset of the union of the $[v_j]$ for $j\in \{ 2, \cdots, t+2\}$, which in turn implies that $|v_1|=|v_{t+2}|=3$, a contradiction. This shows that 
\[
x_0=e_0+e_{t+2}+e_{k_2}-e_{k_3}.
\]

We now argue that $1\not \in \supp(v_{t+2})$. Suppose for contradiction that $1\in \supp(v_{t+2})$. Using (\ref{vt+2}), we get
 that $|v_{t+2}|\ge 4$, $\braket{v_1}{v_{t+2}}=1$ and $\braket{v_2}{v_{t+2}}=0$. To avoid a claw on $v_1, x_0, v_{t+2}, v_2$, we must have $[v_{t+2}]\pitchfork [v_1]$. This implies that $\delta([v_1], [v_{t+2}])=2$. Using Lemma~\ref{intervalproduct} and that $|v_{t+2}|\ge 4$, we see that $|\braket{v_1}{v_{t+2}}|\ge 2$, a contradiction. That is, in~\eqref{vt+2}, we must have $i>1$. 

We claim that $i=2$. If $2<i<t+1$, there will be a claw on $v_i, v_{i-1}, v_{t+2}, v_{i+1}$. If $i=t+1$ (and $i>2$), to avoid a claw on $v_1, x_0, v_{t+2}, v_2$, it must be that $[v_{t+2}]\pitchfork [v_1]$, and so $\delta([v_{t+2}], [v_1])=2$. To get $\braket{v_{t+2}}{v_1}=2$, however, it must be $|v_{t+2}|=4$ which contradicts $i=t+1$. Therefore, in~\eqref{vt+2}, we have $i=2$. In particular, $v_2\sim v_{t+2}$.

Finally, we argue that $t=1$. If $t>1$, we must have $v_1\sim v_{t+2}$ as otherwise we get a claw $(v_2, v_1, v_{t+2}, v_3)$. That is, $[v_{t+2}]$ abuts $[v_1]$. Therefore, to fulfill $\braket{v_{t+2}}{v_1}=2$, $[v_{t+2}]\prec[v_1]$, and that $|v_{t+2}|=3$, which contradicts $t>1$ and (\ref{vt+2}). So $t=1$ as desired.
\end{proof}

As part of the proof of Lemma~\ref{lem:k1>1,v2+t}, we showed that $x_0 = e_0+e_{k_1}+e_{k_2}-e_{k_3}$ when $k_1=3$. Indeed, this is the case also when $k_1=2$.

\begin{lemma}\label{lem:k1>1, chi=2}
		Let $k_1>1$. Then $x_0 = e_0+e_{k_1}+e_{k_2}-e_{k_3}$.
\end{lemma}

\begin{proof}
We only need to show this for $k_1=2$. Suppose for contradiction $x_0 = e_0-e_{2}-e_{k_2}+e_{k_3}$ (see Proposition~\ref{x0}). Note that $v_2=e_0+e_1-e_2$, and therefore, $\braket{v_2}{x_0}=2=\braket{v_1}{x_0}$, and $\braket{v_2}{v_1}=1$. Since $|v_2|=3$, using Lemma~\ref{intervalproduct}, we see that $\epsilon_1=\epsilon_2$ and $\delta([v_1],[v_2])=2$. Since $\braket{[v_2]}{x_0}=\braket{[v_1]}{x_0}=\pm2$, $[v_1],[v_2]$ share their left end point, so we cannot have $\delta([v_1],[v_2])=2$, a contradiction.
\end{proof}

Now we proceed to determine the changemaker vectors. As in Section~\ref{sec:Case2}, we use the notation of~\eqref{Eq:Sj}~and~\eqref{Eq:S'j}. Also, we use the basis $S'$, defined in~\eqref{Eq:S'}, where $v_{k_3}$ is replaced by $x_0$.

\subsection{{\boldmath $k_1=2$}}\label{k1=2}

This subsection is devoted to classifying the changemaker C-type lattices with 
\begin{equation}\label{x0k1=2}
x_0=e_0+e_2+e_{k_2} - e_{k_3}.
\end{equation} 
Recall that the changemaker starts with $(1,2,3)$. We have
\begin{equation}\label{eq:v2parings}
\braket{v_1}{v_2}=1,\quad \braket{v_2}{x_0}=0.
\end{equation}

\begin{lemma}\label{lem:k1>1 intervals}
		The intervals $[v_2]$ and $[v_1]$ are consecutive with $\epsilon_2 = -\epsilon_1$. 
\end{lemma}

\begin{proof}
Using (\ref{eq:v2parings}) and Lemma~\ref{intervalproduct}, either $[v_2]\pitchfork [v_1]$, $|[v_2]\cap[v_1]|=|[v_2]|=3$, $\delta([v_2], [v_1])=2$, and $\epsilon_2=\epsilon_1$, or $[v_2]\dagger [v_1]$, and $\epsilon_2 = -\epsilon_1$. In the former case, $v_2-v_1$ is the sum of two distant intervals, and so is reducible. However, we have $v_2=e_0+e_1-e_2$, and so $v_2-v_1$ is irreducible by Lemma~\ref{lem:SumIrr}~(2).
\end{proof}

\begin{lemma}\label{lem:NoSingle1}
There does not exist an index $j>3$, $j\ne k_3$, such that $\supp(v_j)\cap\{0,1,2\}=\{1\}$.
\end{lemma}
\begin{proof}
Otherwise, we will have $\braket{v_{j}}{v_1}=-\braket{v_{j}}{v_2}=-1$. We also have $\braket{v_{j}}{x_0}=0$ by Lemma~\ref{Lem:X0Odd}. By Lemma~\ref{lem:k1>1 intervals}, $[v_{j}]$ and $[v_1]$ share their right endpoint, so $\delta([v_j],[v_1])=1$. By  Lemma~\ref{intervalproduct}, $|\braket{v_1}{v_j}|=|v_{j}|-1>1$, a contradiction. 
\end{proof}

\begin{lemma}\label{lem:k1=2 sigma3}
		$\sigma_3 \in \{3,4\}$. Furthermore, if $\sigma_3 = 4$ then $[v_3]$ and $[v_1]$ share their left endpoint, and that $\epsilon_3 = \epsilon_1$.  
\end{lemma}

\begin{proof}
All the possibilities for $\sigma_3$ lie in $\{ 3,4,5,6\}$. If $\sigma_3=5$, we get that $v_3=e_1+e_2-e_3$. So $\braket{v_3}{v_1}=-1$ and $v_3$ is orthogonal to $v_2$. By Lemma~\ref{Lem:X0Odd}, $k_2=3$ and $\braket{v_3}{x_0}=0$.
Using Lemma~\ref{lem:k1>1 intervals}, we know that $[v_2]$ abuts $[v_1]$, and therefore, there will be a claw on $v_1, x_0, v_3, v_2$, unless $[v_3]\pitchfork [v_1]$, $|[v_1]\cap [v_3]|=|v_3|=3$, and $\epsilon_3=-\epsilon_1$. Thus $v_1+v_3$ is the sum of two distant intervals and so is reducible. However, $v_3+v_1$ is irreducible by Lemma~\ref{lem:SumIrr}, a contradiction. If $\sigma_3=6$, we see that $v_3=e_0+e_1+e_2-e_3$ (and, in particular, $|v_3|=4$). This implies that $\braket{v_3}{x_0}=2$ and $\braket{v_1}{v_3}=1$. The latter will only be possible if both $\delta([v_1], [v_3])=3$ and $\epsilon_1 = \epsilon_3$, a contradiction to Lemma~\ref{lem:delta}. 

If $\sigma_3=4$, we have $v_3=e_0+e_2-e_3$. Using Lemma~\ref{intervalproduct}, the second statement of the lemma is immediate because $\braket{v_3}{v_1}=\braket{v_3}{x_0}=\braket{v_1}{x_0}=2$ and $|v_3|=3$.
\end{proof}

\begin{lemma}\label{lem:0No2}
If $0\in\supp(v_j)$ and $2\notin\supp(v_j)$ for some $j>3$ and $j\ne k_3$, then $[v_j],[v_1]$ share their right endpoint, and $v_j=e_0+e_{j-1}-e_j$. Moreover, there exists at most one such $j$.
\end{lemma}
\begin{proof}
We have $1\not\in \supp(v_j)$, otherwise $\braket{v_2}{v_j}=2$, a contradiction to Corollary~\ref{unbreakablepairing}. So $\braket{v_1}{v_j}=2$. Since $\braket{v_{j}}{v_2}=1$, $[v_j]$ and $[v_2]$ are consecutive by Corollary~\ref{Cor:ZjDistinct}. It follows from Lemma~\ref{lem:k1>1 intervals} that $[v_{j}]$ and $[v_1]$ share their right endpoint, and so $\delta([v_{j}], [v_1])=1$. Then, to get $\braket{v_1}{v_{j}}=2$, we must have $|v_{j}|=3$ and $v_{j}=e_0+e_{j-1}-e_{j}$. Lastly, there exists at most one such $j$ by Corollary~\ref{Cor:ZjDistinct}.
\end{proof}

\begin{prop}\label{prop6.1}
If $\sigma_3=3$, the initial segment $(\sigma_0, \cdots, \sigma_{k_3})$ of $\sigma$ is $(1,2,3,3,7)$.
\end{prop}

\begin{proof}
Suppose that $\sigma_3=3$ (see Lemma~\ref{lem:k1=2 sigma3}). This implies that $k_2=3$ (Lemma~\ref{Lem:X0Odd}). Using Equation~\eqref{x0k1=2}, we see that $\sigma_{k_3}=7$. We claim that $k_3=k_2+1=4$. If $k_3\neq 4$, by Lemma~\ref{Lem:X0Odd}, $\sigma_4\in \{ 4, 6\}$. Suppose $\sigma_4=4$, or equivalently, $v_4=e_0+e_3-e_4$. This gives us that $\braket{v_4}{x_0}=2$ and $\braket{v_4}{v_2}=1$. By Lemma~\ref{lem:k1>1 intervals}, the interval $[v_1]$ will be a subset of $[v_4]\cup\{x_0\}$, which implies that $|v_1|= 3$, a contradiction. Suppose $\sigma_4=6$, or equivalently, $v_4=e_2+e_3-e_4$. Then there will be a claw $(v_2, v_1, v_4, v_3)$. This justifies the claim, that is, $\sigma_4=7$ and $k_3=4$. 
\end{proof}

\begin{prop}\label{lem:k1>1,2t+1}
If $\sigma_3=4$, the initial segment $(\sigma_0, \cdots, \sigma_{k_3})$ of $\sigma$ is either $(1,2,3,4,5,9)$ or $(1,2,3,4^{[s]}, 4s+3,4s+7)$, $s\ge1$.
\end{prop}

\begin{proof}
All the possibilities for $\sigma_4$ lie in $\{ 4,5,6,7,8,9,10\}$. We first argue that $\sigma_4\not\in\{6, 8, 9, 10\}$. Suppose $\sigma_4=6$, then $v_4=e_1+e_3-e_4$, contradicting Lemma~\ref{lem:NoSingle1}. If $\sigma_4=10$, then $v_4$ will have nonzero inner product with $v_2$ and $v_3$.  Using Lemmas~\ref{lem:k1=2 sigma3}~and~\ref{lem:k1>1 intervals}, the interval $[v_1]$ equals the union of $[v_3]$ and $[v_4]$, that is, $|v_1|=6$, a contradiction. If $\sigma_4=8$, then both the unbreakable vectors $v_3$ and $v_4$ will have nonzero inner product with $x_0$, contradicting Corollary~\ref{x0pairing}. When $\sigma_4=9$, $v_4=e_1+e_2+e_3-e_4$. Notice that $\braket{v_4}{v_1}=-1$ while $v_4$ is orthogonal to $x_0$. The latter gives us that $\delta([v_4], [v_1])\le 2$. Therefore, given that $|v_4|=4$, we must have $[v_4]$ and $[v_2]$ share their left endpoint by Lemma~\ref{lem:k1>1 intervals}, a contradiction to Corollary~\ref{Cor:ZjDistinct}. Therefore $\sigma_4\in \{4, 5, 7\}$.

Suppose that $\sigma_4=5$, that is, $v_4=e_0+e_3-e_4$. Using Lemma~\ref{Lem:X0Odd}, $k_2=4$, and so $\sigma_{k_3}=9$ by Equation~\eqref{x0k1=2}. Since $\braket{v_3}{x_0}=2$, $\braket{v_5}{x_0}=0$ by Corollary~\ref{x0pairing}, unless $k_3=5$. Since $4\in \supp(v_5)$, we get that $k_3=5$.

Let $s\ge 1$ be the integer satisfying that $\sigma_3=\cdots =\sigma_{s+2}=4$, and that $\sigma_{s+3}>4$. By Lemma~\ref{Lem:X0Odd}, $k_2\ge s+3$. Set $j=\text{min }\supp(v_{s+3})<s+2$. If $3<j<s+2$, there will be a claw $(v_j, v_{j-1}, v_{s+3}, v_{j+1})$, and if $j=3$, the claw will be $(v_3, x_0, v_{s+3}, v_4)$. If $j=1$, then $2\in\supp(v_{s+3})$ by Lemma~\ref{lem:NoSingle1}. Thus $|v_{s+3}|\ge 4$. Since $\braket{v_{s+3}}{x_0}=0$, $\delta([v_{s+3}], [v_1])\le 2$. Then 
\[
|\braket{v_{s+3}}{v_1}|\ge4-2\ge 2,
\] 
a contradiction. Lastly, suppose $j=0$. By Corollary~\ref{x0pairing}, $\braket{v_{s+3}}{x_0}=0$, it must be the case that $2\not \in\supp(v_{s+3})$. By Lemma~\ref{lem:0No2}, $v_{s+3}=e_0+e_{s+2}-e_{s+3}$. If $s=1$ and $\sigma_4=5$, this case was discussed in the previous paragraph. However, if $s>1$, then $\braket{v_{s+3}}{v_3}=1$, and so $[v_1]$ will be the union of $[v_3]$ and $[v_{s+3}]$ by Lemma~\ref{lem:k1=2 sigma3} and Lemma~\ref{lem:0No2}. Since $|v_3|=3$, to get $|v_1|=5$, it must be that $|v_{s+3}|=4$, a contradiction. So we are left with the case $j=2$. 

Note that $\braket{v_{s+3}}{v_2}=-1$, and $v_{s+3}$ is orthogonal to $v_1$ and $x_0$, so $[v_{s+3}]$ is distant from $[v_1]$ by Lemma~\ref{lem:k1>1 intervals}. Using Lemma~\ref{lem:k1=2 sigma3}, we get that $v_{s+3}$ is orthogonal to $v_{3}$, and so $3\in \supp(v_{s+3})$. By Lemma~\ref{gappy3}, we get that $4,\cdots, s+1 \in \supp(v_{s+3})$. That is, $\sigma_{s+3}=4s+3$, and that $k_2=s+3$. Using Equation~\eqref{x0k1=2}, we get that $\sigma_{k_3}=4s+7$. With the same argument as in the case $\sigma_4=5$, we get that $k_3=k_2+1=s+4$. This recovers the case $\sigma_4=7$ when $s=1$. 
\end{proof}

\begin{prop}\label{1k2=2}
If $(\sigma_0, \cdots, \sigma_{k_3})=(1,2,3,4,5,9)$, then $n+1=k_3$ (i.e. $v_{k_3}$ is the last standard basis vector).
\end{prop}

\begin{proof}
We claim that the index $6$ does not exist. Suppose for contradiction that it exists. Since $5\in S'_6$, then $S'_6$ must be one of $\{4, 5\}$, $\{2, 5\}$, or $\{0, 5\}$ (Lemma~\ref{Lem:X0Odd} and Corollary~\ref{x0pairing}). 

By Lemma~\ref{lem:0No2}, the intervals $[v_4]$ and $[v_1]$ share their right endpoint, and $S'_6\ne\{ 0,5\}$. 

Suppose that $S'_6=\{ 4, 5\}$ or $\{ 2, 5\}$, then $\braket{v_6}{x_0}=0$. We have that one of $\braket{v_6}{v_4}$ and $\braket{v_6}{v_3}$ is zero and the other one is nonzero, depending on whether or not $3\in S_6$. 
By Lemma~\ref{lem:k1>1 intervals} and Corollary~\ref{Cor:ZjDistinct}, $[v_6]$ and $[v_1]$ are not consecutive. 
Using Lemma~\ref{lem:k1=2 sigma3} and the fact that $[v_4]$ and $[v_1]$ share their right endpoint, we conclude that $[v_6]\subset[v_1]$ and $\delta([v_6], [v_1])\le 2$. Since $|v_6|\ge 3$, we must have $\braket{v_6}{v_1}\not = 0$. That is, $1\in \supp(v_6)$, and so $|v_6|\ge 4$. Using Lemmas~\ref{lem:k1>1 intervals}, \ref{lem:k1=2 sigma3} and Corollary~\ref{Cor:ZjDistinct}, $[v_1]$ will have all the high weight vertices of $[v_3]$, $[v_6]$, and $[v_4]$, and so, $|v_1|\ge 6$, a contradiction. 
 This proves the claim.
\end{proof}

\begin{prop}\label{lem:k1=2,2t+3}
When $(\sigma_0, \cdots, \sigma_{k_3})=(1,2,3,3,7)$, there exists $s\ge0$, such that $v_{s+5}=e_{3}+\cdots +e_{s+4}-e_{s+5}$, $v_5 = e_0 + e_4-e_5$ if $s>0$, and $|v_j|=2$ for $5<j<s+5$ and $j>s+5$. In this case, $\sigma =  (1,2,3,3,7,8^{[s]}, 8s+10^{[t]})$ ($s,t\geq 0$).
\end{prop}	
\begin{proof}
First suppose that $\sigma_5\neq 10$.
Since $k_3=4\in S'_5$, the set $S'_5$ is either $\{0,4\}$, $\{3,4\}$, or $\{0,2,3,4\}$ (Lemmas~\ref{Lem:X0Odd}~and~\ref{gappy3}). 
If $S'_5=\{3,4\}$, as $\sigma_5\neq 10$, we must have $1\in S_5$, a contradiction to Lemma~\ref{lem:NoSingle1}. If $S'_5=\{0,2,3,4\}$, then $\braket{v_1}{v_5}> 0$. Since $|v_5|\ge 5$ and $\delta([v_1], [v_5])\le 3$, we have $\epsilon_1=\epsilon_5$. Since $\braket{v_1}{v_5}\le2$, and that $|v_5|\ge 5$, we must have $\delta([v_5],[v_1])=3$. Since $\epsilon_1=\epsilon_5$, by Lemma~\ref{lem:delta}, $\braket{v_5}{x_0}=-\braket{v_1}{x_0}=\pm2$, which is not true. Therefore, $S'_5=\{ 0, 4\}$ and $v_5=e_0-e_4+e_5$ by Lemma~\ref{lem:0No2}.



We claim that if $S_j\ne\emptyset$ for some $j>5$, then $S_j=\{3,4\}$. Assume that $S_j\ne\emptyset$.
By Lemmas~\ref{Lem:X0Odd} and~\ref{gappy3}, $S_j'$ is one of $\emptyset$, $\{0,3\}$, $\{ 0,4\}$, $\{2, 3\}$, $\{ 3,4\}$, and $\{ 0,2, 3,4\}$. If $S_j'=\emptyset$, then $S_j=\{1\}$, contradicting Lemma~\ref{lem:NoSingle1}. Since $S'_5=\{ 0, 4\}$, $S_j'\ne\{0,3\}$ or $\{0,4\}$ by Lemma~\ref{lem:0No2}.
If $S'_j=\{ 2,3\}$, then $\braket{v_j}{x_0}=2$. Since $\delta([v_j], [v_1])\le 3, |v_j|\ge 4$, we have $\braket{v_j}{v_1}\ne0$. Since
$0\notin \supp(v_j)$, we must have $1\in \supp(v_j)$, and so $|v_j|\ge 5$. Using Lemma~\ref{intervalproduct}, we get $|\braket{v_j}{v_1}|>1$, contradicting the fact that $\braket{v_1}{v_j}=-1$. If $S'_{j} = \{0,2,3,4\}$, then we have $\braket{v_{j}}{x_0}=2$ and $|v_{j}|\ge 6$. Thus $x_1=x_{z_j}$ is contained in $[v_1]$. However, $|v_1|=5<6=|v_j|$, a contradiction. So $S'_j=\{3,4\}$. Using Lemma~\ref{lem:NoSingle1}, we conclude that $1\notin S_j$. So $S_j=\{3,4\}$. 

If $S_j=\emptyset$ for all $j>5$, it follows from Lemma~\ref{lem:AllNorm2} that $|v_j|=2$ whenever $j>5$. Now assume that $S_j\ne\emptyset$ for some $j>5$.
Let $s+5$ be the smallest such $j$. Then $S_{s+5}=\{3,4\}$ by the earlier discussion. We also know that $|v_i|=2$ for any $5<i<s+5$ by Lemma~\ref{lem:AllNorm2}. If $5\not \in \supp(v_{s+5})$, then $\braket{v_{s+5}}{v_5}\neq 0$ and $\braket{v_{s+5}}{v_3}\neq 0$, and so there will be a cycle $(v_{s+5}, v_3, v_2, v_5)$ of length bigger than $3$: see Figure~\ref{fig:case1 123 G(S)}. Thus $5 \in \supp(v_{s+5})$, and as a result $6, \cdots, s+4\in \supp(v_{s+5})$ by Lemma~\ref{gappy3}. Therefore, $\sigma_{s+5}=8s+10$.

Note that, $S_j=\emptyset$ when $j>s+5$. Otherwise, by the earlier discussion, $S_j=\{3,4\}$, and
we would have a heavy triple $(v_{j}, v_{s+5}, v_2)$. 
Given $j>s+5$, let $\ell=\min\supp(v_j)\ge5$. Note that
\[
v_5\sim v_6\sim \cdots \sim v_{s+4},
\]
$[v_5]$ and $[v_1]$ share their right endpoint, $\braket{v_i}{v_1}=0$ and $|v_i|=2$ when $5<i<s+5$, so $[v_{i}]\subset [v_1]$ when $5\le i<s+5$. If $\ell\le s+4$, then $\braket{v_j}{v_{\ell}}\ne0$.
Thus $[v_{j}]\cap [v_1]\neq \emptyset$. Note also that $\delta([v_{j}], [v_1])\le 2$ since $v_{j}$ is orthogonal to $x_0$. Since $|v_{j}|\ge 3$, we get that $|\braket{v_{j}}{v_1}|>0$, a contradiction. Thus we have proved that $\min\supp(v_j)\ge s+5$ when $j>s+5$. It follows from Lemma~\ref{lem:AllNorm2} that $|v_j|=2$ when $j>s+5$.

Finally suppose that $\sigma_5=10$. Assume that there exists $\ell>5$ such that 
$S_{\ell}\ne\emptyset$. By Lemmas~\ref{Lem:X0Odd}~and~\ref{gappy3}, $S'_\ell$ is one of $\emptyset$, $\{0,3\}$, $\{0,4\}$, $\{2, 3\}$, $\{ 3, 4\}$, and $\{0,2,3,4\}$. 
If $S'_\ell=\emptyset$, then $S_{\ell}=\{1\}$, contradicting Lemma~\ref{lem:NoSingle1}. By Lemma~\ref{lem:0No2}, $S'_\ell\ne\{0,3\}$ or $\{0,4\}$. Suppose $S'_\ell = \{2,3\}$. If $1\not \in \supp(v_\ell)$, there will be a claw $(v_2, v_1, v_\ell, v_3)$. If $1\in \supp(v_\ell)$, then $|\braket{v_\ell}{v_1}|=1$. By Lemma~\ref{lem:k1>1 intervals}, $[v_{\ell}]$ and $[v_1]$ are not consecutive. Since $\delta([v_\ell], [v_1])\le 3$ and $|v_\ell|\ge 5$, we get $|\braket{v_\ell}{v_1}|\ge2$, a contradiction. 
If $S'_\ell=\{3,4\}$, there will be a heavy triple $(v_5,v_\ell, v_2)$. If $S'_\ell=\{0,2,3,4\}$, then $|v_\ell|\ge 6$ and $\braket{v_{\ell}}{x_0}=2$, so $x_{z_{\ell}}=x_1$. Thus $|[v_1]|\ge |[v_\ell]|\ge6$, a contradiction. So we proved that $S_\ell=\emptyset$ whenever $\ell>5$. It follows from Lemma~\ref{lem:AllNorm2} that $|v_j|=2$ when $j>5$.
\end{proof}

\begin{figure}
	\begin{align*}
	\xymatrix@R1.5em@C1.5em{
		v_{3}^{(3)}\ar@{=}[r]^{+}\ar@{=}[dr]^{+} & x_0^{(4)} \ar@{=}[d]^{+} \\
		v_4^{(3)} \ar@{=}[r]^{+} \ar@{-}[dr]^{+} & v_1^{(5)} \ar@{-}[d]^{+} \\
		& v_{2}^{(3)} \\
	}
	&&
	\xymatrix@R1.5em@C1.5em{
		v_3^{(3)} \ar@{-}[d] \ar@{=}[r]^{+} \ar@{=}[dr]^{+} & x_0^{(4)} \ar@{=}[d]^{+} \\
		v_{4}^{(2)}\ar@{-}[d] & v_1^{(5)} \ar@{-}[d]^{+} \\
		\vdots & v_{2}^{(3)} \ar@{-}[d] \\
		v_{s+2}^{(2)}\ar@{-}[u] & v_{s+3}^{(s+2)}  
	}
	&&
	\xymatrix@R1.5em@C1.5em{
		& x_0^{(4)} \ar@{=}[d]^{+} \\
                     & v_1^{(5)} \ar@{-}[d]^{+} \\
		 & v_{2}^{(3)} \ar@{-}[d] \\
		& v_{3}^{(2)} \\
	}
	\end{align*}
	\caption{Pairing graphs when $(\sigma_0, \cdots, \sigma_{k_3})$ is $(1,2,3,4,5,9)$ (left), $(1,2,3,4^{[s]}, 4s+3, 4s+7)$ (center), or $(1,2,3,3,7)$ (right).}
	\label{fig:case1 123 G(S)}
\end{figure}
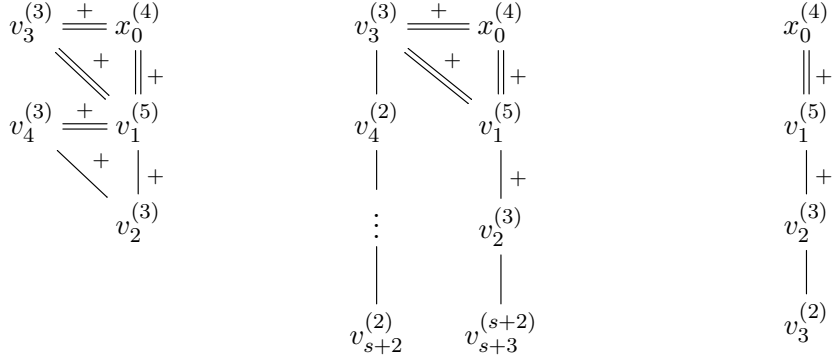

\begin{prop}\label{lem:k1>1,2t+3}
		If $(\sigma_0, \cdots \sigma_{k_3})=(1,2,3,4^{[s]}, 4s+3,4s+7)$, $s>0$, then $v_{s+5}=e_{s+3}+e_{s+4}-e_{s+5}$, and $\abs{v_j} = 2$ for $j>s+5$. In this case, $\sigma =  (1,2,3,4^{[s]}, 4s+3,4s+7, (8s+10)^{[t]})$ ($s>0,t \geq 0$). 
\end{prop}	

\begin{proof}
Suppose that $\ell>s+4$ is an index such that $S_{\ell}\ne\emptyset$. We will prove that $\ell=s+5$ and $v_{s+5}=e_{s+3}+e_{s+4}-e_{s+5}$. Our conclusion then follows from Lemma~\ref{lem:AllNorm2}.

\noindent{\bf Step 1}. $S'_{\ell}$ must be either $\emptyset$ or $\{ s+3, s+4\}$.

Using Lemma~\ref{gappy3} and Corollary~\ref{x0pairing}, $S'_\ell$ is either $\emptyset$, $\{ 0, s+4\}$, $\{ 2, s+4\}$, or $\{ s+3, s+4\}$. Suppose $S'_{\ell}=\{ 0, s+4\}$, by Lemma~\ref{lem:0No2}, $v_{\ell}=e_0+e_{s+4}-e_{s+5}$, $[v_{\ell}]$ and $[v_1]$ share their right endpoint. As $\braket{v_{\ell}}{v_3}\neq 0$, $[v_1]$ equals the union of $[v_3]$ and $[v_{\ell}]$ by Lemma~\ref{lem:k1=2 sigma3}, i.e. $|v_1|=4$, a contradiction. Suppose $S'_{\ell}=\{ 2, s+4\}$. If $1\not \in S_{\ell}$, as $s+3\not\in S_{\ell}$, $\braket{v_{ell}}{v_{s+3}}\ne0$, there will be a heavy triple $(v_{\ell}, v_{s+3}, v_2)$. If $1\in S_{\ell}$ (and consequently, $|v_{\ell}|\ge 4$), then there will be a claw $(v_1, x_0, v_{\ell}, v_2)$, unless $[v_{\ell}]\pitchfork [v_1]$. If $[v_{\ell}]\pitchfork [v_1]$, however, we get $\delta([v_{\ell}],[v_1])=2$, and so $|\braket{v_{\ell}}{v_1}|\ge 2$, a contradiction to the fact that $\braket{v_{\ell}}{v_1}=-1$.

\noindent{\bf Step 2}. If $S'_{\ell}=\emptyset$ or $\{ s+3, s+4\}$, then $S_{\ell}=S'_{\ell}$. In particular, $S_{s+5}=\{ s+3, s+4\}$.

Suppose that $S'_{\ell}=\emptyset$ or $\{ s+3, s+4\}$. Let $i=\min\supp(v_{\ell})$.  By Lemma~\ref{lem:NoSingle1}, $i\ne1$. That is, $\braket{v_\ell}{v_1}=0$. Also, note that $v_\ell$ is orthogonal to $x_0$, and so $\delta([v_\ell], [v_1])\le 2$. If $3\le i\le s+2$, since $v_3\sim v_4\sim  \cdots\sim v_{i}\sim v_{\ell}$, using Lemmas~\ref{lem:k1=2 sigma3}~and~\ref{lem:k1>1 intervals}, $x_{z_{\ell}}\in [v_1]$. Therefore, $\braket{v_1}{v_\ell}\not = 0$, a contradiction. So $i\ge s+3$ and hence $S_{\ell}=S'_{\ell}$. Clearly, $S_{s+5}=\{ s+3, s+4\}$ by Lemma~\ref{lem:j-1}.


\noindent{\bf Step 3}. If $S_{\ell}=\{ s+3, s+4\}$, then $\ell=s+5$. 

Assume that $\ell>s+5$ and $S_{\ell}=\{ s+3, s+4\}$, then we have a heavy triple $(v_{s+3},v_{s+5},v_{\ell})$.
\end{proof}

\subsection{{\boldmath $k_1=3$}}\label{k1=3}

In this subsection we focus on the changemaker C-type lattices with 
\begin{equation}\label{x0k1=3}
x_0=e_0+e_3+e_{k_2} - e_{k_3}.
\end{equation}  
Recall that the changemaker starts with $(1,2,2,3)$.

\begin{lemma}\label{lem:k1=3 intervals}
		The intervals $[v_3]$ and $[v_1]$ share their right endpoint and $\epsilon_3 = \epsilon_1$. Moreover, $[v_2]$ abuts the right endpoint of $[v_1]$ and $[v_3]$.
\end{lemma}
\begin{proof}
Since $|v_3|=3$ and $\braket{v_1}{v_3}=2$, from Lemma~\ref{intervalproduct}, it must be the case that $\epsilon_1=\epsilon_3$ and $\delta([v_1], [v_3])=1$. The first statement of the lemma is now immediate because $v_3$ is orthogonal to $x_0$. Since $\braket{v_2}{v_1}\ne0$ and $\braket{v_2}{x_0}=0$, $[v_2]$ abuts the right endpoint of $[v_1]$.
\end{proof}

\begin{cor}\label{cor:k1=3 intervals}
		Suppose that there exists a vector $v_j$ such that $j>3$, $j\neq k_3$, and $\braket{v_j}{v_1} = 2$. Then $j=4$, and that $v_4 = e_0 + e_3-e_4$. 
\end{cor}

\begin{proof}
Suppose that $j$ is such an index. Therefore, $0\in \supp^+(v_j)$ and $1\not \in \supp^+(v_j)$. (This, in particular, implies that $|v_j|\ge3$). We claim that $\braket{v_j}{x_0}\neq 0$. Otherwise, assume $\braket{v_j}{x_0}= 0$. Since $\braket{v_j}{v_1}=2$, $x_{z_j}\in[v_1]$. Using Lemma~\ref{lem:k1=3 intervals} and Corollary~\ref{Cor:ZjDistinct}, $[v_1]$ contains at least $3$ high weight vertices $x_1,x_{z_j},x_{z_3}$, and $\delta([v_j],[v_1])=2$. Since $|v_1|=5$, we have $|x_{z_j}|=3$, so by Lemma~\ref{intervalproduct} we have $|\braket{v_j}{v_1}|=1$, a contradiction.
This justifies the claim, and therefore, $\braket{v_j}{x_0}=2$. Since $|v_j|\ge 3$ and $\delta([v_1], [v_j])\le 3$, to get $\braket{v_j}{v_1}=2$, we must have $\epsilon_1=\epsilon_j$. Thus, $\delta([v_j], [v_1])=1$ and $|v_j|=3$. That is, $v_j=e_0+e_{j-1}-e_{j}$. We now argue that $j=4$. Suppose for contradiction that $j>4$. Thus $\braket{v_j}{v_3}=1$. Using Lemma~\ref{lem:k1=3 intervals}, we get that the interval $[v_1]$ equals the union of $[v_j]$ and $[v_3]$. Since $|v_j|=|v_3|=3$, we get that $|v_1|=4$, which is a contradiction. 
\end{proof}

\begin{lemma}\label{lem:v1Orth}
Let $v_j$ be a vector such that $j>3$, $j\neq k_3$. Then $\braket{v_j}{v_1}\in\{0,2\}$. As a result, $\min\supp(v_j)\ge2$ unless $j=4$ and $v_4 = e_0 + e_3-e_4$.
\end{lemma}
\begin{proof}
Assume that $\braket{v_j}{v_1}\notin\{0,2\}$, then $\supp(v_j)\cap\{0,1\}=\{1\}$ or $\{0,1\}$. By Lemma~\ref{gappy3}, $2\in \supp(v_j)$. 
If $0\in\supp(v_j)$, since $|\braket{v_j}{v_3}|\le1$ by Corollary~\ref{unbreakablepairing}, we have $3\in\supp(v_j)$. Thus $|v_j|\ge5$. Since $\braket{x_0}{v_j}=2$,  $x_{z_j}=x_1$. By Corollary~\ref{Cor:ZjDistinct} and Lemma~\ref{lem:k1=3 intervals}, $x_{z_j}\ne x_{z_3}$. So \[5=|v_1|\ge|x_{z_j}|+|x_{z_3}|-2\ge5+1,\]
a contradiction.

We have shown that $0\notin\supp(v_j)$. If $3\notin\supp(v_j)$, then $j>4$ and $|v_j|\ge4$. As $\braket{v_j}{v_3}=1$, using Corollary~\ref{Cor:ZjDistinct}, $[v_j]$ and $[v_3]$ are consecutive. By Lemma~\ref{lem:k1=3 intervals} and the fact that  $\braket{v_j}{v_2}=0$ we conclude that $[v_j]\subset[v_1]$. Since $\braket{v_j}{x_0}=0$, $[v_1]$ contains at least three high weight vertices: $x_1, x_{z_j}, x_{z_3}$. This is impossible as $|v_1|=5$ and $|v_j|\ge4$.

Now we have $\supp(v_j)\cap\{0,1,2,3\}=\{1,2,3\}$, so $\braket{v_j}{v_3}=0$. By  Lemma~\ref{lem:j-1}, $|v_j|\ge5$ unless $j=4$. By Lemma~\ref{lem:k1=3 intervals} and the fact that  $\braket{v_j}{v_1}\ne0$ we conclude that $[v_j]\subset[v_1]$. So $[v_1]$ contains at least two high weight vertices: $x_{z_j}, x_{z_3}$. It follows that $|v_j|\le4$. So $j=4$ and $|v_4| = e_1+e_2 + e_3-e_4$. Since $|v_4|=4$, $[v_1]$ contains exactly two high weight vertices, so $x_1$ must be $x_{z_4}$. So $\braket{v_4}{x_0}\ne0$, which is not possible. This shows that $\braket{v_j}{v_1}\in\{0,2\}$.

If $\min\supp(v_j)<2$, then $\braket{v_j}{v_1}\ne0$. We must have $\braket{v_j}{v_1}=2$, so $j=4$ and $v_4 = e_0 + e_3-e_4$ by Corollary~\ref{cor:k1=3 intervals}.
\end{proof}

\begin{lemma}\label{lem:NoSingle}
Let $v_j$ be a vector such that $j>4$, $j\neq k_3$. Then $\supp(v_j)\cap\{0,1,2,3\}\ne\{2\}$ or $\{3\}$.
\end{lemma}
\begin{proof}
Assume that $\supp(v_j)\cap\{0,1,2,3\}$ contains only one element which is $2$ or $3$.
Then $|v_j|\ge3$, $\braket{v_{j}}{v_3}\ne0$ while $\braket{v_{j}}{v_1}=0$. By Lemma~\ref{lem:k1=3 intervals}, $[v_{j}]$ abuts the left endpoint of $[v_3]$, so $[v_j]\subset[v_1]$. Since $|v_j|\ge3$ and $\delta([v_j],[v_1])\le3$, using Lemma~\ref{intervalproduct}, we get that $\braket{v_{j}}{v_1}\ne 0$ unless $|v_j|=\delta([v_j],[v_1])=3$. However, if $\delta([v_j],[v_1])=3$, $[v_1]$ is contained in the union of $[v_j],[v_3]$ and $\{x_0\}$. Since $|v_j|=|v_3|=3$, we have $|v_1|=4$, a contradiction.
\end{proof}

\begin{lemma}\label{lem:k1=3 sigma3}
		$\sigma_4 \in \{3,4,5\}$. Furthermore, if $\sigma_4 = 3$ then $[v_4]$ abuts the left endpoint of $[v_3]$. If $\sigma_4 = 4$ then $[v_4]$ and $[v_1]$ share their left endpoint.
\end{lemma}

\begin{proof}
If  $\min\supp(v_4)<2$,  using Lemma~\ref{lem:v1Orth},  $\sigma_4=4$. By Lemma~\ref{lem:j-1}, if $\min\supp(v_4)\ge2$, $v_4=e_2+e_3-e_4$ or $e_3-e_4$. So $\sigma_4=5$ or $3$.

When $\sigma_4=3$, $[v_4]$ abuts $[v_3]$ and $\braket{v_4}{v_1}=0$. By Lemma~\ref{lem:k1=3 intervals}, $[v_4]$ abuts the left endpoint of $[v_3]$.
When $\sigma_4 = 4$, $\braket{v_4}{v_1}=2=\braket{v_4}{x_0}$. So $\delta([v_4],[v_1])=1$ by Lemma~\ref{intervalproduct}. Thus  $[v_4]$ and $[v_1]$ share their left endpoint by Lemma~\ref{lem:k1=3 intervals}.
\end{proof}

\begin{prop}\label{prop1:k1=3}
If $\sigma_4=3$, the initial segment $(\sigma_0, \cdots, \sigma_{k_3})$ of $\sigma$ is $(1,2,2,3,3,7)$. 
\end{prop}
\begin{proof}
Suppose that $\sigma_4=3$ (see Lemma~\ref{lem:k1=3 sigma3}). This implies that $k_2=4$ (Lemma~\ref{Lem:X0Odd}). Using Equation~\eqref{x0k1=3}, we get that $\sigma_{k_3}=7$. If $k_3\not = 5$, using Lemma~\ref{Lem:X0Odd} and  Lemma~\ref{lem:v1Orth}, we must have $S_5\supset\{3,4\}$. By Lemma~\ref{lem:NoSingle}, we have $2\in S_5$, so $\braket{v_5}{x_0}=2$ and $v_5\sim v_2$. By Lemma~\ref{lem:k1=3 intervals}, $[v_1]$ is contained in the union of $x_0,[v_5],[v_2]$. So $|v_1|=|v_5|=4$, which is not possible.
\end{proof}

\begin{prop}\label{prop2:k1=3}
If $\sigma_4\neq 3$, the initial segment $(\sigma_0, \cdots, \sigma_{k_3})$ of $\sigma$ is $(1,2,2,3,4^{[s]},4s+5, 4s+9)$, $s\ge0$. 
\end{prop}

\begin{proof}
Suppose that $\sigma_4\neq 3$ (see Lemma~\ref{lem:k1=3 sigma3}). Furthermore, let $s\ge 0$ satisfy that $\sigma_i=4$ for any $4\le i < s+4$, and that $\sigma_{s+4}>4$. We have $k_2\ge s+4$ by Lemma~\ref{Lem:X0Odd}.
 Set $j=\min \supp(v_{s+4})<s+3$. Then $j\ge2$ by Lemma~\ref{lem:v1Orth}.  Also, $j \ne 3$ by Lemma~\ref{lem:NoSingle}.
If $4<j<s+3$, we will get a claw $(v_j, v_{j-1}, v_{s+4}, v_{j+1})$, and if $j=4$, the claw will be on $v_4, x_0, v_{s+4}, v_5$. This proves that $j=2$. By Lemma~\ref{lem:NoSingle}, $3\in \supp(v_{s+4})$. 

We will show that $\sigma_{s+4}=4s+5$. If $s=0$, $v_4=e_2+e_3-e_4$, and we are done. If $s>0$, since $2,3\in\supp(v_{s+4})$,
$|v_{s+4}|\ge 4$. Also, $v_{s+4}$ must be orthogonal to $v_4$, as otherwise, using Lemmas~\ref{lem:k1=3 sigma3}~and~\ref{lem:k1=3 intervals}, all the three intervals $[v_4], [v_{s+4}]$, and $[v_3]$ will be subsets of $[v_1]$, which implies that $|v_1|\ge 6$, a contradiction. That is, $4\in \supp(v_{s+4})$. Using Lemma~\ref{gappy3}, $v_{s+4}$ is just right and $\sigma_{s+4}=4s+5$. 

Using Lemma~\ref{Lem:X0Odd}, we see that $k_2=s+4$. By Equation~\eqref{x0k1=3}, we have $\sigma_{k_3}=4s+9$. Note that $k_2\in \supp(v_{k_2+1})$. Since the unbreakable vector $v_4$ has nonzero inner product with $x_0$, using Corollary~\ref{x0pairing}, we get that $k_3=k_2+1$.  
\end{proof}

\begin{prop}\label{1k1=3}
If $(\sigma_0, \cdots, \sigma_{k_3})=(1,2,2,3,3,7)$, then $n+1=k_3$ (i.e. $v_{k_3}$ is the last standard basis vector).
\end{prop}
\begin{proof}
We claim that the index $k_3+1$ (that is, $6$) does not exist. Using Lemmas~\ref{Lem:X0Odd}, \ref{lem:j-1}, \ref{gappy3}, and~\ref{lem:v1Orth}, $S'_6=\{4, 5\}$. Then $\braket{v_6}{v_4}\not = 0$, and also $v_6$ is orthogonal to $x_0$. Using Lemmas~\ref{lem:k1=3 sigma3}~and~\ref{lem:k1=3 intervals}, we must have $[v_6]\subset [v_1]$ which implies that $\braket{v_6}{v_1}\ne 0$ since $|v_6|\ge 3$. This contradicts Lemma~\ref{lem:v1Orth}.
\end{proof}

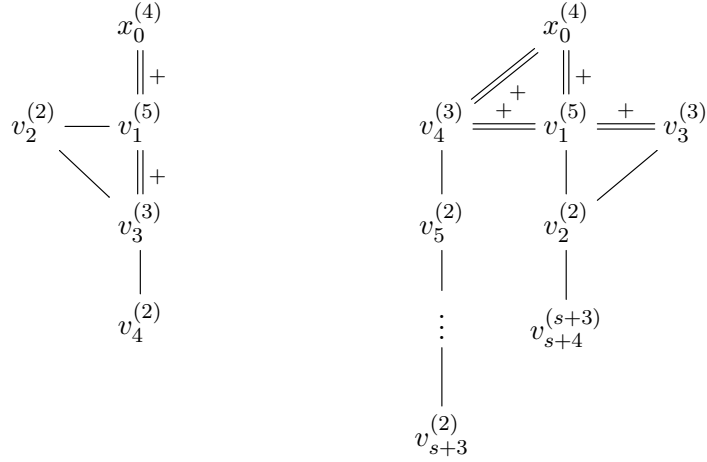
\begin{figure}
	\begin{align*}
	\xymatrix@R1.5em@C1.5em{
		& x_0^{(4)} \ar@{=}[d]^{+} \\
		v_2^{(2)} \ar@{-}[r] \ar@{-}[dr] & v_1^{(5)} \ar@{=}[d]^{+} \\
		& v_{3}^{(3)}\ar@{-}[d]\\
                & v_4^{(2)}\\
	}
	&&
	\xymatrix@R1.5em@C1.5em{
		& x_0^{(4)} \ar@{=}[d]^{+} \ar@{=}[dl]^{+} & \\
		v_4^{(3)}\ar@{=}[r]^{+}\ar@{-}[d] & v_1^{(5)} \ar@{=}[r]^{+} \ar@{-}[d] & v_3^{(3)} \ar@{-}[dl] \\
		v_5^{(2)}\ar@{-}[d] & v_{2}^{(2)} \ar@{-}[d] &  \\
		\vdots & v_{s+4}^{(s+3)} & \\
		v_{s+3}^{(2)}\ar@{-}[u] 
	}
	\end{align*}
	\caption{Pairing graphs when $(\sigma_0, \cdots, \sigma_{k_3})$ is $(1,2,2,3,3,7)$ (left) or $(1,2,2,3,4^{[s]},4s+5, 4s+9)$, $s>0$ (right).}
	\label{fig:case1 1223 G(S)}
\end{figure}

\begin{prop}\label{2k1=3}
		If $(\sigma_0, \cdots, \sigma_{k_3})=(1,2,2,3,4^{[s]},4s+5, 4s+9)$, $s\ge0$, then $v_{s+6}=e_{s+4}+e_{s+5}-e_{s+6}$ if it exists, and $|v_i|=2$ for $i>s+6$. In this case, $\sigma = (1,2,2,3,4^{[s]},4s+5, 4s+9, (8s+14)^{[t]})$, $t\ge0$.
\end{prop}
	
\begin{proof}
Suppose that $\ell>k_3=s+5$ is an index such that $S_{\ell}\ne\emptyset$. We will prove that $\ell=s+6$ and $S_{\ell}=\{s+4,s+5\}$. This, together with Lemma~\ref{lem:AllNorm2}, will imply our desired result.

By Lemmas~\ref{Lem:X0Odd},~\ref{lem:v1Orth}, and Corollary~\ref{x0pairing},  $S'_{\ell}$ is one of $\emptyset,\{3,4\},\{3,5\}$ and $\{4,5\}$ if $s=0$, and
one of $\emptyset$, $\{ 3, s+5\}$, and $\{s+4,s+5\}$ if $s>0$. Let $j=\min\supp(v_{\ell})$, then $j\ge2$ by Lemma~\ref{lem:v1Orth}.  Also, $j \ne 3$ by Lemma~\ref{lem:NoSingle}.

If $s=0$ and $S'_{\ell}=\{3,4\}$, we have $\braket{v_{\ell}}{x_0}=2$ and $|v_{\ell}|\ge4$, so $x_1\in[v_{\ell}]$. Using Lemma~\ref{intervalproduct}, we get $\braket{v_{\ell}}{v_1}\ne0$, a contradiction.

If $S'_{\ell}=\{ 3, s+5\}$, to avoid $\braket{v_{\ell}}{v_{s+4}}>1$, $2\notin\supp(v_{\ell})$. Thus, $j=3$, which is impossible.

Having proved $S'_{\ell}=\emptyset$ or $\{s+4, s+5\}$, we claim that $S_{\ell}=S_{\ell}'$. First, $j\ne2$ by Lemma~\ref{lem:NoSingle}. So our claim holds when $s=0$. When $s>0$,
if $4\le j<s+3$, we have a claw $(v_j,v_{j-1},v_{j+1},v_{\ell})$. If $j=s+3$, $\braket{v_{\ell}}{v_{s+3}}\neq 0$. By Lemma~\ref{lem:k1=3 sigma3}, $[v_4]$ and $[v_1]$ share their left endpoint. Since $|v_5|=\cdots=|v_{s+3}|=2$ and $v_4\sim v_5 \sim \cdots \sim v_{s+3}$, we have $[v_{\ell}]\subset[v_1]$ by Lemma~\ref{lem:k1=3 intervals}. Thus $\braket{v_{\ell}}{v_1}\ne 0$ by Lemma~\ref{intervalproduct}, a contradiction. So our claim is proved.

Now by Lemma~\ref{lem:j-1}, $s+5\in S_{s+6}$. So $S_{s+6}=\{s+4,s+5\}$ by the results in the previous two paragraphs. If there was $\ell>s+6$ satisfying $S_{\ell}=\{s+4,s+5\}$, we would have a heavy triple $(v_{s+4},v_{s+6},v_{\ell})$. Thus  $S_{\ell}=\emptyset$ whenever $\ell>s+6$.
\end{proof}

\section{Determining $p$ and $q$}\label{pq}
\begin{table}[htb]\centering
\ra{1.2}
\caption{$\mathcal{P}^{+}_{q>p}$, table of $P(p,q)$ that are realizable, $q>p$}\label{table:Types}
   \begin{tabular}{@{}lll@{}} \toprule

    Type & \begin{tabular}{l}$P(p,q)$\end{tabular} &\begin{tabular}{l}Range of parameters  \\
    ($p$ and $r$ are always odd, $p>1$)\end{tabular}\\
    \midrule
    \\
  {\bf 1A}   &\begin{tabular}{l}$P\left(p, \frac{1}{2}(p^2 + 3p + 4)\right)$\end{tabular} & \bigskip  \\ \\ 
  {\bf 1B} &\begin{tabular}{l}$P\left(p, \frac{1}{22}(p^2 + 3p + 4)\right)$\end{tabular} & \begin{tabular}{l}$p\equiv5$ or $3 \pmod{22}$ \\  $p\ne 3, 5$\end{tabular}\\\bigskip \\   
  {\bf 2} &\begin{tabular}{l}$P\left(p, \frac{1}{|4r+2|}(r^2p - 1)\right)$\end{tabular} &\begin{tabular}{l}$r \equiv -1\pmod4$\\
$p \equiv -2r+3\pmod{4r+2}$ \\ $r\ne -5, -1, 3$\end{tabular} \bigskip  \\ \\   
  {\bf 3A} &\begin{tabular}{l}
  $P\left(p, \frac{1}{2r}(p-1)(p-4)\right)$\end{tabular} &\begin{tabular}{l}$p\equiv 1\pmod{2r}$ \\ $p\ne 2r+1$\\ $r  \ge 5$ \end{tabular}\bigskip\\ \\  
  {\bf 3B} &\begin{tabular}{l}$P\left(p, \frac{1}{2r}(p-1)(p-4)\right)$\end{tabular} &\begin{tabular}{l}$p\equiv r+4\pmod{2r}$ \\ $p> r+4$ \\ $r \ge 1$ \end{tabular} \bigskip\\ \\  \bigskip
  {\bf 4} &\begin{tabular}{l}$P\left(p, \frac{1}{2r^2}\left((2r+1)^2p - 1\right)\right)$\end{tabular} &\begin{tabular}{l}$p \equiv -4r+1\pmod{2r^2}$ \\ $r\ne 1,-1$ \end{tabular}\\ \\ \bigskip
  {\bf 5} &\begin{tabular}{l}$P\left(p, \frac{1}{r^2 - 2r - 1}(r^2p - 1)\right)$\end{tabular} &\begin{tabular}{l}$r >  1$\\ $p \equiv -2r + 5\pmod{r^2 - 2r - 1}$ \end{tabular}\\ \\ \bigskip
  {\bf Sporadic} &\begin{tabular}{l}$P(11,19)$, $P(13, 34)$\end{tabular}&\\
  \bottomrule
\end{tabular}

\end{table}

In Sections~\ref{sec:Case2},~\ref{sec:k1k21},~and~\ref{sec:Case1}, we have classfied all the $(n+1)$--dimensional C-type lattices that are isomorphic to changemaker lattices. In the present section, we list all the corresponding prism manifolds $P(p,q)$. To do so, we start with the refined basis $S'=\{v_1, \cdots, v_{n+1}\}\setminus \{v_{k_3}\}\cup \{x_0\}$ as defined in~\eqref{Eq:S'}. The first step is changing the basis into the vertex basis $\{x_0, x_1, \cdots, x_{n}\}$. We then recover the $a_i$ from the norms of vertex basis elements. By using Equation~\eqref{eq:ContFrac}, we obtain $p$ and $q$. 


\begin{ex}
We present an example that clarifies how $(p,q)$ is computed in Proposition~\ref{example}. The changemaker is
\[
(1,1,2^{[s]}, 2s-1,2s+1), s=n-2\ge 2.
\]
Let $S'$ denote the modified standard basis for the changemaker lattice $L= (\sigma)^{\perp}$. It is straightforward to check that 
\[
\{x_0\}\cup\{-v_2, \cdots, -v_{s+1}, v_3+\cdots+v_{s+2}, v_1\}
\]
forms the vertex basis $S^*$. Also, the vertex norms are 
\[
\{3, 2^{[s-1]},s+1, 2\}.
\]
Using Lemma~\ref{greene9.5} together with Equation~\eqref{eq:ContFrac}, we have
\[
\displaystyle \frac{2q-p}{q-p}=[3, 2^{[s-1]},s+1, 2]=\frac{4s^2+3}{2s^2-s+2}.
\]
In particular, $p=2s-1$ and $q=2s^2+s+1$. We see that $q=\frac{1}{2}(p^2+3p+4)$, $p\ge3$.
\end{ex}

Similar computations give prism manifolds $P(p,q)$, with $q>p$, so that each falls into one of the families in Table~\ref{table:Types}. We denote the set of such prism manifolds $\mathcal{P}^+_{q>p}$. Here we divide the families so that each changemaker vector corresponds to a unique family. In some cases there are prism manifolds that correspond to more than one family in Table~\ref{table:Types}. For instance, it is straightforward to check that $P(5,22)$ belongs to both Families~5~and~1A. 
The detailed correspondence between the changemaker vectors and $P(p,q)$ can be found in Table~\ref{BigSummary}. Note that the positive integer $p$ is always odd.

\section{Prism manifolds realizable by surgery on knots in $S^3$}\label{realizable}
Table~\ref{table:Types} gives a list of all prism manifolds $P(p,q)$, with $q>p$, that can possibly be realized by surgery on knots in $S^3$. In~\cite[Table~2]{Prism2016}, a list of realizable prism manifolds $P(p,q)$ with $q>0$ is provided. It is straightforward to verify that the manifolds in Table~\ref{table:Types} coincide with those of \cite[Table~2]{Prism2016} with $q>p$. That is, Table~\ref{table:Types} is a complete list of prism manifolds $P(p,q)$, with $q>p$, arising from surgery on knots in $S^3$. 

\subsection{Prism manifolds corresponding to more than one changemaker vector}

As we pointed out in Section~\ref{pq}, some of the prism manifolds in Table~\ref{table:Types} correspond to distinct changemaker vectors. In this subsection, we address this by providing distinct knots corresponding to such prism manifolds. Our strategy is as follows: let $\sigma$ be a changemaker vector whose orthogonal complement is isomorphic to $C(p,q)$ for some $p$ and $q$. Let $\sigma$ correspond to a knot $K$ in $S^3$ on which surgery results in $P(p,q)$. Using Lemma~\ref{lem:AlexanderComputation}, we compute the Alexander polynomial $\Delta_K(T)$. Then we exhibit a P/SF knot $K_\sigma$ that admits a surgery to $P(p,q)$. By directly computing $\Delta_{K_\sigma}(T)$ we show that the two Alexander polynomials coincide. That is, $K_\sigma$ matches with $\sigma$. See \cite[Section~13.2]{Prism2016}. The parameters beneath the P/SF knots in Table~\ref{Overlap} are explained in~\cite{Prism2016}.    

\subsection{Proof of the main results}

\begin{proof}[Proof of Theorem~\ref{thm:lattice}]
If $C(p,q)$ is isomorphic to a changemaker lattice $L$, then it belongs to one of the families enumerated in Sections~\ref{sec:Case2},~\ref{sec:k1k21},~and~\ref{sec:Case1}. Following Section~\ref{pq}, we can find a pair $(p',q')$ such that $L$ is isomorphic to $C(p', q')$, and $P(p',q')\in \mathcal P^+_{q>p}$. Now, Proposition~\ref{pp} finishes the proof.
\end{proof}

\begin{proof}[Proof of Theorem~\ref{thm:Realization}]
Suppose $P(p,q)\cong S^3_{4q}(K)$, it follows from Theorem~\ref{changemakerlatticeembedding} and Theorem~\ref{thm:lattice} that $P(p,q)$ belongs to one of the six families in Table~\ref{table:Types} and $P(p,q)\cong S^3_{4q}(K_0)$ for some Berge--Kang knot $K_0$. To get the result about $\widehat{HFK}$, we note that $K$ and $K_0$ correspond to the same changemaker vector. Using Lemma~\ref{lem:AlexanderComputation}, we know that $\Delta_K=\Delta_{K_0}$, so $\widehat{HFK}(K)\cong \widehat{HFK}(K_0)$ by \cite[Theorem~1.2]{OSzLens}.
\end{proof}

\begin{table}\centering
\caption{Prism manifolds $P(p,q)$ corresponding to more than one changemaker}
\resizebox{\textwidth}{!} {%
\ra{1.3}

   \begin{tabular}{@{}lllll@{}} \toprule
	
	Prism manifold	& Type & Changemaker & P/SF knot  & Braid word \\ \midrule
		& 
		{\small \begin{tabular}{l}{\bf 4}\end{tabular}} & {\small $(1,2,3,3,7,8^{[s]})$} &\begin{tabular}{l} 
		   {\small {\bf KIST IV}, $s>0$}  \\
		   {\small $(2,-3,-1,0,s+2)$}\\ \\
                    {\small {\bf KIST I}, $s=0$}  \\
		   {\small $(1,3,4,-2,-3)$}
		          \end{tabular} & {\small $(\sigma_7\cdots \sigma_{1})^{8s+23}(\sigma_{13} \cdots \sigma_1)^{-8}$}  \\  {\small \begin{tabular}{l}$P(8s+13,16s+18)$\end{tabular}} &&&& \\ 
	&	 {\small \begin{tabular}{l}{\bf 3A}, $s>0$\\ \\ {\bf 3B}, $s=0$\end{tabular}} & {\small $(1,1,3,5,6,8^{[s]})$} & \begin{tabular}{l}
	
		    {\small {\bf OPT II}} \\
		    {\small $(2,3,0,1,s+1)$}
		    \end{tabular} & \begin{tabular}{l}
		  {\small $(\sigma_7 \cdots \sigma_1)^{8s+11}(\sigma_1 \cdots \sigma_7)^{-2}$}
		    
		    \end{tabular}  \\        
&&&& \\  
		         
&{\small \begin{tabular}{l}{\bf 5}, $s=3$\end{tabular}} & {\small $(1,1,1,3,4,6,10,10)$} & \begin{tabular}{l}{\small {\bf KIST IV}}\\{\small $(2,1,1,-3,2)$}\end{tabular} & {\small $(\sigma_1 \cdots \sigma_{25})^{10}\sigma_{3}\sigma_2\sigma_1$}\\
\bottomrule 

         & {\small \begin{tabular}{l}{\bf 5}\end{tabular}}& {\small $(1,1,1,2,3,6,6)$}
	&\begin{tabular}{l} {\small {\bf KIST IV}} \\ {\small $(2,1,1,-3,1)$} \end{tabular}& {\small $(29,3)$--cable of $T(5,2)$} \\ {\small \begin{tabular}{l}$P(5,22)$\end{tabular}}&&&& \\
& {\small \begin{tabular}{l}{\bf 1A}\end{tabular}} & {\small $(1,1,2,2,2,5,7)$} 
		&\begin{tabular}{l}{\small {\bf TKM II}}\\{\small $(1,2,-1,2,2)$}\end{tabular} &{\small $(\sigma_1 \cdots \sigma_{11})^{7}\sigma_1^2$}\\ \bottomrule

 & {\small \begin{tabular}{l}{\bf 3B}\end{tabular}}& {\small $(1,1,3,5,6,6,6)$}
	&\begin{tabular}{l}{\small {\bf OPT III}}\\{\small $(2,3,0,1,2)$}\end{tabular} &{\small $(\sigma_1 \cdots \sigma_{22})^{6}\sigma_2\sigma_3\sigma_4\sigma_1\sigma_2\sigma_3$} \\ {\small \begin{tabular}{l}$P(25,36)$\end{tabular}}&&&& \\
& {\small \begin{tabular}{l}{\bf 5}\end{tabular}} & {\small $(1,1,1,3,4,4,10)$} 
	&\begin{tabular}{l}{\small {\bf KIST IV}}\\{\small $(2,1,1,-1,3)$}\end{tabular} &{\small $(\sigma_1 \cdots \sigma_{13})^{10}\sigma_1\sigma_2\sigma_3$}\\ \bottomrule

 & {\small \begin{tabular}{l}{\bf 3A}\end{tabular}}& {\small $(1,1,2,5,7,10,12,12)$}
	&\begin{tabular}{l}{\small {\bf OPT II}}\\{\small $(2,5,0,1,3)$}\end{tabular} &{\small $(\sigma_1 \cdots \sigma_{40})^{12}\left(\sigma_{1}\cdots\sigma_{11}\right)^{-2}$}\\ {\small \begin{tabular}{l}$P(43,117)$\end{tabular}}&&&& \\
& {\small \begin{tabular}{l}{\bf 4}\end{tabular}} & {\small $(1,1,2,3,5,6,14,14)$} 
		& \begin{tabular}{l}{\small {\bf KIST IV}}\\{\small $(2,-3,1,-3,1)$}\end{tabular} &{\small $(\sigma_1 \cdots \sigma_{33})^{14}\left(\sigma_{7}\cdots\sigma_1\right)^{-1}$}\\ \bottomrule
\end{tabular}%
}
\label{Overlap}

\end{table}

\clearpage

\begin{table}\centering
\caption{C--type changemakers and the corresponding prism manifolds, Part I}
\resizebox{\textwidth}{!} {%
\ra{.9}

   \begin{tabular}{@{}lll@{}} \toprule

\multicolumn{1}{l}{{\small Prop.}} &
\multicolumn{1}{l}{{\small Changemaker vector}} & \multicolumn{1}{l}{{\small Vertex basis (with $x_0$ omitted) $\{ x_1,\cdots,x_n \}$}}  \\ 

\midrule

 \\
{\small \ref{example} }	&	{\small \begin{tabular}{l}$(1,1, 2^{[s]},2s-1,2s+1)$\\ $s\ge2$\end{tabular}}&
		  $\{-v_2, \cdots, -v_{s+1}, v_{[3,s+2]}, v_1\}$  \\ \\

\bottomrule \\

&		  	{\small \begin{tabular}{l}$(1,1,2^{[s]},2s+1,2s+3, 4s+4, 8s+10)$\\$s\ge 1$\end{tabular}}
	     & $\{-v_2,\cdots, -v_{s+1},-v_{s+5}, v_{s+4}, v_{s+2}, v_1\}$   \\ \\

{\small \ref{1k2>2}}	& {\small \begin{tabular}{l}$(1,1,2^{[s]},2s+1,2s+3, 4s+6, 8s+10)$\\$s\ge 1$\end{tabular}}
			 & $\{-v_2,\cdots, -v_{s+1},-v_{s+4}, v_{s+5}, v_{s+2}, v_1\}$   \\ \\

	& 
   	{\small \begin{tabular}{l}$(1,1,2, 3, 5, 8^{[s]},8s+6, (8s+14)^{[t]})$\\$s\ge 1$\end{tabular}}
	     & $\{-v_2, v_{s+5}, v_1, -v_3-v_1,-v_5,\cdots,-v_{s+4},-v_{s+6},\cdots,-v_{s+t+5}\}$   \\ \\

& 
   	{\small \begin{tabular}{l}$(1,1,2, 3, 5, 6, 14^{[t]})$\end{tabular}}
	     & $\{-v_2, v_1+v_5, -v_1, -v_3, -v_6, \cdots, -v_{t+5}\}$   \\ \\


\bottomrule
\\

	 &  {\small \begin{tabular}{l}$(1,1,2^{[s]}, 2s+3,2s+5, (4s+6)^{[t]})$\\$s, t\ge 1$\end{tabular}}
	     & $\{-v_2, \cdots, -v_{s+1}, v_{[1,s+1]}+v_{[s+4,s+t+3]}-v_{s+2}, -v_{s+t+3}, \cdots, -v_{s+4}, -v_1\}$   \\ \\

{\small \ref{2k2>2}} &  {\small \begin{tabular}{l}$(1,1,2^{[s]}, 2s+3,2s+5)$\\$s\ge 1$\end{tabular}}
	     & $\{-v_2, \cdots, -v_{s+1}, v_{[1,s+1]}-v_{s+2}, -v_1\}$   \\ \\

&	   	{\small \begin{tabular}{l}$(1,1,2^{[s]}, 2s+3, 2s+5, 4s+6, (4s+8)^{[t]})$\\$s,t\ge 1$\end{tabular}}
	     &  $\{-v_2,\dots,-v_{s+1},-v_{s+5},\dots,-v_{s+t+4},v_{[1,s+1]}+v_{[s+4,s+t+4]}-v_{s+2},-v_{s+4},-v_1\}$
		     \\ \\


\bottomrule \\

&{\small \begin{tabular}{l}$(1,1,3,5,6^{[t]})$\\$t\ge 1$\end{tabular}}
	     & $\{v_1+v_{[4,t+3]}-v_2, -v_{t+3},\dots, -v_4,-v_1\}$ 
		 \\ \\
	
{\small \ref{lem:k1=1,k2=2,v2tight}} &{\small \begin{tabular}{l}$(1,1,3,5)$\end{tabular}}
	     & $\{-v_2, v_1\}$ 
		 \\ \\
	   
	    &	{\small \begin{tabular}{l}$(1,1,3,5,6,8^{[t+1]})$\end{tabular}}
	     & $\{-v_5,\cdots,-v_{t+5}, v_1+v_{[4,t+5]}-v_2,-v_4, -v_1\}$ 
		    \\ \\ 

\bottomrule	    
 \\

	    &	{\small \begin{tabular}{l}$(1,1,1,3,4, 4^{[t]},4t+6,(4t+10)^{[s]})$\end{tabular}} 
	    	&  $\{-v_{t+5}, -v_1,-v_2,-v_4,\cdots,-v_{t+4}, -v_{t+6},\cdots,-v_{t+s+5}\}$  \\ \\ 

{\small \ref{prop:k1=1,k2=2,v2justright1}} 	& {\small \begin{tabular}{l}$(1,1,1,3,4,10)$ \end{tabular}} 
	    & $\{-v_5,v_4,v_2,v_1\}$ \\ \\

& {\small \begin{tabular}{l}$(1,1,1,3,6,10)$\end{tabular}} 
	    & $\{-v_4,v_5,v_2,v_1\}$
	    \\ \\
\bottomrule
\\
 {\small \ref{prop:k1=1,k2=2,v2justright2}}  &	{\small \begin{tabular}{l}$(1,1,1,2,3,6^{[t]})$\\
  $t\ge1$
   \end{tabular}}
	    	 & $\{-v_3,-v_1,-v_2,-v_5,\cdots,-v_{t+4} \}$ 
	   \\  \\
\bottomrule
	    \\
	{\small \ref{1k2=2}}  & {\small \begin{tabular}{l} $(1,2,3,4,5,9)$\end{tabular}} 
	   	 & {\small $\{-v_3,v_{[3,4]}-v_1,-v_4,v_2\}$} \\ \\

\bottomrule
\\

{\small \ref{lem:k1=2,2t+3}} & {\small \begin{tabular}{l}$(1,2,3,3,7,8^{[s]},(8s+10)^{[t]})$\\ $s\ge 1$\end{tabular}}
	   	 & $\{v_{[5,s+4]}-v_1,-v_{s+4},\cdots,-v_5,v_2,v_3,v_{s+5},\cdots,v_{s+t+4} \}$  \\ \\

& {\small \begin{tabular}{l}$(1,2,3,3,7,10^{[t]})$\end{tabular}}
	   	 & $\{-v_1,v_2,v_3,v_{5},\cdots,v_{t+4} \}$  \\ \\

\bottomrule
\\

{\small \ref{lem:k1>1,2t+3}} & {\small \begin{tabular}{l}$(1,2,3,4^{[s]},4s+3,4s+7,(8s+10)^{[t]})$\\$s\ge 1$\end{tabular}}
	   	 & $\{-v_3,\cdots,-v_{s+2},v_{[3,s+2]}-v_1,v_2,v_{s+3}, v_{s+5}, \cdots,v_{s+t+4} \}$  \\ \\

\bottomrule 
\\
	
{\small \ref{1k1=3}}	  &  \begin{tabular}{l}$(1,2,2,3,3,7)$
\end{tabular}  & $\{v_{[3,4]}-v_1, -v_4,-v_3,-v_2\}$ \\  \\
\bottomrule 
\\

 {\small \ref{2k1=3}}& 	{\small \begin{tabular}{l}$(1,2,2,3,4^{[s]},4s+5,4s+9,(8s+14)^{[t]})$\\$s\ge 1$\end{tabular}} 
& $\{-v_4,\cdots,-v_{s+3},v_{[3,s+3]}-v_1, -v_3,-v_2,-v_{s+4},-v_{s+6},\cdots, -v_{s+t+5}\}$  \\ \\
	
& 	{\small \begin{tabular}{l}$(1,2,2,3,5,9,14^{[t]})$\end{tabular}} 
& $\{v_{3}-v_1, -v_3,-v_2,-v_{4},-v_{6},\cdots, -v_{t+5}\}$  \\ \\   	
	 	   
\bottomrule
\end{tabular}%
}

\end{table}

\clearpage

\begin{table}
\addtocounter{table}{-1}
\caption{C--type changemakers and the corresponding prism manifolds, Part II}
{\small
\centering
\begin{adjustbox}{max width=\textwidth}
\ra{1.4}

   \begin{tabular}{@{}llll@{}} \toprule

\multicolumn{1}{l}{Prop.} &
\multicolumn{1}{l}{Vertex norms $\{a_1,\dots,a_n\}$} & \multicolumn{1}{l}{Prism manifold parameters} &\multicolumn{1}{l}{$\mathcal P^+_{q>p}$ type}  \\
\midrule

	\ref{example}	&	$\{3,2^{[s-1]},s+1,2\}$ &
		
		\begin{tabular}{l}$p=2s-1$\\ $q=2s^2+s+1$	 \end{tabular}
			 & {\bf 1A} \\  \bottomrule

		& $\{3,2^{[s-1]},5,3,s+2,2\}$ & 
	    \begin{tabular}{l}$p=22s+25$\\$q=22s^2+53s+32$ \end{tabular}
        & {\bf 1B} \\ \cline{2-4}

	\ref{1k2>2}   & $\{3,2^{[s-1]},4,4,s+2,2\}$ & 
	\begin{tabular}{l}$p=22s+27$	\\ $q=22s^2+57s+37$	 \end{tabular}
			 & {\bf 1B} \\  \cline{2-4}  

	 &    $\{3,s+3,2,3,3,2^{[s-1]}, 3, 2^{[t-1]}\}$ & 
	    \begin{tabular}{l}$r=2s+3$\\$p=2r^2(t+1)-4r+1$\\$q=(2r+1)^2(t+1)-8r-6$ \end{tabular} & {\bf 4}\\ \cline{2-4}

&    $\{3,3,2,3,4,2^{[t-1]}\}$ & 
	    \begin{tabular}{l}$r=3$\\$p=18t+7$ \\ $q=49t+19$\end{tabular} & {\bf 4}\\ 

	    
 \bottomrule 

	    & $\{3,2^{[s-1]}, 4, 2^{[t-1]},s+3,2\}$ &
	    	\begin{tabular}{l} $r=2t+1$\\$p=2r(s+1)+r+4$\\$q=\frac{1}{2}(2rs+3(r+1))(2s+3)$ \end{tabular}
	     &  {\bf 3B} \\ \cline{2-4}

 \ref{2k2>2}& $\{3,2^{[s-1]}, s+5, 2\}$ &
	    	\begin{tabular}{l} $r=1$\\$p=2s+7$\\$q=(s+3)(2s+3)$ \end{tabular}
	     &  {\bf 3B} \\ \cline{2-4}

 & $\{3,2^{[s-1]},3,2^{[t-1]},3,s+3,2\}$ &
	    
	    \begin{tabular}{l} $r=2t+3$\\$p=2r(s+2)+1$\\$q=(s+2)(2r(s+2)-3)$ \end{tabular}
	     &  {\bf 3A} \\ 

%

 \bottomrule 
	 
	    & $\{5, 2^{[t-1]},3,2\}$ &
	   \begin{tabular}{l}
$r=2t+1$\\
$p=6t+7$\\
$q=9t+9$
	   \end{tabular}  & 
	    {\bf 3B} \\ \cline{2-4}

 \ref{lem:k1=1,k2=2,v2tight} & $\{6,2\}$ &
	   \begin{tabular}{l}
$r=1$\\
$p=7$\\
$q=9$
	   \end{tabular}  & 
	    {\bf 3B} \\ \cline{2-4}
	  
	    &   $\{4, 2^{[t]},3,3,2\}$ &
	    
	 \begin{tabular}{l}  $r=2t+5$\\$p=8t+21$\\$q=16t+34$ 	 \end{tabular}
	    	 &   {\bf 3A}
	   \\ 

\bottomrule

%

\end{tabular}%
\end{adjustbox}
}

\end{table}	

\clearpage

\begin{table}
\addtocounter{table}{-1}
\caption{C--type changemakers and the corresponding prism manifolds, Part III}\label{BigSummary}

{\small
\centering
\begin{adjustbox}{max width=\textwidth}
\ra{1.4}

   \begin{tabular}{@{}llll@{}} \toprule

\multicolumn{1}{l}{Prop.} &
\multicolumn{1}{l}{Vertex norms $\{a_1,\dots,a_n\}$} & \multicolumn{1}{l}{Prism manifold parameters} &\multicolumn{1}{l}{$\mathcal P^+_{q>p}$ type}  \\ 
\midrule

&  	   $\{t+4,2,2,3,2^{[t]},3,2^{[s-1]}\}$&
	   
	   \begin{tabular}{l}	
	   $r=2t+5$\\$p=(r^2-2r-1)(s+1)-2r+5$\\$q=r^2(s+1)-2r+1$
	   \end{tabular}
     & {\bf 5} \\ \cline{2-4}

\ref{prop:k1=1,k2=2,v2justright1} &   	$\{6,3,2,2\}$ &
	   
	  \begin{tabular}{l} 	
	 $p=25$\\$q=32$
	  \end{tabular} 	 & {\bf 1B} \\  \cline{2-4}

&   	$\{5,4,2,2\}$ &
	   \begin{tabular}{l}
	   	$p=27$\\$q=37$              
       \end{tabular}
       & {\bf 1B} \\ 

\bottomrule
\ref{prop:k1=1,k2=2,v2justright2} &    $\{3,2,2,4,2^{[t-1]}\}$ &
	    \begin{tabular}{l}
	   $r=3$\\
           $p=2t+1$\\
           $q=9t+4$
\end{tabular}& {\bf 5}
\\
\bottomrule
\ref{1k2=2} &   $\{3,3,3,3\}$ &
	 
	\begin{tabular}{l}
	  $p=13$\\
          $q=34$
	\end{tabular}   	
	   	 & {\bf Sporadic}  \\ 
\bottomrule

	\ref{lem:k1=2,2t+3}   	&   	$\{4, 2^{[s-1]},3,3,2,s+3,2^{[t-1]}\}$ &
	 \begin{tabular}{l}
	 $r=-3-2s$\\$p=2r^2t-4r+1$\\$q=t(2r+1)^2-8r-6$
	 \end{tabular}
	 & {\bf 4} \\ \cline{2-4}

&   	$\{5,3,2,3,2^{[t-1]}\}$ &
	 \begin{tabular}{l}
	 $r=-3$\\
         $p=18t+13$\\
         $q=25t+18$
	 \end{tabular}
	 & {\bf 4} \\ 
\bottomrule

\ref{lem:k1>1,2t+3}	&   	$\{3,2^{[s-1]},4,3,s+2,3,2^{[t-1]}\}$ &
	  \begin{tabular}{l} 
	  $r=-5-4s$\\
          $p=(-4r-2)t-2r+3$\\
          $q=r^2t+\frac{1}{2}(r^2-2r+1)$
	  \end{tabular}
	  & {\bf 2} \\ 

\bottomrule

\ref{1k1=3}	 &  	$\{4,2,3,2\}$ &\begin{tabular}{l}
	   $p=11$\\
            $q=19$     
	 \end{tabular}
	 & {\bf Sporadic} \\
   \bottomrule
         
\ref{2k1=3}	&   $\{3, 2^{[s-1]},3,3,2,s+3,3,2^{[t-1]}\}$ &
	  \begin{tabular}{l}
	     $r=7+4s$\\
             $p=(4r+2)t+2r+5$\\
             $q=r^2t+\frac{1}{2}(r^2+2r-1)$
	  \end{tabular} 
	   	
	   	 &  {\bf 2} \\ \cline{2-4}

&   $\{4, 3, 2,3,3,2^{[t-1]}\}$ &
	  \begin{tabular}{l}
	     $r=7$\\
             $p=30t+19$\\
             $q=49t+31$
	  \end{tabular} 
	   	
	   	 &  {\bf 2} \\ 
		
\bottomrule
\end{tabular}%

\end{adjustbox}
}
\caption*{\small In this table, $v_{[a,b]}$ means $v_a+v_{a+1}+\cdots+v_{b}$ for $a<b$. All vertex bases are presented in the form $\{x_1,\cdots,x_n \}$. The parameters $s,t\ge0$ unless otherwise stated.
A superscript $^{[-1]}$ at an element in the sequence of vertex norms means that the sequence is truncated at this element and the element preceding it. For example, the sequence $\{3,2^{[s-1]},4,3,s+2,3,2^{[t-1]}\}$ becomes $\{3,2^{[s-1]},4,3,s+2\}$ when $t=0$.}
\end{table}	
\clearpage

\bibliographystyle{amsalpha2}
\bibliography{Reference}

\end{document}